\documentclass[11pt, reqno]{amsart} 
\usepackage{amssymb,amscd,amsfonts,amsbsy}
\usepackage{latexsym}
\usepackage{exscale}
\usepackage{amsmath,amsthm,amsfonts}
\usepackage{mathrsfs}
\usepackage{xcolor} 
\usepackage[colorlinks=true, linkcolor=blue, citecolor=red, urlcolor=red, backref=page]{hyperref} 
\usepackage{esint} 
\usepackage{amssymb} 
\usepackage{stmaryrd}
\usepackage{pifont}
\usepackage{cases}
\usepackage[utf8]{inputenc}

\calclayout

\allowdisplaybreaks

\newtheorem{theorem}[equation]{Theorem}
\newtheorem{proposition}[equation]{Proposition}
\newtheorem{lemma}[equation]{Lemma}

\theoremstyle{definition}
\newtheorem{definition}[equation]{Definition}
\newtheorem{example}[equation]{Example}
\theoremstyle{remark}
\newtheorem{remark}[equation]{Remark}

\numberwithin{equation}{section} 

\def\B{\mathcal{B}}
\def\C{\mathbb{C}}
\def\M{\mathbb{M}}
\def\N{\mathbb{N}}
\def\F{\mathcal{F}}
\def\H{\mathcal{H}}
\def\L{\mathcal{L}}
\def\Ln{\mathcal{L}^n} 
\def\R{\mathbb{R}}
\def\re{\mathbb{R}}
\def\Rn{\mathbb{R}^n}
\def\X{\mathbb{X}}
\def\Y{\mathbb{Y}}

\def\w{\omega}
\def\p{\mathfrak{p}}
\def\diam{\operatorname{diam}}
\def\div{\operatorname{div}}
\def\dist{\operatorname{dist}}
\def\loc{\operatorname{loc}}
\def\Lip{\operatorname{Lip}}
\def\supp{\operatorname{supp}}
\def\BMO{\operatorname{BMO}}
\def\VMO{\operatorname{VMO}}
\def\Re{\operatorname{Re}}
\def\inf{\operatornamewithlimits{inf\vphantom{p}}}

\newcommand{\norm}[1]{\left\lVert#1\right\rVert}
\newcommand\restr[2]{\ensuremath{\left.#1\right|_{#2}}}
\DeclareMathOperator*{\esssup}{ess\,sup}
\DeclareMathOperator*{\essinf}{ess\,inf}

\begin{document}

\title{Extrapolation on function and modular spaces, and applications}

\author{Mingming Cao}
\address{Mingming Cao\\
Instituto de Ciencias Matemáticas CSIC-UAM-UC3M-UCM\\
Con\-se\-jo Superior de Investigaciones Científicas\\
C/ Nicolás Cabrera, 13-15\\
E-28049 Ma\-drid, Spain} \email{mingming.cao@icmat.es}

\author{Juan José Marín}
\address{Juan José Marín\\
Instituto de Ciencias Matemáticas CSIC-UAM-UC3M-UCM\\
Con\-se\-jo Superior de Investigaciones Científicas\\
C/ Nicolás Cabrera, 13-15\\
E-28049 Ma\-drid, Spain} \email{juanjose.marin@icmat.es}

\author{José María Martell}
\address{José María Martell\\
Instituto de Ciencias Matemáticas CSIC-UAM-UC3M-UCM\\
Con\-se\-jo Superior de Investigaciones Científicas\\
C/ Nicolás Cabrera, 13-15\\
E-28049 Ma\-drid, Spain} \email{chema.martell@icmat.es}

\thanks{All authors acknowledge financial support from the Spanish Ministry of Science and Innovation, through the ``Severo Ochoa Programme for Centres of Excellence in R\&D'' (CEX2019-000904-S) and  from the Spanish National Research Council, through the ``Ayuda extraordinaria a Centros de Excelencia Severo Ochoa'' (20205CEX001). The first author also acknowledges financial support from Spanish Ministry of Science and Innovation, through the Juan de la Cierva-Formación 2018 (FJC2018-038526-I). The second and third authors also acknowledge that the research leading to these results has received funding from the European Research Council under the European Union's Seventh Framework Programme (FP7/2007-2013)/ ERC agreement no. 615112 HAPDEGMT.  The third author was partially supported by the Spanish Ministry of Science and Innovation,  MTM PID2019-107914GB-I00.}

\date{\today}




\begin{abstract}
We generalize the extrapolation theory of Rubio de Francia to the context of Banach function spaces and modular spaces. Our results are formulated in terms of some natural weighted estimates for the Hardy-Littlewood maximal function and are stated in measure spaces and for general Muckenhoupt bases. Finally, we give several applications in analysis and partial differential equations. 
\end{abstract}

\maketitle 
\tableofcontents

\section{Introduction}\label{sec:intro}

The celebrated extrapolation result by Rubio de Francia shows that if an operator $T$ satisfies 
\begin{equation}\label{eq:intro-RdF-1}
\begin{array}{c}
\|Tf\|_{L^{p_0}(\Rn,w_0)}\leq C\|f\|_{L^{p_0}(\Rn,w_0)}
\\[4pt]
\text{ for some }p_0\in[1,\infty)\text{ and every }w_0\in A_{p_0}, 
\end{array}
\end{equation}
then
\begin{equation}\label{eq:intro-RdF-2}
\begin{array}{c}
\|Tf\|_{L^{p}(\Rn,w)}\leq C\|f\|_{L^{p}(\Rn,w)}
\\[4pt]
\text{ for every }p\in(1,\infty)\text{ and every }w\in A_{p}.
\end{array}
\end{equation}
Here, $A_{p}$ denotes the class of Muckenhoupt weights in $\Rn$ with the underlying Lebesgue measure and $L^{p}(\Rn,w)$ denotes the associated weighted Lebesgue space. The theory of Muckenhoupt weights has been extensively studied and one of its most basic features is the fact that $w\in A_{p}$ if and only if the Hardy-Littlewood maximal function $M$ (cf. \eqref{eq:M-def}) is bounded on $L^p(\Rn,w)$ whenever $p\in(1,\infty)$; and for $p=1$, $w\in A_{1}$ if and only if $M$ (cf. \eqref{eq:M-def}) maps continuously $L^1(\Rn,w)$ into $L^{1,\infty}(\Rn,w)$. 
As a consequence, \eqref{eq:intro-RdF-2} may be recast as 
\begin{equation}\label{eq:intro-RdF-3}
\begin{array}{c}
\|Tf\|_{L^{p}(\Rn,w)}\leq C\|f\|_{L^{p}(\Rn,w)}
\text{ for every }p\in(1,\infty) 
\\[4pt]
\text{and every weight $w$ such that $M$ is bounded on }L^p(\Rn,w).
\end{array}
\end{equation}

In the last forty years, Rubio de Francia's result has been extended and complemented in several different ways, see \cite{CMP} and the references therein. In particular, one can see that \eqref{eq:intro-RdF-1} encapsulates information about the boundedness of $T$ on many different function spaces, see for instance \cite[Theorem~1.3]{CFMP} and \cite[Theorems~4.6 and 4.10]{CMP}. All these results show, in hindsight, that what is needed to extrapolate is a good behavior of the Hardy-Littlewood maximal function. 
Elaborating on this there are some two other equivalent formulations of \eqref{eq:intro-RdF-3} which will be relevant momentarily. For any weight $v$ (i.e., $v$ is a measurable function such that $0<v<\infty$ a.e.) let us introduce the ``dual'' operator $M_v' h(x)= M(h\,v)(x)/v(x)$. Fix then $p\in (1,\infty)$ and recall the well-known fact that  $w\in A_p$ if and only if $w^{1-p'}\in A_p'$.  Hence,  we can rewrite \eqref{eq:intro-RdF-3} as
\begin{equation}\label{eq:intro-RdF-4}
\begin{array}{c}
\|Tf\|_{L^{p}(\Rn,v)}\leq C\|f\|_{L^{p}(\Rn,v)}
\text{ for every weight $v$ such that}
\\[4pt]
\|Mh\|_{L^{p}(\Rn,v)} \leq C \|h\|_{L^{p}(\Rn,v)}
\text{ and } 
\|M'_{v}h \|_{L^{p'}(\Rn,v)} \leq C \|h\|_{L^{p'}(\Rn,v) }.
\end{array}
\end{equation}
or as
\begin{equation}\label{eq:intro-RdF-5}
\begin{array}{c}
\|(Tf)\,u\|_{L^p(\Rn,\Ln)}\leq C\|f\, u\|_{L^p(\Rn,\Ln)}
\text{ for every weight $u$}
\\[4pt]\text{such that }
\|(Mh)\,u\|_{L^p(\Rn,\Ln)} \leq C \|h u\|_{L^p(\Rn,\Ln)}
\\[4pt]\text{ and } 
\|(Mh)\,u^{-1} \|_{L^{p'}(\Rn,\Ln)} \leq C \|h u^{-1}\|_{L^{p'}(\Rn,\Ln) }.
\end{array}
\end{equation}
These say, as already announced, that extrapolation is intimately related with the boundedness of the Hardy-Littlewood maximal on the given space and its dual. To continue with our discussion we bring two extrapolation results which resemble the previous formulations. The first one, implicit in \cite[Theorem~4.10, Lemma~4.12]{CMP} in the case of $\mathbb{X}_v=\mathbb{X}(v)$ with $\mathbb{X}$ being rearrangement invariant Banach and recently extended in \cite[Theorem~8.2]{MMMMM20} to the general case where $\mathbb{X}_v$ is a weighted Banach function space with respect to the underlying measure $v\,d\mathcal{L}^n$ with $v\in L^1_{\rm loc}(\Rn,\mathcal{L}^n)$ and $v>0$ $\mathcal{L}^n$-a.e., establishes that \eqref{eq:intro-RdF-1} yields
\begin{equation}\label{eq:intro-RdF-6}
\begin{array}{c}
\norm{Tf}_{\X_v}\leq C\|f\|_{\X_v}
\text{ whenever}
\\[4pt]
\|Mh\|_{\X_v} \leq C \|h\|_{\X_v}
\text{ and } 
\|M'_{v}h\|_{\X'_v} \leq C \|h\|_{\X'_v},
\end{array}
\end{equation}
where $\X'_v$ is the associated space of $\X_v$. Hence, \eqref{eq:intro-RdF-6} is a generalization of \eqref{eq:intro-RdF-4} to the context of weighted Banach function spaces. Examples of spaces in this class to which \eqref{eq:intro-RdF-6} applies are Lebesgue spaces $L^p(v)$ for $1<p<\infty$ and $v\in A_p$; Lorentz spaces $L^{p,q}(v)$ with $1<p<\infty$, $1\le q\le\infty$, and  $v\in A_p$; Orlicz spaces such as $L^p(\log L)^\alpha(v)$ with $1<p<\infty$, $\alpha\in\R$, and  $v\in A_p$; or Lebesgue spaces with variable exponents $L^{p(\cdot)}$ where $v\equiv 1$ and $p(\cdot):\R^n\to (1,\infty)$ is a measurable function such that $M$ is bounded in  $L^{p(\cdot)}$ and $L^{p'(\cdot)}$, where $p'(\cdot)$ is the conjugate exponent of $p(\cdot)$, that is, $1/p(x)+1/p'(x)=1$ for every $x \in \Rn$. 

Our second example, which generalizes \eqref{eq:intro-RdF-4}, is \cite[Theorem 2.6]{CW}, where it is shown that \eqref{eq:intro-RdF-1} has implications not only for weighted Lebesgue spaces but also for weighted Lebesgue spaces with variable exponents.  Specifically, they prove that \eqref{eq:intro-RdF-1} implies 
\begin{equation}\label{eq:intro-Lpvar}
\begin{array}{c}
\|(Tf)\,u\|_{L^{p(\cdot)}(\Rn, \Ln)}\leq C\|f\,u\|_{L^{p(\cdot)}(\Rn, \Ln)}
\text{ for every weight $u$} 
\\[6pt]	
\text{such that } \|(Mh)\,u\|_{L^{p(\cdot)}(\Rn, \Ln)}\leq C\|h\,u\|_{L^{p(\cdot)}(\Rn, \Ln)}
\\[6pt]	
\text{and } \|(Mh)\,u^{-1}\|_{L^{p'(\cdot)}(\Rn, \Ln)}\leq C\|h\,u^{-1}\|_{L^{p'(\cdot)}(\Rn, \Ln)}.
\end{array}
\end{equation}
Here, $p(\cdot):\R^n\to (1,\infty)$ is a measurable function satisfying $1<\essinf_{\Rn}p(\cdot ) \le \esssup_{\Rn}p(\cdot)<\infty$, and $p'(\cdot)$ is the conjugate exponent of $p(\cdot)$, that is, $1/p(x)+1/p'(x)=1$ for every $x \in \Rn$.

The previous discussion reveals that extrapolating from \eqref{eq:intro-RdF-1} to a given weighted function space, requires assumptions on the weight guaranteeing that the Hardy-Littlewood maximal function $M$ is bounded on that weighted function space. However, it is important to emphasize that the equivalence of \eqref{eq:intro-RdF-4} and \eqref{eq:intro-RdF-5} stems from the fact that for every weight $w$ and every measurable function $f$ one has  
$ \|f w\|_{L^{p}(\Rn,\mathcal{L}^n)}= \|f\|_{L^{p}(\Rn,w^p)}$. However, that equality ceases to be valid if one replaces $L^p$ by more general spaces (say, $L^{p,q}$ with $p\neq q$, $L^p(\log L)^\alpha$ with $\alpha\neq 0$, or $L^{p(\cdot)}$ with $p$ being non-constant). That is, having the weight as part of the underlying measure (as $v$ in \eqref{eq:intro-RdF-4} or \eqref{eq:intro-RdF-6}) is in general different from having it as a pointwise multiplier (as $u$ in \eqref{eq:intro-RdF-5} or \eqref{eq:intro-Lpvar}) ---think for instance in $ \|f\|_{L^{p,\infty}(\Rn,w^p)}$ and $ \|f w\|_{L^{p,\infty}(\Rn,\mathcal{L}^n)}$.

Continuing this line of research, one of the goals of this paper is to reconsider the extrapolation on some new weighted Banach function spaces so that the two different approaches formulated above can be framed under the same result. More precisely, let $u$ and $v$ be weights (that is, measurable functions which are strictly positive and finite almost everywhere with respect to Lebesgue measure $\L^n$), and let $\X_v$ be a Banach function space with underlying measure $v\,d\L^n$ (cf. Definition~\ref{def:BFS}). Introduce the weighted norm related to $\X_v$ and $u$ as $f \longmapsto \|f\,u\|_{\X_v}$.
Our aim is to see that the boundedness of $M$ with respect to this new weighted space (and an appropriate associated space) allows us to extrapolate from \eqref{eq:intro-RdF-1} to obtain estimates in the just introduced weighted norm. More specifically, a particular case of Theorem~\ref{thm:BFSAp} assures that \eqref{eq:intro-RdF-1} implies 
\begin{equation}\label{eq:intro-BFS}
\begin{array}{c}
\|(Tf)\,u\|_{\X_v} \leq C \|f\,u\|_{\X_v} 
\ \text{for every Banach function space $\X_v$ and every weight}   
\\[6pt]	
\text{$u$ such that } \|(Mh)\,u\|_{\X_v} \leq C \|h\,u\|_{\X_v}
\text{ and } 
\|(M'_{v}h)\,u^{-1}\|_{\X'_v} \leq C \|h\,u^{-1}\|_{\X'_v}.
\end{array}
\end{equation}
In the particular case $u\equiv 1$ and $\X_v=L^{p}(\Rn,v\,d\L^n)$ with $p\in(1,\infty)$, one immediately sees that \eqref{eq:intro-BFS} becomes  \eqref{eq:intro-RdF-4}. More generally,  if $u\equiv 1$, \eqref{eq:intro-BFS} agrees with \eqref{eq:intro-RdF-6}. On the other hand, if $v\equiv 1$ and $\X_v=L^{p}(\Rn,\L^n)$ with $p\in(1,\infty)$, one can easily see that \eqref{eq:intro-BFS} turns to be \eqref{eq:intro-RdF-5}. Also, if $v\equiv 1$ and $\X_v=L^{p(\cdot)}(\Rn,\L^n)$ then \eqref{eq:intro-BFS} becomes \eqref{eq:intro-Lpvar}. We note that Theorem~\ref{thm:BFSAp} is written in much more general terms 
 using pairs of functions (in place of operators), the setting is that of general measure spaces, and the Hardy-Littlewood maximal function is associated with a general basis. We also prove other extrapolation results for Banach function spaces such as an $A_{\infty}$ extrapolation result (cf. Theorem~\ref{thm:BFSAi}) and a limited range extrapolation result (cf. Theorems~\ref{thm:lim-RIBFS} and \ref{thm:LRE}). These generalize previously known results (\cite[Theorem~2.1]{CMP04}, \cite[Theorem~4.9]{AM1}, and \cite[Theorem~3.31]{CMP}, respectively). 

Besides the new extrapolation results in the context of Banach function spaces we study the extrapolation on modular spaces associated with Young functions. More specifically, let $\Phi$ be a Young function so that $\Phi\in\Delta_2$ and let $\overline{\Phi}$ be its complementary Young function. A particular case of Theorem~\ref{thm:PhiAp} establishes that  \eqref{eq:intro-RdF-1} implies 
\begin{equation}\label{eq:intro-BFS:mod}
	\begin{array}{c}
		\displaystyle\int_{\re^n}\Phi(|Tf|\,u)\,v\,dx \leq C \int_{\re^n}\Phi(|f|\,u)\,v\,dx 
		\ \text{for any weights $u$ and $v$}  
		\\[12pt]	
		\text{such that } \displaystyle\int_{\re^n}\Phi((Mf)\,u)\,v\,dx \leq C \int_{\re^n}\Phi(|f|\,u)\,v\,dx 
		\\[12pt]	
		\text{and }\displaystyle \int_{\re^n}\overline{\Phi}((M'_v f)\,u^{-1})\,v\,dx \leq C \int_{\re^n}\overline{\Phi}(|f|\,u^{-1})\,v\,dx. 
		\end{array}
\end{equation}
In the particular case $u\equiv 1$ (resp. $v\equiv 1$) and $\Phi(t)=t^p$, $t\ge 0$ with $p\in(1,\infty)$, it is trivial to see that \eqref{eq:intro-BFS:mod} becomes  \eqref{eq:intro-RdF-4} (resp. \eqref{eq:intro-RdF-5}). Let us mention that Theorem~\ref{thm:PhiAp} is written in much more general terms 
using pairs of functions (in place of operators), the setting is that of general measure spaces, and the Hardy-Littlewood maximal function is associated with a general basis. Additionally, we establish an $A_{\infty}$ extrapolation result in modular spaces (cf. Theorem \ref{thm:PhiAi}).

Finally, we give several applications of the extrapolation theory developed here. We first establish the well-posedness results for some Dirichlet problem in the upper half-space whenever the boundary data belongs to a general weighted Banach function space or a weighted modular space. And, secondly, we obtain weighted inequalities on Banach function and modular spaces for the layer potential operators and the related commutators on Ahlfors regular domains satisfying a two-sided local John condition.  
Beyond that, an application of $A_p$ extrapolation is presented for the square function on non-homogeneous spaces. 
We also prove estimates for singular integral operators including the pseudo-differential operators, $L^{\Psi}$-Hörmander operators and singular integrals of Calderón-type.  We end up with a limited range extrapolation result for the Schrödinger operators with potentials and for the operators associated with the Kato conjecture.

\section{Preliminaries}\label{sec:pre}

\subsection{Muckenhoupt weights}
Throughout this paper, we make the assumption that $(\Sigma, \mu)$ is a non-atomic $\sigma$-finite measure space with $\mu(\Sigma)>0$. We recall that $\mu$ is said to be non-atomic if for every $\mu$-measurable set $F$ with $\mu(F)>0$, there exists $F'\subset F$ such that $\mu(F)>\mu(F')>0$.
The characteristic function of a $\mu$-measurable set $E$ will be denoted by $\mathbf{1}_E$.
Let $\B$ be a {\tt basis}, that is, a collection of $\mu$-measurable sets in $\Sigma$ such that $0<\mu(B)<\infty$ for every $B \in \B$. Let us introduce the Muckenhoupt weights associated with $\B$ and some of its most relevant properties in this section. Given a basis $\B$, the Hardy-Littlewood maximal operator $M_{\B}$  on $(\Sigma,\mu)$ associated with $\B$ is defined for each $\mu$-measurable function $f$ on $\Sigma$  by 
\begin{equation}\label{eq:M-def}
M_{\B}f(x) := \sup_{x\in B \in \B} \fint_B |f(y)|\,d\mu(y), \quad\text{if } x \in \Sigma_\B:=\bigcup_{B \in \B} B, 
\end{equation} 
and $M_{\B}f(x)=0$ otherwise.

A $\mu$-measurable function $w$ on $\Sigma$ is called a {\tt $\B$-weight} on $(\Sigma,\mu)$ if $0<w(x)<\infty$ for $\mu$-a.e.~$x \in \Sigma_\B$. 
Given $p \in (1, \infty)$ and a basis $\B$ on $(\Sigma, \mu)$, we define the {\tt Muckenhoupt class} $A_{p, \B}$ as the collection of all $\B$-weights $w$ on $(\Sigma,\mu)$ satisfying 
\begin{equation*}
[w]_{A_{p,\B}}:=\sup_{B \in \B} \left(\fint_{B} w\, d\mu \right) \left(\fint_{B}w^{1-p'}\, d\mu \right)^{p-1}<\infty,
\end{equation*} 
where $p'$ is the H\"older conjugate exponent of $p$, i.e., $\frac1p+\frac{1}{p'}=1$. As for the case $p=1$, we say that $w\in A_{1,\B}$ if  
\begin{equation*}
[w]_{A_{1,\B}} := \|(M_{\B}w)\,w^{-1}\,\mathbf{1}_{\Sigma_\B}\|_{L^{\infty}(\Sigma, \mu)} <\infty.
\end{equation*}
Finally, we define 
\begin{equation*}
A_{\infty,\B} :=\bigcup_{p\geq 1}A_{p,\B},
\end{equation*} 
and $[w]_{A_{\infty,\B}}=\inf_{p:w \in A_{p, \B}} [w]_{A_{p,\B}}$.

For every $p\in(1,\infty)$ and every $\B$-weight $w$ on $(\Sigma,\mu)$, we define the associated weighted Lebesgue space 
$L^p(\Sigma,w):=L^p(\Sigma,\,w d\mu)$ as the collection of $\mu$-measurable functions $f$ with $\int_\Sigma |f|^p\,w\,d\mu<\infty$.

Given $1<p\le\infty$, $1 \leq q < \infty$, and a basis $\B$ on $(\Sigma, \mu)$, we define the $A_{p,q, \B}$ class as the collection of all $\B$-weights $w$ on $(\Sigma,\mu)$ satisfying
\begin{equation}\label{Apq}
[w]_{A_{p, q, \B}} := \sup_{B \in \B} \left(\fint_{B}w^q\, d\mu \right) \left(\fint_{B}w^{-p'}\, d\mu\right)^{\frac{q}{p'}}<\infty.
\end{equation} 
The {\tt reverse Hölder classes} are defined in the following way: we say that $w\in RH_{s, \B}$ for $s\in(1,\infty)$ and a basis $\B$ on $(\Sigma, \mu)$, if 
\begin{equation*}
[w]_{RH_{s, \B}} :=\sup_{B \in \B} \left(\fint_B w^s\,d\mu\right)^{\frac1s} \left(\fint_B w\,d\mu\right)^{-1} < \infty. 
\end{equation*} 
Regarding the endpoint $s=\infty$, $w \in RH_{\infty, \B}$ means that
\begin{equation*}
[w]_{RH_{\infty, \B}} := \sup_{B \in \B} \|w\mathbf{1}_{B}\|_{L^{\infty}(\Sigma, \mu)} \left(\fint_{B}w\,d\mu\right)^{-1}<\infty.  
\end{equation*}

\begin{remark}
Note that, by definition, the operator $M_{\B}$ does not take into account the values of the functions in the set $\Sigma\setminus\Sigma_\B$. Also the $A_{p,\B}$ or $RH_{s,\B}$ classes do not depend on the values of the weights in the set $\Sigma\setminus\Sigma_\B$. This may create some technical issues in the arguments below and to avoid them, we will assume from now on that $\mu(\Sigma\setminus\Sigma_\B)=0$.  With this assumption in place, $w$ is a $\B$-weight on $(\Sigma,\mu)$ if $0<w(x)<\infty$ for $\mu$-a.e.~$x \in \Sigma$. In the general situation where $\mu(\Sigma\setminus\Sigma_\B)>0$ one can alternatively work with $\Sigma_\B$ in place of $\Sigma$ (or, what is the same, restrict all functions and weights to the set $\Sigma_\B$).
\end{remark}

The properties listed in the following result follow much as in the Euclidean case (see, for instance, \cite{GR}).

\begin{lemma}\label{lemma:AP-trivial}
Let  $(\Sigma, \mu)$ be a non-atomic $\sigma$-finite measure space and let  $\B$ be a basis. Then the following properties hold.
\begin{list}{\textup{(\theenumi)}}{\usecounter{enumi}\leftmargin=1cm \labelwidth=1cm \itemsep=0.2cm 
			\topsep=.2cm \renewcommand{\theenumi}{\alph{enumi}}}
\item $A_{1, \B} \subset A_{p, \B} \subset A_{q, \B}\subset A_{\infty, \B} $ and $RH_{\infty, \B} \subset RH_{q, \B} \subset RH_{p, \B}$
for every $1 <p \leq q <\infty$. 

\item For every $p\in(1,\infty)$, $w\in A_{p,\B}$ if and only if $w^{1-p'}\in A_{p', \B}$, and 
\begin{equation}\label{eq:Ap-dual}
[w^{1-p\,'}]_{A_{p', \B}}=[w]^{\frac{1}{p-1}}_{A_{p,\B}}. 
\end{equation}

\item For every $p \in (1,\infty)$ and $w_1,w_2\in A_{1, \B}$, 
\begin{equation}\label{eq:Ap-product}
w_1\,w_2^{1-p}\in A_{p,\B}  \quad\text{with}\quad  
[w_1\,w_2^{1-p}\,]_{A_{p,\B}} \leq [w_1]_{A_{1, \B}}\,[w_2]^{p-1}_{A_{1, \B}}. 
\end{equation}

\item If $w_1, w_2\in A_{p,\B}$ and $0\le \theta\le 1$, then $w_1^{\theta}\,w_2^{1-\theta}\in A_{p,\B}$ with
\[
[w_1^\theta\,w_2^{1-\theta}\,]_{A_{p,\B}} \leq [w_1]_{A_{p, \B}}^\theta\,[w_2]^{1-\theta}_{A_{p, \B}}. 
\]

\item For all $p\in[1,\infty)$ and $s\in(1,\infty)$,
\begin{align}\label{eq:JN}
w\in A_{p,\B} \cap RH_{s, \B} \quad 
\Longleftrightarrow\quad w^s\in A_{\tau, \B}, \quad \tau=s(p-1)+1.
\end{align}

\item $A_{p_1,q_1, \B} \subset A_{p_2,q_2, \B} $ for every $1<p_1\le p_2\le\infty$ and $1\leq q_2\le q_1 < \infty$.

\item For all $1<p\le\infty$ and $1\leq q < \infty$, 
\begin{align}\label{eq:Apq}
w \in A_{p,q, \B} &\,  \Longleftrightarrow\, w^{-1} \in A_{q',p', \B} 
\,  \Longleftrightarrow\,  w^q \in A_{1+\frac{q}{p'}, \B}\,  \Longleftrightarrow\, w^{-p'} \in A_{1+\frac{p'}{q}, \B}.
\end{align}

\item For all $1<p<\infty$ 
\begin{equation}\label{eq:App}
w^p \in A_{p,\B} \, \,  \Longleftrightarrow\, \, w^{-p'} \in A_{p',\B} \, \, \Longleftrightarrow\, \, w\in A_{p,p, \B} \, \, \Longleftrightarrow \, \, w^{-1} \in A_{p',p'}.
\end{equation}
\end{list}
\end{lemma}

\begin{definition}
A basis $\B$ is called a {\tt Muckenhoupt basis} if for every $p \in (1, \infty)$ and for every $w \in A_{p,\B}$, 
\begin{align}\label{eq:MBLp}
\|M_{\B}f\|_{L^p(\Sigma, w)} \le C \|f\|_{L^p(\Sigma, w)}. 
\end{align}
\end{definition}

Given a basis $\B$ on $(\Sigma, \mu)$ and a $\B$-weight $w$, we set $M'_{\B, w}f(x):=M_{\B}(f\,w)/w(x)$ if $x \in \Sigma_\B$, and $M'_{\B, w}f(x)=0$ otherwise. Note that if $\B$ is a Muckenhoupt basis on $(\Sigma, \mu)$, by definition and property (b) above, it follows that $M_{\B}$ is bounded on $L^p(w)$ and $M'_{\B, w}$ is bounded on $L^{p'}(w)$ for every $w \in A_{p, \B}$, $1<p<\infty$.

\begin{definition}
A basis $\B$ is \texttt{open for the Muckenhoupt classes}, if for every $p \in (1, \infty)$, 
\begin{equation*}
w \in A_{p, \B} \, \Longrightarrow \, w \in A_{p-\varepsilon, \B} \text{ for some } \varepsilon \in (0, p-1). 
\end{equation*}
\end{definition}

We present some examples of Muckenhoupt bases. In $(\R^n, \L^n)$, the set of all cubes $\mathcal{Q}$,  the set of all dyadic cubes $\mathcal{D}$, and the set of all rectangles $\mathcal{R}$,  whose sides are parallel to the coordinate axes, are all Muckenhoupt bases, see \cite{GR}. Another interesting example is the basis $\mathcal{Z}$ of Zygmund rectangles in $(\re^3, \L^3)$ whose sides are parallel to the coordinate axes and have lengths $s$, $t$ and $st$ with $s,t>0$, see \cite{FP}. 
Moving on, let $(\Sigma,d,\mu)$ be a doubling measure metric space, that is, $(\Sigma,d,\mu)$ is a metric space endowed with a doubling measure. The latter condition means that 
\begin{equation}
\mu(B(x,2r)) \leq C_{\mu} \ \mu(B(x,r)), \quad x \in \Sigma \text{ and } r>0.
\end{equation}
If we set $\B$ to be the collection of all metric balls in $(\Sigma, d)$, then $\B$ is a Muckenhoupt basis since \cite[Proposition~7.13]{HK} gives that $M_{\B}$ is bounded in $L^p(\Sigma, w)$, $1<p<\infty$, if and only if $w \in A_{p, \B}$, in which case 
\begin{align}\label{eq:Msharp}
\|M_{\B}f\|_{L^p(\Sigma, w)} \le C [w]_{A_{p, \B}}^{\frac{1}{p-1}} \|f\|_{L^p(\Sigma, w)}.  
\end{align}

The five examples of Muckenhoupt bases given above all have the openness property for the Muckenhoupt classes; in each case it is a consequence of a reverse Hölder  inequality. Indeed, there holds $A_{\infty, \B} \subset \bigcup_{s>1} RH_{s, \B}$. In \cite[p.~29]{CMP} one can find an example of a Muckenhoupt basis $\B$ (consisting on a single element) and a weight $w\in A_{2,\B}$ with $w\not\in A_{p, B}$ for any $p<2$, that is, the basis $\B$ is not open for the Muckenhoupt classes.

\begin{lemma}\label{lem:ApRH}
Let  $(\Sigma, \mu)$ be a non-atomic $\sigma$-finite measure space and let  $\B$ be a basis.  Assume that $A_{\infty, \B} \subset \bigcup_{s>1} RH_{s, \B}$ and fix $p,q\in (1, \infty)$. Then the following hold.
\begin{list}{\textup{(\theenumi)}}{\usecounter{enumi}\leftmargin=1cm \labelwidth=1cm \itemsep=0.2cm 
			\topsep=.2cm \renewcommand{\theenumi}{\alph{enumi}}}
\item\label{list:open-1} If $w \in A_{p, \B}$, then $w^r \in A_{p, \B}$ for some $r>1$ and $w \in A_{p/s, \B}$ for some $1<s<p$. In particular, $\B$ is open for the Muckenhoupt classes. 

\item\label{list:open-2} If $u^p\, v \in A_{p, \B}$ and $v \in A_{\infty, \B}$ then $u^{p/r}\, v \in A_{p/r, \B}$ for every $r\in [1,r_0]$ and some $1<r_0<p$

\item\label{list:open-3} If $u^q\, v \in A_{q, \B}$ and $v \in A_{\infty, \B}$ then $u^{q\,s}\, v \in A_{q\,s, \B}$  for every $s\in [1,s_0]$ and some $s_0>1$.

\item\label{list:open-4} If  $u^q\, v \in A_{p, \B}$,  $v \in A_{p, \B}$, and $1<p\le q<\infty$, then $u^p\, v \in A_{p, \B}$ and $u^q\, v \in A_{q, \B}$.

\end{list}
\end{lemma}

\begin{proof}
We begin by showing \eqref{list:open-1}. Let $w \in A_{p, \B}$. Then $w^{1-p'} \in A_{p', \B}$. By the assumption $A_{\infty, \B} \subset \bigcup_{s>1} RH_{s, \B}$, there exist $s_1, s_2>1$ such that $w \in RH_{s_1, \B}$ and $w^{1-p'} \in RH_{s_2, \B}$. Pick $r:=\min\{s_1, s_2\}>1$. Then for every $B \in \B$, Jensen's inequality gives 
\begin{align*}
&\left(\fint_{B} w^r d\mu\right) \left(\fint_{B} w^{r(1-p')} d\mu\right)^{p-1} 
\le \left(\fint_{B} w^{s_1} d\mu\right)^{\frac{r}{s_1}} \left(\fint_{B} w^{s_2(1-p')} d\mu\right)^{\frac{r(p-1)}{s_2}} 
\nonumber \\
&\qquad\qquad\le [w]_{RH_{s_1, \B}}^r  [w^{1-p'}]_{RH_{s_2, \B}}^{r(p-1)} \left(\fint_{B} w\, d\mu\right)^{r} \left(\fint_{B} w^{1-p'}d\mu\right)^{r(p-1)} 
\nonumber \\
&\qquad\qquad\le [w]_{RH_{s_1, \B}}^r  [w^{1-p'}]_{RH_{s_2, \B}}^{r(p-1)} [w]_{A_{p, \B}}^r, 
\end{align*}
which yields that $w^r \in A_{p, \B}$. Moving on we let $p_0:=(1+(p'-1)r)'\in (1, p)$ so that $r(1-p')=1-p'_0$.  Note that for every $B \in \B$, by Jensen's inequality
\begin{multline*}
\left(\fint_{B} w\, d\mu\right) \left(\fint_{B} w^{1-p'_0} d\mu\right)^{p_0-1} 
=\left(\fint_{B} w\, d\mu\right) \left(\fint_{B} w^{r(1-p')} d\mu\right)^{\frac{p-1}{r}} 
\\
\le \left(\fint_{B} w^r d\mu\right)^{\frac1r} \left(\fint_{B} w^{r(1-p')} d\mu\right)^{\frac{p-1}{r}} 
\le [w^r]_{A_{p, \B}}^{\frac1r}. 
\end{multline*}
Hence $w \in A_{p_0, \B}$ with $p_0\in(1,p)$ and this eventually shows that $\B$ is open for the Muckenhoupt classes. 

Next, we prove \eqref{list:open-2}. Fix $u^p\, v \in A_{p, \B}$ and $v \in A_{\infty, \B}$. Pick $q_0\in(1,\infty)$ so that $v\in A_{q_0, \B}$. By Lemma~\ref{lemma:AP-trivial}, $(u\,v^{1/p})^{-p'}=(u^p\,v)^{1-p'}\in A_{p', \B}$. By the assumption $A_{\infty, \B} \subset \bigcup_{s>1} RH_{s, \B}$ it follows that
$(u\,v^{1/p})^{-p'}\in RH_{t_0, \B}$ for some $t_0>1$. Fix $1<t<t_0$ and note that
\[
\lim_{r\to 1^+} 1+\frac{t'}{r'}\,\frac{(p/r)'}{(p/r)}=1
\quad\text{and}\quad 
\lim_{r\to 1^+} t\,(p/r)'= t\,p'.
\]
Therefore, we can find $r_0\in (1,p)$ such that for every $ 1<r\le r_0$ one has 
$q':=1+(t'/r')(p/r)'/(p/r)<q_0'$ and $t\,(p/r)'<t_0\, p'$. After all these considerations we observe that Hölder's inequality with $r>1$ yields
\[
\left(\fint_{B} u^{p/r}\,v\,d\mu\right)^{\frac{r}{p}}
=
\left(\fint_{B} u^{p/r}\,v^{1/r}\,v^{1/r'}\, d\mu\right)^{\frac{r}{p}}
\le
\left(\fint_{B} u^{p}\,v\, d\mu\right)^{\frac{1}{p}}
\left(\fint_{B} v\,d\mu\right)^{\frac{r}{p\,r'}}.
\]
On the other hand, Hölder's inequality with $t>1$, the facts that $q'<q_0'$ and $t\,(p/r)'<t_0\, p'$, Jensen's inequality, and that $(u\,v^{1/p})^{-p'}\in RH_{t_0, \B}$ lead to 
\begin{align*}
&\left(\fint_{B} (u^{p/r}\,v)^{1-(p/r)'}\,d\mu\right)^{\frac{1}{(p/r)'}}
=
\left(\fint_{B} (u\,v^{1/p})^{-(p/r)'}\,v^{(1-q')/t'}\,d\mu\right)^{\frac{1}{(p/r)'}}
\\
&\qquad\le
\left(\fint_{B} (u\,v^{1/p})^{-t\,(p/r)'}\,d\mu\right)^{\frac{1}{t\,(p/r)'}}
\left(\fint_{B} v^{1-q'}\,d\mu\right)^{\frac{1}{t'(p/r)'}}
\\
&\qquad\le
\left(\fint_{B} (u\,v^{1/p})^{-t_0\,p'}\,d\mu\right)^{\frac{1}{t_0\,p'}}
\left(\fint_{B} v^{1-q_0'}\,d\mu\right)^{\frac{r\,(q_0-1)}{p\,r'}}
\\
&\qquad\le
[(u\,v^{1/p})^{-p'}]_{RH_{t_0, \B}}^{\frac{1}{p'}}
\left(\fint_{B} (u^p\,v)^{1-p'}\,d\mu\right)^{\frac{1}{p'}}
\left(\fint_{B} v^{1-q_0'}\,d\mu\right)^{\frac{r\,(q_0-1)}{p\,r'}}.
\end{align*}
Collecting the obtained estimates we conclude that
\begin{align*}
&\left(\fint_{B} u^{p/r}\,v\,d\mu\right)^{\frac{r}{p}} \left(\fint_{B} (u^{p/r}\,v)^{1-(p/r)'}\,d\mu\right)^{\frac{1}{(p/r)'}}
\\
&
\qquad
\le
[(u\,v^{1/p})^{-p'}]_{RH_{t_0, \B}}^{\frac{1}{p'}}
\left(\fint_{B} u^{p}\,v\, d\mu\right)^{\frac{1}{p}}
\left(\fint_{B} (u^p\,v)^{1-p'}\,d\mu\right)^{\frac{1}{p'}}
\\
&
\hskip5cm\times
\left(\fint_{B} v\,d\mu\right)^{\frac{r}{p\,r'}}
\left(\fint_{B} v^{1-q_0'}\,d\mu\right)^{\frac{r\,(q_0-1)}{p\,r'}}
\\
&
\qquad
\le
[(u\,v^{1/p})^{-p'}]_{RH_{t_0, \B}}^{\frac{1}{p'}} [u^p\,v]_{A_{p, \B}}^{\frac{1}{p}}
[v]_{A_{q_0, \B}}^{\frac{r}{p\,r'}}.
\end{align*}
Consequently, $ u^{p/r}\,v\in A_{p/r,\B}$ as desired with 
\[
[u^{p/r}\,v]_{A_{p/r,\B}}\le [(u\,v^{1/p})^{-p'}]_{RH_{t_0, \B}}^{\frac{p}{p'\,r}} [u^p\,v]_{A_{p, \B}}^{\frac{1}{r}}
[v]_{A_{q_0, \B}}^{\frac{1}{r'}}.
\]

To proceed we next consider \eqref{list:open-3}.  Fix $u^q\, v \in A_{q, \B}$ and $v \in A_{\infty, \B}$. Pick $q_0 \in (1, \infty)$ so that $v\in A_{q_0, \B}$. By the assumption $A_{\infty, \B} \subset \bigcup_{s>1} RH_{s, \B}$ it follows that $u^q\,v\in RH_{t_0, \B}$ for some $t_0>1$. Fix $1<t<t_0$, pick $s_0$ so that  $1<s_0<\min\{t_0/t,1+(q_0'-1)/t'\}$, and let $1<s<s_0$. Hölder's inequality with exponent $t>1$, together with our choices guarantee that $ts<r_0$ and $t'(s-1)<q_0'-1$, we arrive at
\begin{align*}
\left(\fint_{B} u^{q\,s}\,v\,d\mu\right)^{\frac{1}{q\,s}}
&=
\left(\fint_{B} u^{q\,s}\,v^s\,v^{1-s}\,d\mu\right)^{\frac{1}{q\,s}}
\\
&\le
\left(\fint_{B} (u^q\,v)^{t\,s}\,d\mu\right)^{\frac{1}{t\,q\,s}}
\left(\fint_{B} v^{t'\,(1-s)}\,d\mu\right)^{\frac{1}{t'\,q\,s}}
\\
&\le
\left(\fint_{B} (u^q\,v)^{t_0}\,d\mu\right)^{\frac{1}{t_0\,q}}
\left(\fint_{B} v^{1-q_0'}\,d\mu\right)^{\frac{(q_0-1)\,(s-1)}{q\,s}}
\\
&\le 
[u^q\,v]_{RH_{t_0,\B}}^{\frac1q}\left(\fint_{B} u^q\,v\,d\mu\right)^{\frac{1}{q}}
\left(\fint_{B} v^{1-q_0'}\,d\mu\right)^{\frac{q_0-1}{q\,s'}}.
\end{align*}
On the other hand, Hölder's inequality with $q'/(q\,s)'>1$
\begin{align*}
\left(\fint_{B} (u^{q\,s}\,v)^{1-(q\,s)'}\,d\mu\right)^{\frac{1}{(q\,s)'}}
&=
\left(\fint_{B} (u^{q}\,v)^{-\frac{(q\,s)'}{q}}\,v^{\frac{(q\,s)'}{q\,s'}}\,d\mu\right)^{\frac{1}{(q\,s)'}}
\\
&\le
\left(\fint_{B} (u^{q}\,v)^{-\frac{q'}{q}}\,d\mu\right)^{\frac{1}{q'}}
\left(\fint_{B}\,v^{\frac{(q\,s)'}{q\,s'}\,(\frac{q'}{(q\,s)'})'}\,d\mu\right)^{\frac{1}{(q\,s)'(q'/(q\,s)')'}}
\\
&
=
\left(\fint_{B} (u^{q}\,v)^{1-q'}\,d\mu\right)^{\frac{1}{q'}}
\left(\fint_{B}\,v\,d\mu\right)^{\frac{1}{q\,s'}}
.
\end{align*}
Collecting the obtained estimates
\begin{align*}
&
\left(\fint_{B} u^{q\,s}\,v\,d\mu\right)^{\frac{1}{q\,s}}
\left(\fint_{B} (u^{q\,s}\,v)^{1-(q\,s)'}\,d\mu\right)^{\frac{1}{(q\,s)'}}
\\
&\qquad\le
[u^q\,v]_{RH_{t_0,\B}}^{\frac1q}
\left(\fint_{B} u^q\,v\,d\mu\right)^{\frac{1}{q}}
\left(\fint_{B} (u^{q}\,v)^{1-q'}\,d\mu\right)^{\frac{1}{q'}}
\left(\fint_{B}\,v\,d\mu\right)^{\frac{1}{q\,s'}}
\left(\fint_{B} v^{1-q_0'}\,d\mu\right)^{\frac{q_0-1}{q\,s'}} 
\\
&\qquad\le
[u^q\,v]_{RH_{t_0,\B}}^{\frac1q}
[u^q\,v]_{A_{q,\B}}^{\frac1q}
[v]_{RH_{t_0,\B}}^{\frac1{q\,s'}}.
\end{align*}
Thus, $ u^{q\,s}\,v\in A_{q\,s,\B}$ as desired with 
\[
[u^{q\,s}\,v]_{A_{q\,s,\B}}
\le 
[u^q\,v]_{RH_{t_0,\B}}^{s}
[u^q\,v]_{A_{q,\B}}^{s}
[v]_{RH_{t_0,\B}}^{\frac{s}{s'}}.
\]

It remains to prove \eqref{list:open-4}. Let $u^q\, v \in A_{p, \B}$ with $v \in A_{p, \B}$. Since $p\le q$ we clearly have by Lemma~\ref{lemma:AP-trivial} that $u^q\, v \in A_{p, \B}\subset A_{q, \B}$. On the other hand, $u^p\,v=(u^q\,v)^\theta\,v^{1-\theta}$ with $\theta=p/q\in [0,1]$, hence Lemma~\ref{lemma:AP-trivial} readily gives that  $u^p\, v \in A_{p, \B}$. 
\end{proof}

\begin{lemma}\label{lem:RHApq}
Let  $(\Sigma, \mu)$ be a non-atomic $\sigma$-finite measure space and let  $\B$ be a basis. Let $0<\p_{-}\le p \le q \le \p_{+}\le \infty$. Then $u^p \in A_{p/\p_{-}, \B} \cap RH_{(\p_{+}/p)', \B}$ and $u^q \in A_{q/\p_{-}, \B} \cap RH_{(\p_{+}/q)', \B}$  
if and only if $u^{\p_{-}} \in A_{p/\p_{-}, (q/\p_-)(\p_+/q)', \B}$.  
\end{lemma}

\begin{proof}
Let us observe that it suffices to consider the case $\p_{-}=1$. Indeed, set $\widetilde{\p}_{-}=1$, $\widetilde{p}=p/\p_{-}$, $\widetilde{q}=q/\p_{-}$, and 
$\widetilde{\p}_{+}=\p_{+}/\p_{-}$ so that  $1\le \widetilde{p}\le\widetilde{q}\le\widetilde{\p}_{+}\le\infty$. Then from the case $\widetilde{\p}_{-}=1$ we obtain
\begin{multline*}
u^p \in A_{p/\p_{-}, \B} \cap RH_{(\p_{+}/p)', \B}\, \text{ and }\,  u^q \in A_{q/\p_{-}, \B} \cap RH_{(\p_{+}/q)', \B} 
\\
\Longleftrightarrow\ 
(u^{\p_-}\big)^{\widetilde{p}} \in A_{\widetilde{p}, \B} \cap RH_{(\widetilde{\p}_{+}/\widetilde{p})', \B}
\, \text{ and }\,  
(u^{\p_-}\big)^{\widetilde{q}} \in A_{\widetilde{q}, \B} \cap RH_{(\widetilde{\p}_{+}/\widetilde{q})', \B}
\\
\Longleftrightarrow\ 
u^{\p_{-}} \in A_{\widetilde{p}, \widetilde{q}(\widetilde{\p}_+/\widetilde{q})', \B}\ \Longleftrightarrow\ u^{\p_{-}} \in A_{p/\p_{-}, (q/\p_-)(\p_+/q)', \B}.
\end{multline*}
	
Assume then that $\p_{-}=1$.
We first prove the forward implication. Let $u^p \in A_{p, \B} \cap RH_{(\p_{+}/p)', \B}$ and $u^q \in A_{q, \B} \cap RH_{(\p_{+}/q)', \B}$. By \eqref{eq:JN}, the latter gives that $u^{q(\p_+/q)'} \in A_{\tau_q, \B} \subset A_{q(\p_+/q)', \B}$, where $\tau_q=(\p_+/q)'(q-1)+1$. Now we have $u^p \in A_{p, \B}$ and $u^{q(\p_+/q)'} \in A_{q(\p_+/q)', \B}$, which is equivalent to $u \in A_{p, q(\p_+/q)', \B}$. To show the backward implication, we assume that $u \in A_{p, q(\p_+/q)', \B}$. Observe that $p(\p_+/p)' \le q(\p_+/q)'$ and $q(\p_+/q)'(\tau'_q -1)=q' \le p'=p(\p_+/p)'(\tau'_p -1)$. Hence, Jensen's inequality implies 
\begin{multline*}
\bigg(\fint_{B}u^{p(\p_+/p)'}d\mu \bigg) \bigg(\fint_{B}u^{p(\p_+/p)'(1-\tau'_p)}d\mu \bigg)^{\tau_p-1}
\\
\le \bigg[\bigg(\fint_{B}u^{q(\p_+/q)'}d\mu\bigg)^{\frac{1}{q(\p_+/q)'}} 
\bigg(\fint_{B}u^{-p'}d\mu \bigg)^{\frac{1}{p'}}\bigg]^{p(\p_+/p)'} 
\le [u]_{A_{p, q(\p_+/q)', \B}}^{p(\p_+/p)'}, 
\end{multline*}
and
\begin{multline*}
\bigg(\fint_{B}u^{q(\p_+/q)'}d\mu \bigg) \bigg(\fint_{B}u^{q(\p_+/q)'(1-\tau'_q)}d\mu \bigg)^{\tau_q-1}
\\
\le \bigg[\bigg(\fint_{B}u^{q(\p_+/q)'} d\mu\bigg)^{\frac{1}{q(\p_+/q)'}} 
\bigg(\fint_{B}u^{-p'}d\mu\bigg)^{\frac{1}{p'}}\bigg]^{q(\p_+/q)'} 
\le [u]_{A_{p, q(\p_+/q)', \B}}^{q(\p_+/q)'}. 
\end{multline*}
This shows $u^{p(\p_+/p)'} \in A_{\tau_p, \B}$ and $u^{q(\p_+/q)'} \in A_{\tau_q, \B}$, which by \eqref{eq:JN} are equivalent to $u^p \in A_{p, \B} \cap RH_{(\p_{+}/p)', \B}$ and $u^q \in A_{q, \B} \cap RH_{(\p_{+}/q)', \B}$ respectively.
\end{proof}

\subsection{Banach function spaces}\label{sec:BFS}

Continue to assume that $(\Sigma, \mu)$ is a given non-atomic $\sigma$-finite measure space with $\mu(\Sigma)>0$.  Let $\nu$ be another non-atomic $\sigma$-finite measure with $\nu(\Sigma)>0$. A typical example that we will consider throughout this paper is that on which $\nu$ is a ``weighted measure'' (with respect to $\mu$).  To elaborate on this, let $v$ be a non-negative $\mu$-measurable function such that $v<\infty$ $\mu$-a.e.~and $v$ is strictly positive on a set of $\mu$-positive measure. Consider $d\nu=v\,d\mu$, in which case we agree to identify the weight function $v$ with the weighted measure $dv := v d\mu$. We claim that $(\Sigma, v)$ is a non-atomic $\sigma$-finite measure space with $v(\Sigma)>0$ (we implicitly assume that the measurable sets for $v$ are those for $\mu$). That $v(\Sigma)>0$ follows from the assumption that $v$ is strictly positive on a set of $\mu$-positive measure. To see that $v$ is $\sigma$-finite, we first use that $\mu$ is  $\sigma$-finite to find a family $\{\Sigma_j\}_{j\in\N}$ of $\mu$-measurable sets with $\Sigma=\bigcup_{j\in\N} \Sigma_j$ and $\mu(\Sigma_j)<\infty$ for every $j$. Consider next $F_k:=\{x\in\Sigma: v(x)\le 2^k\}$, $k\in\N$, and $F_\infty=\Sigma\setminus \bigcup_{k\in\N} F_k$ so that $\Sigma=F_\infty\bigcup(\bigcup_{j,k\in\N} \Sigma_j\cap F_k)$. Note that for every $j,k\in\N$
\[
v(\Sigma_j\cap F_k)
=\int_{\Sigma_j\cap F_k}v\,d\mu
\le 2^k\mu(\Sigma_j)<\infty.
\]
and also $v(F_\infty)=0$ since $v<\infty$ $\mu$-a.e. To see that $v$ is non-atomic, take an arbitrary $\mu$-measurable (hence $v$-measurable) with $v(F)>0$. By definition we clearly have that $\mu(F)>0$. Set $F_0=F\cap\{v>0\}$ so that $v(F)=v(F_0)$. Since $\mu$ is non-atomic we can find $F_0'\subset F_0$ so that $\mu(F_0)>\mu(F_0')>0$. Note that $v(F_0')>0$, otherwise $v=0$ $\mu$-a.e.~in $F_0'$, which can only happens if $\mu(F_0')=0$, thus we get a contradiction. Also, we should have  $v(F_0)>v(F_0')$. Otherwise, $v(F_0\setminus F_0')=0$ which implies that $v=0$ $\mu$-a.e.~in $F_0\setminus F_0'$, and this can only happens if $\mu(F_0\setminus F_0')=0$ which leads us again to a contradiction.

Assume in what follows that $(\Sigma, \mu)$ is a given non-atomic $\sigma$-finite measure space with $\mu(\Sigma)>0$ and that $\nu$ is another non-atomic $\sigma$-finite measure with $\nu(\Sigma)>0$. Define, 
\begin{equation*}
\M_{\nu}:=\{f:\Sigma\to\mathbb{C}:\,f\text{ is }\nu\text{-measurable}\}.
\end{equation*} 

\begin{definition}\label{def:BFS}
We say that a mapping $\|\cdot\|:\M_{\nu}\to[0,\infty]$ is a {\tt function norm} provided following properties are satisfied for all $f,g\in \M_{\nu}$: 
\begin{list}{\textup{(\theenumi)}}{\usecounter{enumi}\leftmargin=1cm \labelwidth=1cm \itemsep=0.2cm 
		\topsep=.2cm \renewcommand{\theenumi}{\roman{enumi}}}
\item $\|f\|=\big\| |f|\big\|$ and $\|f\|=0$ if and only if $f=0$ $\nu$-a.e.

\item $\norm{f+g} \leq\|f\| + \|g\|$.

\item $\norm{\lambda f}=|\lambda|\|f\|$ for every $\lambda \in\R$.

\item If $|f|\leq|g|$ $\nu$-a.e., then $\|f\| \leq\|g\|$.

\item If $\{f_j\}_{j\in\mathbb{N}} \subset\mathbb{M}_{\nu}$ is a sequence such that $|f_j|$ increases 
to $|f|$ $\nu$-a.e.~as $j\to\infty$, then $\norm{f_j}$ increases to $\|f\|$ as $j\to\infty$.

\item If $E\subset\Sigma$ is a $\nu$-measurable set with $\nu(E)<\infty$, then one has $\norm{\mathbf{1}_E}<\infty$ and there is a constant $C_E\in(0,\infty)$ such that $\int_E |f| d\nu\leq C_E \|f\|$.  
\end{list}
\end{definition}

Given a function norm $\|\cdot\|$, the set 
\begin{equation*}
\X=\{f\in\mathbb{M}_{\nu}: \|f\|<\infty\} 
\end{equation*} 
is called a {\tt Banach function space} over $(\Sigma,\nu)$.  In such a scenario, we shall write $\|\cdot\|_{\X}$ in place of $\|\cdot\|$ in order to emphasize the connection between the function norm $\|\cdot\|$ and its associated function space $\X$. Then $(\X,\|\cdot\|_{\X})$ is a Banach space.   

For a Banach function space $\X$ over $(\Sigma,\nu)$,  it is not difficult to see that one can define its associate space $\X'$ by means of the function form 
\begin{equation*}
\|f\|_{\X'}=\sup\left\{\int_{\Sigma}|f(x)g(x)|\,d\nu(x): \,g\in\X,\,\,\|g\|_{\X}\leq 1\right\},
\end{equation*}
and with this definition $\X'$ is also a Banach function space. Then it follows from the definition of $\X'$ that the following generalized Hölder's inequality holds: 
\begin{equation}\label{eq:Holder}
\int_{\Sigma}|f(x)g(x)|\,d\nu(x)\leq\|f\|_{\X}\|g\|_{\mathbb{X}'},
\quad f\in\X \, \text{ and }\, g\in\X'.
\end{equation}
It turns out that $\X=(\X')'=:\X''$ (cf. \cite[Theorem 2.7, p.\,10]{BS}). Therefore, one has 
\begin{equation}\label{eq:fX-norm}
\|f\|_{\X}=\sup\left\{\int_{\Sigma}|f(x)g(x)|\,d\nu(x):
\,g\in\mathbb{X}',\,\,\|g\|_{\mathbb{X}'}\leq 1\right\}.
\end{equation}

\begin{remark}\label{rem:supoverg}
It is useful to note that the supremum in \eqref{eq:fX-norm} does not change if it is taken only over functions $g\in\X'$ with $\|g\|_{\X'}\leq 1$ which are non-negative and positive on a set of positive $\mu$-measure (that is, non-negative $g\in\X'$ with $0<\|g\|_{\X'}\leq 1$). Indeed, the fact that we can consider only non-negative functions can be directly seen from \eqref{eq:fX-norm}. If $\|f\|_{\X}>0$ then there is $g\in\X'$ with $\|g\|_{\X'}\leq 1$ such that $0<\|f\|_{\X} \leq 2 \int_{\Sigma} |f\,g|\,d\nu$ and this forces $g$ to be non-zero on a set of positive $\nu$-measure. Finally, the case $\|f\|_{\X}=0$ is trivial.
\end{remark}

Given a Banach function space $\X$ over $(\Sigma,\nu)$, we define the scale of spaces $\X^r$ with $r \in(0,\infty)$ by
\begin{equation}\label{Xr}
\X^r:=\{f\in\mathbb{M}_\nu: |f|^r \in\X\} \quad\text{and}\quad \|f\|_{\X^r} := \norm{|f|^r}_{\X}^{1/r}. 
\end{equation}
If $r \ge 1$, then $\X^r$ is again an actual norm and $\X^r$ is a Banach function space. However, if $r<1$, then $\X^r$ needs not be a function space. 

To define a rearrangement invariant function space, let $\nu_{f}$ denote the distribution function of $f\in\mathbb{M}_{\nu}$: 
\begin{equation*}
\nu_f(\lambda) = \nu(\{x \in \Sigma: |f(x)|>\lambda\}), \qquad \lambda\in[0,\infty). 
\end{equation*}
A Banach function space $\X$ over $(\Sigma ,\nu)$ is {\tt rearrangement invariant} if $\|f\|_{\X}=\|g\|_{\X}$ for every pair of functions $f,g\in \X$ such that $\nu_f=\nu_g$. This means that the function norm of $f$ in $\X$  depends only on its distribution function. Note that if $(\Sigma, \nu)$ is nonatomic, then $\X$ is rearrangement invariant if and only if $\X'$ is (cf. \cite[Corollary~4.4,~p.60]{BS}). 

For each $f\in\mathbb{M}_{\nu}$, the decreasing rearrangement of $f$ with respect to $\nu$ is defined as 
\begin{equation*}
f^{*}_{\nu}(t)=\inf\{\lambda\geq 0:\nu_f(\lambda) \leq t\},\quad t \in[0, \nu(\Sigma)].
\end{equation*} 
Note that the functions $f$ and $f^{*}_{\nu}$ have the same distribution function. One remarkable consequence of this is
Luxemburg representation theorem: if $\X$ is rearrangement invariant Banach function space over $(\Sigma ,\nu)$, then there exists a rearrangement invariant Banach function space $\overline{\X}$ over $[0,\nu(\Sigma))$ such that $f\in\X$ if and only if  $f^{*}_{\nu}\in\overline{\X}$ and $\|f^{*}_{\nu}\|_{\overline{\X}}=\|f\|_{\X}$ (cf.  \cite[Theorem 4.10, p.\,62]{BS}). This allows us to define the {\tt Boyd indices} by 
\begin{equation}\label{eq:Boydindices}
p_{\X}:=\lim_{t\to\infty} 
\frac{\log t}{\log\norm{D_t}_{\overline{\X}\to\overline{\X}}}, \qquad q_{\X} 
:=\lim_{t\to 0^{+}}\frac{\log t}{\log\norm{D_t}_{\overline{\X}\to\overline{\X}}}, 
\end{equation}
where $D_t:\overline{\X}\to\overline{\X}$ is the dilation operator defined by
\begin{equation}\label{eq:Dilationoperators}
D_t f(s):=\left\lbrace
\begin{array}{lcl}
f(s/t), &\text{ if}& s\in[0,\,\nu(\Sigma)\min\{1,t\}],
\\[4pt]
0, &\text{ if}& s\in(\nu(\Sigma)\min\{1,t\},\,\nu(\Sigma)].
\end{array}
\right. 
\end{equation}
It is not hard to see that
\begin{equation}\label{eq:Boyd}
1\leq p_{\X}\leq q_{\X}\leq\infty, \quad 
(p_{\X})'=q_{\X'} \quad\text{and}\quad  (q_{\X})'=p_{\X'},
\end{equation} 
and also that $p_{\X^r}=r\,p_{\X}$ and $q_{\X^r}=r\,q_{\X}$.

Given $\X$, a rearrangement invariant Banach function space defined over the measure space $(\Sigma, \mu)$, we want to define a weighted version $\X(v)$. Let $\B$ be a basis and let $v$ be a $\B$-weight  (i.e., $0<v(x)<\infty$ for $\mu$-a.e.~$x \in \Sigma$). Then $(\Sigma, v)$, where $dv=v\,d\mu$, is a non-atomic $\sigma$-finite measure space with $\mu(\Sigma)>0$.
Define the weighted space 
\begin{align}\label{def:Xv:ri}
\X(v) := \{f \in \M_{\nu}: \|f^*_{v}\|_{\overline{\X}}<\infty\}.  
\end{align}
This is a Banach function space over $(\Sigma, v)$ with norm $\|f\|_{\X(v)} := \|f^*_{v}\|_{\overline{\X}}$. Then one has $\X'(v)=\X(v)'$ and 
\begin{align}
\|f\|_{\X(v)} = \sup \left\{\int_{\Sigma} |f(x) g(x)| \, dv(x): \, g \in \X'(v), \|g\|_{\X'(v)} \le 1 \right\}.
\end{align}
Similarly, $\X(v)^r = \X^r(v)$ for $0<r<\infty$. 

Continue to assume that $(\Sigma, \mu)$ is a given non-atomic $\sigma$-finite measure space with $\mu(\Sigma)>0$ and let $\B$ be a basis. Let $v$ be a non-negative $\mu$-measurable function such that $v<\infty$ $\mu$-a.e.~and $v$ is strictly positive on a set of $\mu$-positive measure. We can then consider the associated weighted measure $v\,d\mu$, in which case we agree to identify the weight function $v$ with the weighted measure $dv := v d\mu$. Let $\X_v$ be a Banach function space over $(\Sigma, v)$. 

Recall the definition of the Hardy-Littlewood maximal operator $M_{\B}$ in \eqref{eq:M-def} (where the underlying measure is $\mu$) and introduce $M'_{\B, v}h(x) := M_{\B}(h\,v)(x)/v(x)$ if $v(x) \neq 0$, $M'_{\B, v}h(x)=0$ otherwise.  Observe that $M'_{\B, v}=M_{\B}$ if $v \equiv 1$. We are interesting in deriving some properties from the boundedness of $M_{\B}$ and its ``dual'' operator $M'_{\B, v}$: 
\begin{align}
\label{eq:Xv-N1} \|(M_{\B}h)\,u\|_{\X_v} &\leq \mathcal{N}_1 \|h\,u\|_{\X_v}, \quad \forall\,h \in\mathbb{M}_{\mu}, 
\\[0.2cm]
\label{eq:Xv-N2} \|(M'_{\B, v}h)\,u^{-1}\|_{\X'_v} &\leq \mathcal{N}_2 \|h\,u^{-1}\|_{\X'_v}, \quad \forall\,h \in\mathbb{M}_{\mu}.
\end{align}
\begin{lemma}\label{lem:AXprev}
Fix $(\Sigma, \mu)$, a non-atomic $\sigma$-finite measure space with $\mu(\Sigma)>0$, and a basis $\B$. Consider $v$, a non-negative $\mu$-measurable function so that $v(B)<\infty$ for every $B\in\B$  and, with the notation above, $v(\Sigma)>0$; and $\X_v$ a Banach function space over $(\Sigma, v)$. Let $u:\Sigma \to [0, \infty]$ be so that \eqref{eq:Xv-N1} holds. Then, for every $B \in \B$, the following hold:
\begin{list}{\textup{(\theenumi)}}{\usecounter{enumi}\leftmargin=1cm \labelwidth=1cm \itemsep=0.2cm 
		\topsep=.2cm \renewcommand{\theenumi}{\alph{enumi}}}

\item If $\|\mathbf{1}_B u\|_{\X_v}=\infty$ then $u \equiv\infty$ $\mu$-a.e.~in $B$. Hence, either $\|\mathbf{1}_B u\|_{\X_v}<\infty$ or $u \equiv\infty$ $\mu$-a.e.~in $B$; and either $u \in L^1(B, v)$ or $u \equiv\infty$ $\mu$-a.e.~in $B$.

\item If $\|\mathbf{1}_B u\|_{\X_v}=0$ then $u \equiv0$ $v$-a.e.~in $B$. Hence, either $\|\mathbf{1}_B u\|_{\X_v}>0$ or $u \equiv 0$ $v$-a.e.~in $B$; and either $u>0$ $v$-a.e.~in $B$ or $u \equiv0$ $v$-a.e.~in $B$.
\end{list}

In particular, if in addition $0<u, v<\infty$ $\mu$-a.e.~in $\Sigma_\B$, then
\begin{equation}\label{eq:BuXv}
0<\|\mathbf{1}_B\,u\|_{\X_v}<\infty \, \text{ for every } B \in \B. 
\end{equation}
\end{lemma}

\begin{proof}
Fix $B \in \B$. For every $0\le h \in \M_{\mu}$ with $0<\fint_B h\, d\mu < \infty$ we have that \eqref{eq:Xv-N1} gives
\begin{equation*}
\bigg(\fint_B h d\mu\bigg)\,
\|\mathbf{1}_B u\|_{\X_v}
\leq
\|\mathbf{1}_B\,M_{\B}(h\mathbf{1}_B)\,u\|_{\X_v}
\leq
\|M_{\B}(h\mathbf{1}_B)\,u\|_{\X_v}
\leq 
\mathcal{N}_1 \|h\,\mathbf{1}_B u\|_{\X_v}.
\end{equation*}
In particular, for any $\mu$-measurable set $S\subset B$ with $\mu(S)>0$, taking $h=\mathbf{1}_S$ we obtain that
\begin{equation}\label{eq:SBSB}
\frac{\mu(S)}{\mu(B)} \|\mathbf{1}_B\,u\|_{\X_v}\leq \mathcal{N}_1 \norm{\mathbf{1}_S\,u}_{\X_v}.
\end{equation}

Assume now that $\|\mathbf{1}_B u\|_{\X_v}=\infty$. Then, \eqref{eq:SBSB} implies that 
\begin{equation}\label{eq:SuXv}
\norm{\mathbf{1}_S\,u}_{\X_v}=\infty \, \text{ for every measurable set $S \subset B$ with $\mu(S)>0$}.
\end{equation}
For every $r>0$, define 
\begin{equation*}
S_r := \{x \in B: 0 \le u(x) \le r\}.
\end{equation*}
If $\mu(S_r)>0$, then $\infty=\|\mathbf{1}_{S_r} u\|_{\X_v} \leq r \|\mathbf{1}_B\|_{\X_v}$ by \eqref{eq:SuXv}. This contradicts item (vi) in Definition~\ref{def:BFS} because $v(B)<\infty$. Hence $\mu(S_r)=0$ for every $r>0$, which in turn indicates that $u \equiv \infty$ $\mu$-a.e.~in $B$. On the other hand, if $u \not\in L^1(B, v)$, that is, $\int_{B} u dv=\infty$, then, using that $v(B)<\infty$  and item (vi) in Definition~\ref{def:BFS}, we obtain $\|\mathbf{1}_B u\|_{\X_v}=\infty$. We have already shown that the latter implies that $u \equiv \infty$ $\mu$-a.e.~in $B$. This finishes the proof of (a).

Next, assume that $\norm{\mathbf{1}_{B} u}_{\X_v}=0$. Then item (i) in Definition~\ref{def:BFS} readily implies that $u \equiv 0$ $v$-a.e.~in $B$. Furthermore, if $u \equiv 0$ in a measurable set $S \subset B$ with $v(S)>0$ (hence $\mu(S)>0$), then \eqref{eq:SBSB} implies that $\norm{\mathbf{1}_{B} u}_{\X_v}=0$. Thus, $u \equiv 0$ $v$-a.e.~in $B$. This proves (b).

Finally, \eqref{eq:BuXv} is a consequence of (a) and (b) since $u \equiv 0$ $v$-a.e.~in $B$ is equivalent to $u \equiv 0$ $\mu$-a.e.~in $B$ whenever $0<v<\infty$ $\mu$-a.e.
\end{proof}

\begin{lemma}\label{lem:X-geq1}
Fix $(\Sigma, \mu)$, a non-atomic $\sigma$-finite measure space with $\mu(\Sigma)>0$, and a basis $\B$. Let $u$ and $v$ be two $\B$-weights so that  
$v(B)<\infty$ for every $B\in\B$.
\begin{list}{\textup{(\theenumi)}}{\usecounter{enumi}\leftmargin=1cm \labelwidth=1cm \itemsep=0.2cm 
		\topsep=.2cm \renewcommand{\theenumi}{\alph{enumi}}}
\item If \eqref{eq:Xv-N1} holds, then $\mathcal{N}_1 \ge 1$. 

\item If \eqref{eq:Xv-N2} holds, then $\mathcal{N}_2 \ge 1$. 
\end{list}
\end{lemma}

\begin{proof} 
Let $B \in \B$. Observe that $M_{\B}(\mathbf{1}_B)(x) \equiv 1$ for every $x \in B$.  This and \eqref{eq:Xv-N1}, immediately yield
\begin{align}\label{eq:1BX-1BX}
\| \mathbf{1}_B\, u\|_{\X_v} =  
\| \mathbf{1}_B M_{\B}(\mathbf{1}_B)\,u\|_{\X_v}
\leq \|(M_{\B}\mathbf{1}_B)\,u\|_{\X_v} \leq \mathcal{N}_1 \|\mathbf{1}_B\, u\|_{\X_v}. 
\end{align}
By \eqref{eq:BuXv}, and the fact that $u$ and $v$ are $\B$-weights,  we have that $0<\|\mathbf{1}_B\, u\|_{\X_v}<\infty$. Hence re readily obtain that $\mathcal{N}_1 \geq 1$. 

To show (b), note that \eqref{eq:Xv-N2} is equivalent to 
\begin{align}\label{eq:MM}
\|(M_{\B} h)\,u^{-1}\, v^{-1} \|_{\X'_v} &\leq \mathcal{N}_2\|h\,u^{-1}\,v^{-1}\|_{\X'_v}, \quad \forall\,h \in \M_{\mu}.
\end{align}
Additionally, the fact that $u$ and $v$ are $\B$-weights gives that $u^{-1}\, v^{-1}$ are also $\B$-weights. Therefore, the conclusion (a) applied to \eqref{eq:MM} eventually yields $ \mathcal{N}_2 \geq 1$.  
\end{proof}

We now present examples of Banach function spaces.

\begin{example}\label{ex:Lp0}
Suppose that $(\Sigma, \mu)$ is a non-atomic $\sigma$-finite measure space with $\mu(\Sigma)>0$. Let $\B$ be a Muckenhoupt basis and let $u$, $v$ be two $\B$-weights. Set $\X=L^p(\Sigma,d\mu)$, $1<p<\infty$, which is a rearrangement invariant Banach function space over $(\Sigma, \mu)$. Clearly $\X(v)=L^p(\Sigma,v)$, in other words, the space $\X_v=L^p(\Sigma,v)$ is a Banach function space over $(\Sigma,v)$. Note that in this case $\X'_v=L^{p'}(\Sigma, v)$. Hence \eqref{eq:Xv-N1} and \eqref{eq:Xv-N2} can be rewritten respectively as
	\begin{align}
	\label{eq:exLp-1} \|M_{\B}h\|_{L^p(\Sigma, u^p\, v)} 
	&\leq \mathcal{N}_1 \|h\|_{L^p(\Sigma, \, u^p\, v)}, \quad \forall\,h \in \M_{\mu},
	\\
	\label{eq:exLp-2} \|M_{\B}h\|_{L^{p'}(\Sigma, \, u^{-p'}\,v^{1-p'})} 
	&\leq \mathcal{N}_2 \|h\|_{L^{p'}(\Sigma, \, u^{-p'}\,v^{1-p'})}, \quad \forall\,h \in \M_{\mu}. 
	\end{align}
Since $\B$ is a Muckenhoupt basis, we have that $u^p\,v \in A_{p, \B}$ yields \eqref{eq:exLp-1}, and $u^{-p'}v^{1-p'} \in A_{p', \B}$ implies \eqref{eq:exLp-2}. On the other hand, by definition $u^p\,v \in A_{p, \B}$ if and only if  $u^{-p'}v^{1-p'} \in A_{p', \B}$. All these show that in this scenario, \eqref{eq:Xv-N1} and \eqref{eq:Xv-N2} (or, \eqref{eq:exLp-1} and \eqref{eq:exLp-2}) hold  provided $u^p\,v \in A_{p, \B}$. 

As mentioned above, in $(\Rn,\Ln)$, the set of all cubes $\mathcal{Q}$,  the set of all dyadic cubes $\mathcal{D}$, the set of all rectangles $\mathcal{R}$  whose sides are parallel to the coordinate axes, and the collection $\mathcal{Z}$ of Zygmund rectangles in $(\re^3, \L^3)$ whose sides are parallel to the coordinate axes and have lengths $s$, $t$ and $st$ with $s,t>0$, are all Muckenhoupt bases. The same occurs if we consider the collection of all metric balls in  a non-atomic doubling measure metric space $(\Sigma,d,\mu)$. In each of these cases, \eqref{eq:exLp-1} and \eqref{eq:exLp-2} hold  provided $u^p\,v \in A_{p, \B}$. 

\end{example}

\begin{example}\label{ex:Lpvar}
Given a measurable function $p(\cdot):\Rn\to (1,\infty)$, then the norm
\begin{equation}\label{eq:Lpvar-norm}
\|f\|_{L^{p(\cdot)}(\Rn,\Ln)}:=\inf\Big\{\lambda>0:
\int_{\Rn} \Big(\frac{|f(x)|}{\lambda}\Big)^{p(x)}dx \leq 1\Big\}
\end{equation}
defines a Banach function space $L^{p(\cdot)}(\Rn,\Ln)$ called a {\tt variable Lebesgue space} whose associate space is $L^{p'(\cdot)}(\Rn,\Ln)$, where $1/p(x)+1/p'(x)=1$ for every $x\in \Rn$. The space $L^{p(\cdot)}(\Rn,\Ln)$ is not generally rearrangement invariant (typically, the norm of a ball depends on the location). In this context, we set 
\begin{equation}\label{eq:Lpvar-pp}
p_{-}:=\essinf_{x\in\Rn}p(x) \quad\text{and}\quad p_{+}:=\esssup_{x\in\Rn}p(x).
\end{equation}
Furthermore, we say that $p(\cdot)\in{\rm LH}$ if there are constants $C_0,C_{\infty}\in(0,\infty)$ and $p_{\infty}\in[0,\infty)$ such that
\begin{equation*}
|p(x)-p(y)|\leq\frac{C_0}{-\log(|x-y|)}, \quad |x-y|<1/2, 
\end{equation*}
and
\begin{equation*}
|p(x)-p_{\infty}|\leq\frac{C_{\infty}}{\log(e+|x|)}, \quad\quad x \in \Rn.
\end{equation*}
For further details on these spaces, the reader is referred to \cite{CF}.

Let $\B$ be the collection of all balls in $\Rn$. The class of $A_{p(\cdot)}=A_{p(\cdot), \B}$ weights consists of all weights that satisfy the condition 
\begin{equation}\label{eq:Lpvar-Apvar}
[w]_{A_{p(\cdot)}}:=\sup_{B } {\Ln(B)}^{-1}\,
\norm{w\mathbf{1}_B}_{L^{p(\cdot)}(\Rn,\Ln)}
\|w^{-1}\mathbf{1}_B\|_{L^{p'(\cdot)}(\Rn,\Ln)}<\infty. 
\end{equation}
Under the background hypothesis that $1<p_{-}\leq p_{+}<\infty$ and $p(\cdot)\in{\rm LH}$, it turns out that in \cite{CDH, CFN} it is proved that 
\begin{align}\label{eq:Ap=AX}
w \in A_{p(\cdot)} &\Longleftrightarrow \|(Mh)\, w\|_{L^{p(\cdot)}(\Rn,\Ln)} \le C_1 \|h\,w\|_{L^{p(\cdot)}(\Rn,\Ln)}, 
\nonumber\\
 &\Longleftrightarrow \|(Mh)\, w^{-1}\|_{L^{p'(\cdot)}(\Rn,\Ln)} \le C_2 \|h\,w^{-1}\|_{L^{p'(\cdot)}(\Rn,\Ln)}.
\end{align}
In particular, by \cite[Theorem~2.4]{CW} one can easily see that 
\begin{align}
w \equiv 1 \in A_{p(\cdot)},\quad \forall\,p(\cdot) \in{\rm LH} \text{ with } 1<p_{-}\leq p_{+}<\infty. 
\end{align}

We would like to observe that in the general case, with $p(\cdot)$ not satisfying the ${\rm LH}$-condition, it is not known whether any of the equivalences in \eqref{eq:Ap=AX} holds for general weights. Nonetheless, \cite[Theorem~1.2]{Ko} shows that if $p(\cdot)$ is constant outside some large ball and $1<p_{-} \leq p_{+}<\infty$, then 
\begin{align}\label{eq:M-Lp-Ap}
M \text{ is bounded on } L^{p(\cdot)}(\Rn,\Ln) \quad\text{if and only if}\quad 1 \in A_{p(\cdot)}. 
\end{align}

\end{example}

\begin{example}\label{ex:RIBFS}
Suppose that $(\Sigma, \mu)$ is a non-atomic $\sigma$-finite measure space with $\mu(\Sigma)>0$. Let $\B$ be a Muckenhoupt basis so that $A_{\infty, \B}\subset\bigcup_{s>1} RH_{s, \B}$,  and let $\X$ be a rearrangement invariant Banach function space over $(\Sigma, \mu)$. 

Since $\B$ is a Muckenhoupt basis and $w\equiv 1\in A_{p,\B}$ for all $p>1$, it follows that $M_\B$ is bounded on $L^p(\Sigma,\mu)$ for all $1<p<\infty$. This and Boyd's interpolation theorem (cf. \cite[Theorem 5.16, p.\,153]{BS} or \cite[Theorem 8.44]{MMMMM20}) imply that 
\begin{align}\label{eq:w=1}
1<p_{\X} \leq q_{\X}<\infty\quad\Longrightarrow\quad M_\B \text{ is bounded on $\X$ and $\X'$}.
\end{align}

On the other hand, we have already observed in Example~\ref{ex:Lp0} that $u^p\,v \in A_{p, \B}$ implies \eqref{eq:exLp-1} and \eqref{eq:exLp-2}. We will use this and Boyd's interpolation theorem to obtain \eqref{eq:Xv-N1}. Assume now that $1<p_{\X} \leq q_{\X}<\infty$. Let $v$ be a $\B$-weight so that $v \in A_{\infty, \B}$. We claim that
\begin{equation}\label{eq:RIBFS-M:u-v}
\|(M_{\B}h)\,u\|_{\X(v)} \leq C \|h\,u \|_{\X(v)}, \quad 
\text{whenever $u^{p_{\X}}\, v \in A_{p_{\X}, \B}$ and $u^{q_{\X}}\, v \in A_{q_{\X}, \B}$},
\end{equation}
for every $h \in \M_{\mu}$. To see this, fix  $u$ so that $u^{p_{\X}}\, v \in A_{p_{\X}, \B}$, $u^{q_{\X}}\, v \in A_{q_{\X}, \B}$. By Lemma~\ref{lem:ApRH} parts \eqref{list:open-2} and \eqref{list:open-3}, there exists $r\in (1,p)$ such that $u^{p_{\X}/r} v \in A_{p_{\X}/r, \B}$ and $u^{q_{\X}\,r} v \in A_{q_{\X}\,r, \B}$. Hence,  since $\B$ is a Muckenhoupt basis, we arrive at
\begin{align}
\label{eq:RIBFS-1} \|(M_{\B}h)\,u\|_{L^{p_{\X}/r}(\Sigma, v)}
&\leq C \|h u\|_{L^{p_{\X}/r}(\Sigma, v)}, \quad \forall\,h \in \M_{\mu},
\\
\label{eq:RIBFS-2} \|(M_{\B}h)\,u\|_{L^{q_{\X}\,r}(\Sigma, \, v)} 
&\leq C \|h u\|_{L^{q_{\X}\,r}(\Sigma, \, v)}, \quad \forall\,h \in \M_{\mu}.
\end{align} 
Equivalently, the sublinear operator $h\mapsto M_{\B}(h\, u^{-1})\,u$ is bounded both on $L^{p_{\X}/r}(\Sigma, v)$ and on $L^{q_{\X}\,r}(\Sigma, v)$. Then, Boyd's interpolation theorem (cf. \cite[Theorem 5.16, p.\,153]{BS} or \cite[Theorem 8.44]{MMMMM20}) in the measure space $(\Sigma,v)$ gives as desired \eqref{eq:RIBFS-M:u-v}. 

Our next claim is that 
\begin{equation}\label{eq:RIBFS-M:u-v:dual} 
\|(M_{\B}h)(u\,v)^{-1}\|_{\X'_v} \leq C \|h\, (u\,v)^{-1} \|_{\X'_v},  \quad 
\text{if $u^{p_{\X}}\, v \in A_{p_{\X}, \B}$ and $u^{q_{\X}}\, v \in A_{q_{\X}, \B}$},
\end{equation}
for every $h \in \M_{\mu}$. To show this we note that if we set $\widetilde{u}=(u\,v)^{-1}$, then \eqref{eq:RIBFS-M:u-v} gives that
\begin{equation}\label{eq:RIBFS-M:u-v:dual:alt} 
\|(M_{\B}h)\widetilde{u}\|_{\X'_v} \leq C \|h\, \widetilde{u}\|_{\X'_v},  \quad 
\text{if $\widetilde{u}^{p_{\X'}}\, v \in A_{p_{\X'}, \B}$ and $\widetilde{u}^{q_{\X'}}\, v \in A_{q_{\X'}, \B}$}.
\end{equation}
On the other hand, by \eqref{eq:Boyd} and Lemma~\ref{lemma:AP-trivial} we have 
\[
\widetilde{u}^{p_{\X'}}\, v \in A_{p_{\X'}, \B}
\quad\Longleftrightarrow\quad
u^{-(q_\X)'}\,v^{1-(q_\X)'}\in A_{(q_{\X})', \B}
\quad\Longleftrightarrow\quad
u^{q_\X}\,v\in A_{q_{\X}, \B}
\]
and
\[
\widetilde{u}^{q_{\X'}}\, v \in A_{q_{\X'}, \B}
\quad\Longleftrightarrow\quad
u^{-(p_\X)'}\,v^{1-(p_\X)'}\in A_{(p_{\X})', \B}
\quad\Longleftrightarrow\quad
u^{p_\X}\,v\in A_{p_{\X}, \B}.
\]
All these eventually yield \eqref{eq:RIBFS-M:u-v:dual}. Combining \eqref{eq:RIBFS-M:u-v} and \eqref{eq:RIBFS-M:u-v:dual},  we conclude that if $\X$ is a rearrangement invariant Banach function space over $(\Sigma,\mu)$ with Boyd index satisfying $1<p_\X\le q_\X <\infty$, then
\begin{equation}\label{eq:intro-CMP-ri}
\begin{array}{c}
\|(M_{\B}h)\,u\|_{\X(v)} \leq C \|h\,u \|_{\X(v)}
\text{ and } 
\|(M'_{\B, v} h)\,u^{-1}\|_{\X'(v)} \leq C \|h\,u^{-1}\|_{\X'(v)},
\\[4pt]
\text{whenever $u^{p_{\X}}\, v \in A_{p_{\X}, \B}$, $u^{q_{\X}}\, v \in A_{q_{\X}, \B}$, and $v \in A_{\infty, \B}$}.
\end{array}
\end{equation}
Before considering some particular examples we note that when $v\equiv 1$ one can easily see that  
\begin{equation}\label{eq:intro-CMP-ri:v=1}
\begin{array}{c}
	\|(M_{\B}h)\,u\|_{\X} \leq C \|h\,u \|_{\X}
	\text{ and } 
	\|(M_{\B} h)\,u^{-1}\|_{\X'} \leq C \|h\,u^{-1}\|_{\X'},
	\\[4pt]
	\text{whenever $u^{p_{\X}}\in A_{p_{\X}, \B}$, $u^{q_{\X}}\in A_{q_{\X}, \B}$, or equivalently, $u\in A_{p_\X,q_\X}$ (cf.~\eqref{Apq}).}
\end{array}
\end{equation}

Consider the Lorentz spaces  $\X=L^{p,q}(\Sigma,\mu)$, $1<p<\infty$, $1\le q<\infty$ defined by the function norm
\begin{equation*} 
\|f\|_{L^{p,q}(\Sigma, \mu)} := \left(\int_{0}^{\infty} \big( t^{\frac1p} f^*_{\mu}(t) \big)^q \frac{dt}{t}\right)^{\frac1q}, 
\end{equation*}
when $1\le q<\infty$, and 
\begin{equation*} 
\|f\|_{L^{p,\infty}(\Sigma, \mu)} := \sup_{0<t<\infty} t^{\frac1p} f^*_{\mu}(t).  
\end{equation*}
These are rearrangement invariant Banach function space over $(\Sigma, \mu)$ and the Boyd indices are $p_{\X}=q_{\X}=p\in (1,\infty)$, see  \cite[Theorem~4.6, p.\,219]{BS}.   Hence, \eqref{eq:intro-CMP-ri} becomes
\begin{equation}\label{eq:intro-CMP-ri:Lpq}
\begin{array}{c}
\|(M_{\B}h)\,u\|_{L^{p,q}(\Sigma, v)} \leq C \|h\,u \|_{L^{p,q}(\Sigma, v)}
\text{ and } \\[4pt] 
\|(M'_{\B, v} h)\,u^{-1}\|_{L^{p',q'}(\Sigma, v)} \leq C \|h\,u^{-1}\|_{L^{p',q'}(\Sigma, v)},
\\[4pt]
\text{whenever $u^{p}\, v \in A_{p, \B}$, $v \in A_{\infty, \B}$, $1<p<\infty$, $1\le q\le\infty$}.
\end{array}
\end{equation}
Note that when $p=q$ this agrees with Example~\ref{ex:Lp0}.

The same occurs with the Orlicz space $\X=L^p(\log L)^\alpha(\Sigma,\mu)$ with $1<p<\infty$, $\alpha\in\re$. This is a rearrangement invariant Banach function space whose function norm is given  by 
\begin{equation}\label{Orlicz-norm}
\|f\|_{L^p(\log L)^\alpha(\Sigma,\mu)}:=\inf\Big\{\lambda>0:
\int_{\Rn} \Phi\Big(\frac{|f(x)|}{\lambda}\Big)dx \leq 1\Big\},
\end{equation}
where $\Phi(t)=t^p\,\log(e+t)^\alpha$, $t\ge 0$.
In this case we also have $p_{\X}=q_{\X}=p\in (1,\infty)$ and therefore \eqref{eq:intro-CMP-ri} means that
\begin{equation}\label{eq:intro-CMP-ri:LplogL}
\begin{array}{c}
\|(M_{\B}h)\,u\|_{L^p(\log L)^\alpha(\Sigma, v)} \leq C \|h\,u \|_{L^p(\log L)^\alpha(\Sigma, v)}
\text{ and } 
\\[4pt]
\|(M'_{\B, v}h)\,u^{-1}\|_{L^{p'}(\log L)^{\alpha\,(1-p')}(\Sigma, v)} \leq C \|h\,u^{-1}\|_{L^{p'}(\log L)^{\alpha\,(1-p')}(\Sigma, v)},
\\[4pt]
\text{whenever $u^{p}\, v \in A_{p, \B}$, $v \in A_{\infty, \B}$, $1<p<\infty$, $\alpha\in\re$}.
\end{array}
\end{equation}
Let us point out that when $\alpha=0$ we are back to Example~\ref{ex:Lp0}.

Our last examples are $\X=(L^4+L^6)(\Sigma,\mu)$ or $\X=(L^4\cap L^6)(\Sigma,\mu)$ which are Orlicz spaces whose function norms are given as in \eqref{Orlicz-norm} with $\Phi(t)\approx \max\{t^4,t^6\}$  and $\Phi(t)\approx \min\{t^4,t^6\}$, $t\ge 0$, respectively. In either case, if $\mu(\Sigma)=\infty$ then $p_\X=4$ and $q_\X=6$, hence we can rewrite \eqref{eq:intro-CMP-ri} as
\begin{equation}\label{eq:intro-CMP-ri:LplogL2}
\begin{array}{c}
\|(M_{\B}h)\,u\|_{\X(v)} \leq C_0 \|h\,u \|_{\X(v)}
\text{ and } 
\|(M'_{\B, v}h)\,u^{-1}\|_{\X'(v)} \leq C \|h\,u^{-1}\|_{\X'(v)},
\\[4pt]
\text{whenever $u^{4}\, v \in A_{4, \B}$, $u^{6}\, v \in A_{6, \B}$, and $v \in A_{\infty, \B}$,}
\end{array}
\end{equation}
with $\X=(L^4+L^6)(\Sigma,\mu)$ or $\X=(L^4\cap L^6)(\Sigma,\mu)$, where the associated spaces are respectively $\X'=(L^{4/3}\cap L^{6/5})(\Sigma,\mu)$ or $\X'=(L^{4/3}+ L^{6/5})(\Sigma,\mu)$. Of course the same can be done with $\X=(L^p+L^q)(\Sigma,\mu)$ or $\X=(L^p\cap L^q)(\Sigma,\mu)$ with $1<p, q<\infty$ in which case 
$p_\X=\min\{p,q\}$ and $q_\X=\max\{p,q\}$. Further details are left to the interested reader.

As mentioned above, in $(\Rn,\Ln)$, the set of all cubes $\mathcal{Q}$,  the set of all dyadic cubes $\mathcal{D}$, the set of all rectangles $\mathcal{R}$  whose sides are parallel to the coordinate axes, and the collection $\mathcal{Z}$ of Zygmund rectangles in $(\re^3, \L^3)$ whose sides are parallel to the coordinate axes and have lengths $s$, $t$ and $st$ with $s,t>0$, are all Muckenhoupt bases. The same occurs if we consider the collection of all metric balls in  a non-atomic doubling measure metric space $(\Sigma,d,\mu)$. In all the cases one has $A_{\infty, \B}\subset\bigcup_{s>1} RH_{s, \B}$. 
\end{example}

We close this subsection with the following technical lemma. 
\begin{lemma}\label{lem:hh}
	Fix $(\Sigma, \mu)$, a non-atomic $\sigma$-finite measure space with $\mu(\Sigma)>0$, and a basis $\B$. Let $u$ and $v$ be $\B$-weights so that  $v(B)<\infty$ for every $B\in\B$ and $v(\Sigma)>0$. Then for every $\varepsilon>0$ and for every non-negative function $h$ on $\Sigma$ with $\|h\,u\|_{\X_v}\le 1$, there exists a function $h_{\varepsilon}$ on $\Sigma$ such that 
	\begin{align}\label{eq:hhu}
		h \le h_{\varepsilon},\quad h_{\varepsilon} > 0 \ \mu \text{-a.e.}\quad\text{and}\quad \|h_{\varepsilon}u\|_{\X_v} \le 1+\varepsilon.
	\end{align} 
\end{lemma}

\begin{proof}
	Since $(\Sigma, \mu)$ is $\sigma$-finite, there exists an increasing sequence of $\mu$-measurable sets $\{\Sigma_j\}_{j=1}^{\infty}$ such that $\Sigma=\bigcup_{j=1}^{\infty} \Sigma_j$ and $0<\mu(\Sigma_j)<\infty$ for each $j$. Set
\begin{equation}\label{def:E0-Ej}
		\begin{split}
		E_0 &:= \{x \in \Sigma: u(x)=0 \text{ or } u(x)=\infty \text{ or } v(x)=0 \text{ or } v(x)=\infty\},
		\\
		E_j &:=\{x \in \Sigma_j: 0<u(x) \le j, 0<v(x) \le j\},\quad j \ge 1,
				\end{split}
	\end{equation}
	Let us observe that $\Sigma=\bigcup_{j=0}^{\infty} E_j$ and also that $\{E_j\}_{j\ge 1}$ is a non-decreasing sequence. Hence, there exists $j_0\ge 1$ such that $\mu(E_{j_0})>0$ and $\mu (E_j)=0$ for $1\le j< j_0$. Note also that since $u$ and $v$ are $\B$-weights, then $\mu(E_0)=0$. Define then $E_0'= \Sigma\setminus \bigcup_{j=j_0}^\infty E_j$ which clearly satisfies $\mu(E_0')=0$. We claim that
	\begin{align}\label{eq:pinf}
		0 < \|\mathbf{1}_{E_j}u\|_{\X_v} < \infty, \quad \forall\,j \ge j_0. 
	\end{align}
	In fact, $\|\mathbf{1}_{E_j}u\|_{\X_v} \le j \|\mathbf{1}_{E_j}\|_{\X_v}<\infty$ since $v(E_j) \le j \mu(\Sigma_j)<\infty$. If $\|\mathbf{1}_{E_j}u\|_{\X_v}=0$, then item (i) in Definition \ref{def:BFS} gives that $u=0$ $v$-a.e.~in $E_j$. Since $0<v<\infty$ $\mu$-a.e.~in $\Sigma$, it follows that $u=0$ $\mu$-a.e.~in $E_j$, which leads to a contradiction since $\mu(E_j)>0$ for every $j\ge j_0$. Thus,  \eqref{eq:pinf} holds. Now we define 
	\begin{align*}
		h_{\varepsilon}(x) := h(x)+\varepsilon\, F(x) \quad\text{and}\quad F(x) := \mathbf{1}_{E_0'}(x) + \sum_{j=j_0}^{\infty} 2^{-j} \frac{\mathbf{1}_{E_j}(x)}{\|\mathbf{1}_{E_j}u\|_{\X_v}}, \quad x \in \Sigma. 
	\end{align*}
	Note that by \eqref{eq:pinf},  $h_{\varepsilon}$ and $F$ are well defined in $\Sigma$. 
	
	By construction $h \le h_{\varepsilon}$. On the other hand, the fact that $\mu(E_0')=0$ implies that $v(E_0')=0$, and hence $\|\mathbf{1}_{E_0'}\,u\|_{\X_v}=0$ by item (i) in Definition \ref{def:BFS}. This and the properties listed in Definition \ref{def:BFS} lead to $\|F\,u\|_{\X_v}\le 1$ and, thus,
	\[
	\|h_\varepsilon\,u\|_{\X_v}
	\le
	\|h\,u\|_{\X_v}
	+
	\varepsilon\,\|F\,u\|_{\X_v}
	\le
	1+\varepsilon.
	\]
	To complete the proof of \eqref{eq:hhu}, we just need to see that $F>0$ in $\Sigma \setminus E'_0$, since $\mu(E'_0)=0$.  Let $x \in \Sigma \setminus E_0'=\bigcup_{j=j_0}^{\infty} E_j$. Then there exists $j_x \ge j_0$ such that $x \in E_{j_x}$. From this and \eqref{eq:pinf}, we conclude that 
	\begin{align*}
		F(x) \ge 2^{-j_x} \frac{\mathbf{1}_{E_{j_x}}(x)}{\|\mathbf{1}_{E_{j_x}}u\|_{\X_v}} 
		= \frac{2^{-j_x}}{\|\mathbf{1}_{E_{j_x}}u\|_{\X_v}} >0. 
	\end{align*}
	The proof is complete. 
\end{proof}

\subsection{Modular spaces}

We say that $\Phi:[0,\infty)\to[0,\infty)$ is a {\tt Young function} if it is continuous, convex, strictly increasing, and satisfies 
\begin{equation}\label{eq:zi}
	\lim_{t\to 0^{+}}\frac{\Phi(t)}{t}=0,\qquad \lim_{t\to\infty}\frac{\Phi(t)}{t}=\infty.
\end{equation}
A function $\Phi$ satisfies the {\tt doubling condition}, or $\Phi \in \Delta_2$, if there is a constant $C_\Phi>0$ such that $\Phi(2t)\leq C_\Phi\,\Phi(t)$ for every $t>0$. Given a Young function $\Phi$, its complementary function $\overline{\Phi}: [0,\infty) \to [0,\infty)$ is defined as 
\begin{equation*}
	\overline{\Phi}(t):=\sup_{s>0}\{st-\Phi(s)\},
\end{equation*}
which clearly implies that 
\begin{equation}\label{eq:Young-ineq}
	st \leq \Phi(s) + \overline{\Phi}(t), \qquad s,t>0.
\end{equation}
Moreover, one can check that $\overline{\Phi}$ is also a Young function and
\begin{equation}\label{eq:Young-1}
	t\leq\Phi^{-1}(t)\overline{\Phi}^{-1}(t)\leq 2t,\qquad t>0.
\end{equation}
In turn, by replacing $t$ by $\Phi(t)$ in first inequality of \eqref{eq:Young-1}, we obtain
\begin{equation}\label{eq:Young-2}
	\overline{\Phi}\Big(\frac{\Phi(t)}{t}\Big)\leq\Phi(t),\qquad t>0.
\end{equation}
In analogy with \eqref{eq:Boydindices}, we define the dilation indices
\begin{equation}\label{eq:Phiindices}
	i_{\Phi} := \lim_{t\to0^{+}} \frac{\log h_{\Phi}(t)}{\log t} \qquad \text{and}\qquad 
	I_{\Phi} := \lim_{t\to\infty}\frac{\log h_{\Phi}(t)}{\log t}, 
\end{equation}
where $h_{\Phi}$ is defined as
\begin{equation*}
	h_{\Phi}(t) := \sup_{s>0} \frac{\Phi(st)}{\Phi(s)},\qquad t>0.
\end{equation*}
From the definitions, one can show that 
\begin{align}\label{dilation:duality}
	1\leq i_{\Phi}\leq I_{\Phi}\leq\infty,\qquad 
	(I_{\Phi})'=i_{\overline{\Phi}}, \qquad\text{and}\qquad (i_{\Phi})'=I_{\overline{\Phi}}.
\end{align} 
Additionally, it turns out that $\Phi\in\Delta_2$ if and only if $I_{\Phi}<\infty$, and hence 
\begin{equation}\label{eq:Delta2equiv}
	\Phi,\overline{\Phi}\in\Delta_2\qquad\text{ if and only if }\qquad
	1<i_{\Phi}\leq I_{\Phi}<\infty.
\end{equation}

Assume that $(\Sigma, \mu)$ is a given non-atomic $\sigma$-finite measure space with $\mu(\Sigma)>0$, for $f \in \M_{\mu}$ and a weight $w$, we define the modular $\rho_w^{\Phi}$ by
\begin{equation*}
	\rho_w^{\Phi}(f):=\int_\Sigma \Phi(|f(x)|)\,w(x)\,d\mu(x)=\int_\Sigma \Phi(|f(x)|)\,dw(x). 
\end{equation*}
When $w \equiv 1$, we write $\rho^{\Phi}$ in place of $\rho_w^{\Phi}$. We then introduce 
\begin{align*}
	\mathcal{M}_{w}^{\Phi} :=\{f\in \M_{\mu}: \rho_{w}^{\Phi}(f)<\infty\},
\end{align*}
which is referred to as a {\tt modular space}. 

As we have seen in Section \ref{sec:BFSextr}, to formulate the extrapolation results on Banach function spaces, one requires the boundedness of maximal operator $M_{\B}$ or its dual operator $M'_{\B,v}$. Likewise, to develop the extrapolation analogs on modular spaces, the assumptions in the current setting read as follows: 
\begin{align}
	\label{eq:PhiN1} \int_{\Sigma} \Phi((M_{\B}h)\, u)\, v\, d\mu 
	&\le \mathcal{N}_1\, \int_{\Sigma} \Phi(|h|\, u)\,v\, d\mu,\qquad\forall\,h \in \M_{\mu}, 
	\\
	\label{eq:PhiN2} \int_{\Sigma} \overline{\Phi}((M'_{\B, v}h)\, u^{-1})\, v\, d\mu 
	&\le \mathcal{N}_2\, \int_{\Sigma} \overline{\Phi}(|h|\, u^{-1})\,v\, d\mu,\qquad\forall\,h \in \M_{\mu}. 
\end{align}

\begin{lemma}\label{lem:APhiprev}
	Fix $(\Sigma, \mu)$, a non-atomic $\sigma$-finite measure space with $\mu(\Sigma)>0$, and a basis $\B$. Consider $v$, a non-negative $\mu$-measurable function so that $v(B)<\infty$ for every $B \in \B$ and $v(\Sigma)>0$; and $\Phi$ a Young function. Let $u: \Sigma \to [0, \infty]$ be such that \eqref{eq:PhiN1} holds. Then for every $B \in \B$, the following hold: 
	\begin{list}{\textup{(\theenumi)}}{\usecounter{enumi}\leftmargin=1cm \labelwidth=1cm \itemsep=0.2cm 
			\topsep=.2cm \renewcommand{\theenumi}{\alph{enumi}}}
		
		\item If  $\rho_v^{\Phi}(\mathbf{1}_B u)=\infty$ then $u \equiv\infty$ $\mu$-a.e.~in $B$. Hence, either  $\rho_v^{\Phi}(\mathbf{1}_B u)<\infty$ or $u \equiv\infty$ $\mu$-a.e.~in $B$; and either $u \in L^1(B, v)$ or $u \equiv\infty$ $\mu$-a.e.~in $B$.
		
		\item If  $\rho_v^{\Phi}(\mathbf{1}_B u)=0$  then $u \equiv0$ $v$-a.e.~in $B$. Hence, either  $\rho_v^{\Phi}(\mathbf{1}_B u)>0$  or $u \equiv 0$ $v$-a.e.~in $B$; and either $u>0$ $v$-a.e.~in $B$ or $u \equiv0$ $v$-a.e.~in $B$.
	\end{list}
	
	In particular, if in addition $0<u, v<\infty$ $\mu$-a.e.~in $\Sigma_\B$, then
	\begin{equation}\label{eq:Swfi}
		0<\rho_v^{\Phi}(\mathbf{1}_B u)<\infty \, \text{ for every } B \in \B.
	\end{equation}
\end{lemma}

\begin{proof}
	Fix $B\in \B$. Then, for every $h \in \M_{\mu}$ with $0<\fint_B h\, d\mu < \infty$ we have that 
	\begin{equation*}
		\int_{B}\Phi\Big(\Big(\fint_B h\,d\mu\Big)\, u\Big)\, v\,d\mu
		\leq\int_{\Sigma}\Phi(M_{\B}(\mathbf{1}_B h)\,u)\, v \,d\mu
		\leq \mathcal{N}_1\,\int_{B}\Phi(|h|u) v\,d\mu, 
	\end{equation*}
	where the last inequality follows from \eqref{eq:PhiN1}. In particular, for any measurable set $S\subset B$ with $\mu(S)>0$, taking $h=\frac{\mu(B)}{\mu(S)}\,\mathbf{1}_S$, we obtain 
	\begin{equation}\label{eq:APhiprev-3}
		\rho_v^{\Phi}(\mathbf{1}_B u) =
		\int_{B}\Phi( u)\, dv
		\leq \mathcal{N}_1\,\int_{B}\Phi\Big(\frac{\mu(B)}{\mu(S)}\,\mathbf{1}_S\,u\Big)\, dv.
	\end{equation}
	
	Assume next $\rho_v^{\Phi}(\mathbf{1}_B u)=\infty$. Set
	\begin{equation*}
		S_r:=\{x \in B: 0 \le u(x) \le r\}, \qquad r>0.
	\end{equation*}
	If $\mu(S_r)>0$ then  \eqref{eq:APhiprev-3} yields
	\begin{equation}\label{eq:APhiprev-5}
		\infty
		=
		\rho_v^{\Phi}(\mathbf{1}_B u) 
		\leq 
		\mathcal{N}_1\,\int_{B}\Phi\Big(\frac{\mu(B)}{\mu(S_r)}\,\mathbf{1}_{S_r}\,u\Big)\, dv
		\le
		\mathcal{N}_1\, \Phi\Big(\frac{\mu(B)}{\mu(S_r)}\,r\Big)\,v(B),
	\end{equation}
which contradicts our assumption $v(B)<\infty$ for every $B \in \B$. Hence $\mu(S_r)=0$ for every $r>0$, that is, $u \equiv \infty$ $\mu$-a.e.~in $B$. On the other hand, if $u \not\in L^1(B, v)$, then $\int_{B}u \, dv=\infty$. Note that $0<v(B)<\infty$ in this case. Hence, by Jensen's inequality, 
	\begin{align}
		\infty = \Phi\left(\fint_{B} u\, dv\right) 
		\le \fint_B \Phi(u)\, dv = v(B)^{-1} \rho_v^{\Phi}(\mathbf{1}_B u),  
	\end{align}
	which gives at once $\rho_v^{\Phi}(\mathbf{1}_B u)=\infty$. By the previous argument, $u \equiv \infty$ $\mu$-a.e.~in $B$. The proof of (a) is then complete.
	
	Next, assume that $\rho_v^{\Phi}(\mathbf{1}_B\,u)=0$. Then $\Phi(u(x))=0$ for $v$-a.e.~$x\in B$ and since $\Phi$ is strictly increasing with $\Phi(0)=0$, one readily gets $u \equiv 0$ $v$-a.e.~in $B$. On the other hand, if $u\equiv 0$ in a measurable set $S \subset B$ with $v(S)>0$ (hence $\mu(S)>0$), then \eqref{eq:APhiprev-3} gives  $\rho_v^{\Phi}(\mathbf{1}_B\,u)=0$, and hence, $u\equiv 0$ $v$-a.e. in $B$. This proves (b). 
	
	Finally, \eqref{eq:Swfi}  is a consequence of (a) and (b) since $u \equiv 0$ $v$-a.e.~in $B$ is equivalent to $u \equiv 0$ $\mu$-a.e.~in $B$ whenever $0<v<\infty$ $\mu$-a.e.
\end{proof}

\begin{lemma}\label{lem:NN}
	Fix $(\Sigma, \mu)$, a non-atomic $\sigma$-finite measure space with $\mu(\Sigma)>0$, and a basis $\B$. Let $\Phi$ be a Young function, and let $u$ and $v$ be $\B$-weights so that $v(B)<\infty$ for every $B\in\B$.  Then 
	\begin{list}{\textup{(\theenumi)}}{\usecounter{enumi}\leftmargin=1cm \labelwidth=1cm \itemsep=0.2cm 
			\topsep=.2cm \renewcommand{\theenumi}{\alph{enumi}}}
		\item If  \eqref{eq:PhiN1} holds, then $\mathcal{N}_1 \ge 1$. 
		
		\item If  \eqref{eq:PhiN2} holds, then $\mathcal{N}_2 \ge 1$. 
	\end{list}
\end{lemma}

\begin{proof}
	Let $B \in \B$. Observe that $M_{\B}(\mathbf{1}_B)(x) \equiv 1$ for every $x \in B$.  This and 	 \eqref{eq:PhiN1} readily give
	\begin{multline}\label{eq:geq1-1}
		\int_{B} \Phi(u) v\,d\mu
		=
		\int_{B}\Phi(M_{\B}(\mathbf{1}_B)\,u)\, v\,d\mu 
		\le
		\int_{\Sigma}\Phi(M_{\B}(\mathbf{1}_B)\,u)\, v\,d\mu 
		\\
		\le
		\mathcal{N}_1\, \int_{\Sigma}\Phi(\mathbf{1}_B\,u)\, v\,d\mu 
		=
		\int_{B} \Phi(u) v\,d\mu,
	\end{multline}
	thus $\mathcal{N}_1 \ge 1$. This shows (a). To prove (b), we observe that \eqref{eq:PhiN2} can be written as  
	\begin{align}\label{eq:geq1-3}
		\int_{\Sigma} \overline{\Phi}((M_{\B}h)\, u^{-1}\,v^{-1})\, v\, d\mu 
		&
		\le 
		\mathcal{N}_2 
		\int_{\Sigma} \overline{\Phi}(|h|\, u^{-1}\,v^{-1})\,v\, d\mu,\qquad\forall\,h \in \M_{\mu}. 
	\end{align}
	Besides, the fact that $u$ and $v$ are $\B$-weights gives that $u^{-1}\, v^{-1}$ are also $\B$-weights. Therefore, the conclusion (a) applied to \eqref{eq:geq1-3} eventually yields $ \mathcal{N}_2 \geq 1$.  
\end{proof}

\begin{example}\label{ex:Lp0:mod}
	Suppose that $(\Sigma, \mu)$ is a non-atomic $\sigma$-finite measure space with $\mu(\Sigma)>0$. Let $\B$ be a Muckenhoupt basis and let $u$, $v$ be two $\B$-weights. Given $1<p<\infty$, let $\Phi(t)=t^p$, $t\ge 0$, which is a Young function whose complementary function is $\overline{\Phi}(t)=(p-1)\,t^{p'}$, $t\ge 0$ . Clearly, $h_\Phi(t)=t^p$, $t\ge 0$, hence $i_\Phi=I_\Phi=p$. Note that \eqref{eq:PhiN1} and \eqref{eq:PhiN2} can be rewritten respectively as 
	\begin{align}	\label{eq:exLp-1:mod} 
		\int_{\Sigma} (M_{\B}h)^p\, u^p\, v\, d\mu 
		&\le 
		\mathcal{N}_1\, \int_{\Sigma} |h|^p\, u^p\,v\, d\mu,\qquad\forall\,h \in \M_{\mu}, 
		\\
		\label{eq:exLp-2:mod} 
		\int_{\Sigma} (M_{\B} h)^{p'}\, u^{-p'}\, v^{1-p'}\, d\mu 
		&\le 
		\mathcal{N}_2\, \int_{\Sigma} |h|^{p'}\, u^{-p'}\, v^{1-p'}\, d\mu ,\qquad\forall\,h \in \M_{\mu}. 
	\end{align}
	Since $\B$ is a Muckenhoupt basis, we have that $u^p\,v \in A_{p, \B}$ yields \eqref{eq:exLp-1:mod}, and $u^{-p'}v^{1-p'} \in A_{p', \B}$ implies \eqref{eq:exLp-2:mod}. On the other hand, by definition $u^p\,v \in A_{p, \B}$ if and only if  $u^{-p'}v^{1-p'} \in A_{p', \B}$. All these show that in this scenario, \eqref{eq:PhiN1} and \eqref{eq:PhiN2} (or, \eqref{eq:exLp-1:mod} and \eqref{eq:exLp-2:mod}) hold  provided $u^p\,v \in A_{p, \B}$. 
	
	As mentioned above, in $(\Rn,\Ln)$, the set of all cubes $\mathcal{Q}$,  the set of all dyadic cubes $\mathcal{D}$, the set of all rectangles $\mathcal{R}$  whose sides are parallel to the coordinate axes, and the collection $\mathcal{Z}$ of Zygmund rectangles in $(\re^3, \L^3)$ whose sides are parallel to the coordinate axes and have lengths $s$, $t$ and $st$ with $s,t>0$, are all Muckenhoupt bases. The same occurs if we consider the collection of all metric balls in  a non-atomic doubling measure metric space $(\Sigma,d,\mu)$. In each of these cases, \eqref{eq:exLp-1:mod} and \eqref{eq:exLp-2:mod} hold  provided $u^p\,v \in A_{p, \B}$. 
	
\end{example}

\begin{example}\label{ex:Phi:mod}
	Suppose that $(\Sigma, \mu)$ is a non-atomic $\sigma$-finite measure space with $\mu(\Sigma)>0$. Let $\B$ be a Muckenhoupt basis so that $A_{\infty, \B}\subset\bigcup_{s>1} RH_{s, \B}$,  and let $\Phi$ be a Young function. By interpolation for modular spaces, see \cite{Miy95,CH}, 
	\begin{align}\label{eq:w=1:mod}
		1<i_\Phi \leq I_\Phi<\infty\quad\Longrightarrow\quad \rho^{\Phi}(M_\B f)\le \mathcal{N}_1\,\rho^{\Phi}(|f|),\ \rho^{\overline{\Phi}}(M_\B f)\le \mathcal{N}_2\,\rho^{\overline{\Phi}}(|f|). 
	\end{align}
	
	On the other hand, we have already observed in Example~\ref{ex:Lp0:mod} that $u^p\,v \in A_{p, \B}$ implies \eqref{eq:exLp-1:mod} and \eqref{eq:exLp-2:mod}. We will use this and interpolation for modular spaces, see \cite{Miy95,CH}, to obtain \eqref{eq:PhiN1} . Assume now that $	1<i_\Phi \leq I_\Phi<\infty$. Let $v$ be a $\B$-weight so that $v \in A_{\infty, \B}$. We claim that
	\begin{equation}\label{eq:RIBFS-M:u-v:mod}
		\text{\eqref{eq:PhiN1} holds whenever $u^{i_\Phi}\, v \in A_{i_\Phi, \B}$ and $u^{I_\Phi}\, v \in A_{I_\Phi, \B}$}.
	\end{equation}
	To see this, fix  $u$ so that $u^{i_\Phi}\, v \in A_{i_\Phi, \B}$ and $u^{I_\Phi}\, v \in A_{I_\Phi, \B}$. By Lemma~\ref{lem:ApRH} parts \eqref{list:open-2} and \eqref{list:open-3}, there exists $r\in (1,p)$ such that $u^{i_\Phi/r} v \in A_{i_\Phi/r, \B}$ and $u^{I_\Phi\,r} v \in A_{I_\Phi\,r, \B}$. Hence,  since $\B$ is a Muckenhoupt basis, we arrive at
	\begin{align}
		\label{eq:RIBFS-1:mod} \|(M_{\B}h)\,u\|_{L^{i_\Phi/r}(\Sigma, v)}
		&\leq C \|h u\|_{L^{i_\Phi/r}(\Sigma, v)}, \quad \forall\,h \in \M_{\mu},
		\\
		\label{eq:RIBFS-2_:mod} \|(M_{\B}h)\,u\|_{L^{I_\Phi\,r}(\Sigma, \, v)} 
		&\leq C \|h u\|_{L^{I_\Phi\,r}(\Sigma, \, v)}, \quad \forall\,h \in \M_{\mu}.
	\end{align} 
	Equivalently, the sublinear operator $h\mapsto M_{\B}(h\, u^{-1})\,u$ is bounded both on $L^{i_\Phi/r}(\Sigma, v)$ and on $L^{I_\Phi\,r}(\Sigma, v)$. 
	Then, interpolation for modular spaces, see \cite{Miy95,CH},  much as in \cite[Lemma~4.20]{CMP}, gives as desired \eqref{eq:RIBFS-M:u-v:mod}.
	
	Our next claim is that 
	\begin{equation}\label{eq:RIBFS-M:u-v:dual:mod} 
		\int_{\Sigma} \overline{\Phi}((M_{\B}h)\, (u\,v)^{-1}\, v\, d\mu 
		\leq C\, \int_{\Sigma} \overline{\Phi}(|h|\, (u\,v)^{-1})\, v\, d\mu, 
		\quad \forall\,h \in \M_{\mu}, 
	\end{equation}
	whenever $u^{i_\Phi}\, v \in A_{i_\Phi, \B}$ and $u^{I_\Phi}\, v \in A_{I_\Phi, \B}$. To show this we note that if we set $\widetilde{u}=(u\,v)^{-1}$, then \eqref{eq:RIBFS-M:u-v:mod} gives that
	\begin{equation}\label{eq:RIBFS-M:u-v:dual:alt:mod} 
		\int_{\Sigma} \overline{\Phi}((M_{\B}h)\, \widetilde{u})\, v\, d\mu 
		\leq C\, \int_{\Sigma} \overline{\Phi}(|h|\, \widetilde{u})\, v\, d\mu ,  \quad 
		\text{if  $\widetilde{u}^{i_{\overline{\Phi}}}\, v \in A_{i_{\overline{\Phi}}, \B}$,\, $\widetilde{u}^{I_{\overline{\Phi}}}\, v \in A_{I_{\overline{\Phi}}, \B}$}.
	\end{equation}
	On the other hand, by \eqref{dilation:duality} and Lemma~\ref{lemma:AP-trivial} we have 
	\[
	\widetilde{u}^{i_{\overline{\Phi}}}\, v \in A_{i_{\overline{\Phi}}, \B}
	\iff 
	u^{-(I_\Phi)'}\,v^{1-(I_\Phi)'}\in A_{(I_\Phi)', \B}
	\iff 
	u^{I_\Phi}\,v\in A_{I_\Phi, \B}
	\]
	and
	\[
	\widetilde{u}^{I_{\overline{\Phi}}}\, v \in A_{I_{\overline{\Phi}}, \B}
	\iff 
	u^{-(i_\Phi)'}\,v^{1-(i_\Phi)'}\in A_{(i_\Phi)', \B}
	\iff 
	u^{i_\Phi}\,v\in A_{i_\Phi, \B}.
	\]
	All these eventually yield \eqref{eq:RIBFS-M:u-v:dual:mod}. Combining \eqref{eq:RIBFS-M:u-v:mod} and \eqref{eq:RIBFS-M:u-v:dual:mod},  we conclude that if $\Phi$ is a Young function with dilation indices satisfying $1<i_\Phi\le I_\Phi <\infty$, then
	\begin{equation}\label{eq:intro-CMP-ri:mod}
		\text{\eqref{eq:PhiN1} and \eqref{eq:PhiN2} hold whenever $u^{i_\Phi}\, v \in A_{i_\Phi, \B}$, $u^{I_\Phi}\, v \in A_{I_\Phi, \B}$, and $v \in A_{\infty, \B}$}.
	\end{equation}
	Before considering some particular examples we note that when $v\equiv 1$ one can easily see that  $u^{i_\Phi} \in A_{i_\Phi, \B}$ and $ u^{I_\Phi}\in A_{I_\Phi, \B}$ if and only if $u\in A_{i_\Phi,I_\Phi}$ (cf.~\eqref{Apq}).

	Consider the Young function $\Phi(t)=t^p\,(\log(e+t))^\alpha$, $1<p<\infty$, $\alpha\in\re$ whose complimentary function is  $\overline{\Phi}(t)\approx t^{p'}\,(\log(e+t))^{\alpha\,(1-p')}$. The dilation indices are $i_{\Phi}=I_{\Phi}=p\in (1,\infty)$,  Hence, \eqref{eq:intro-CMP-ri:mod} becomes
	\begin{equation}\label{eq:intro-CMP-ri:LplogL:mod}
		\text{\eqref{eq:PhiN1} and \eqref{eq:PhiN2} hold whenever
			$u^{p}\, v \in A_{p, \B}$, $v \in A_{\infty, \B}$, $1<p<\infty$}.
	\end{equation}
	Note that when $\alpha=0$ this agrees with Example~\ref{ex:Lp0:mod}.

	Our last examples are $\Phi(t)\approx \min\{t^4,t^6\}$ or $\Phi(t)\approx\max\{t^4,t^6\}$ whose complementary functions are respectively $\overline{\Phi}(t)\approx \max\{t^{4/3},t^{6/5}\}$ and $\overline{\Phi}(t)\approx \min\{t^{4/3},t^{6/5}\}$. In either case $i_{\Phi}=4$ and $I_{\Phi}=6$, hence we can rewrite \eqref{eq:intro-CMP-ri:mod} as
	\begin{equation}\label{eq:intro-CMP-ri:L2L4:mod}
		\text{\eqref{eq:PhiN1} and \eqref{eq:PhiN2} hold whenever $u^{4}\, v \in A_{4, \B}$, $u^{6}\, v \in A_{6, \B}$, and $v \in A_{\infty, \B}$,}
	\end{equation}
	Of course the same can be done with $\Phi(t)\approx \min\{t^p,t^q\}$ or $\Phi(t)\approx\max\{t^p,t^q\}$ with $1<p, q<\infty$ in which case 
	$i_{\Phi}=\min\{p,q\}$ and $I_{\Phi}=\max\{p,q\}$. Further details are left to the interested reader.
	
	As mentioned above, in $(\Rn,\Ln)$, the set of all cubes $\mathcal{Q}$,  the set of all dyadic cubes $\mathcal{D}$, the set of all rectangles $\mathcal{R}$  whose sides are parallel to the coordinate axes, and the collection $\mathcal{Z}$ of Zygmund rectangles in $(\re^3, \L^3)$ whose sides are parallel to the coordinate axes and have lengths $s$, $t$ and $st$ with $s,t>0$, are all Muckenhoupt bases. The same occurs if we consider the collection of all metric balls in  a non-atomic doubling measure metric space $(\Sigma,d,\mu)$. In all the cases one has $A_{\infty, \B}\subset\bigcup_{s>1} RH_{s, \B}$. 
\end{example}

\begin{lemma}\label{lem:Phh}
	Fix $(\Sigma, \mu)$, a non-atomic $\sigma$-finite measure space with $\mu(\Sigma)>0$, and a basis $\B$. Let $\Phi$ be a Young function, and let $u$ and $v$ be $\B$-weights so that  $v(B)<\infty$ for every $B\in\B$ and $v(\Sigma)>0$. Then for every $\varepsilon \in (0, 1)$ and for every non-negative function $h$ on $\Sigma$ with $0<\rho_v^{\Phi}(h\,u)<\infty$, there exists a function $h_{\varepsilon}$ on $\Sigma$ such that 
	\begin{align}\label{eq:Phhu}
		h \le \frac{h_{\varepsilon}}{1-\varepsilon},\quad h_{\varepsilon} > 0 \ \mu \text{-a.e.,} 
		\quad\text{and}\quad \rho_v^{\Phi}(h_{\varepsilon}\,u) \le  \rho_v^{\Phi}(h\,u) 
	\end{align} 
\end{lemma}

\begin{proof}
We modify the proof of Lemma~\ref{lem:hh} as follows. Use that $(\Sigma, \mu)$ is $\sigma$-finite to find an increasing sequence of $\mu$-measurable sets $\{\Sigma_j\}_{j=1}^{\infty}$ such that $\Sigma=\bigcup_{j=1}^{\infty} \Sigma_j$ and $0<\mu(\Sigma_j)<\infty$ for each $j$. Consider $E_0$ and the increasing family  $\{E_j\}_{j\ge 1}$ introduced in \eqref{def:E0-Ej} so that $\Sigma=\bigcup_{j=0}^{\infty} E_j$. Recall that there exists $j_0\ge 1$ such that $\mu(E_{j_0})>0$ and $\mu(E_0')=0$ with $E_0'= \Sigma\setminus \bigcup_{j=j_0}^\infty E_j$. We claim that 
\begin{align}\label{eq:Pinf}
	0 < \rho_v^{\Phi}(\mathbf{1}_{E_j}u) < \infty, \quad \forall j \ge j_0. 
\end{align}
In fact, $\rho_v^{\Phi}(\mathbf{1}_{E_j}\,u) \le j\, \Phi(j)\, \mu(\Sigma_j)<\infty$ and the lower bound is obvious since $u,v>0$ in $E_j$ and $\mu(E_j)>0$. 
Let $\theta:=\min\{1, \rho_v^{\Phi}(h\,u)\}\in(0,1]$, and for every $\varepsilon>0$  set
	\begin{align*}
	h_{\varepsilon} := (1-\varepsilon)\,h + \varepsilon\,\theta\, F \quad\text{and}\quad 
	F :=\mathbf{1}_{E'_0} + \sum_{j=j_0}^{\infty} 2^{-j}\, \frac{\mathbf{1}_{E_j}}{1+\rho_v^{\Phi}(\mathbf{1}_{E_j}\,u)}. 
\end{align*}
Note that by \eqref{eq:Pinf},  $h_{\varepsilon}$ and $F$ are well defined in $\Sigma$. By construction $(1-\varepsilon)\,h \le h_{\varepsilon}$. Given $x \in \Sigma \setminus E'_0=\bigcup_{j=j_0}^{\infty} E_j$,  there exists $j_x \ge j_0$ such that $x \in E_{j_x}$. From this and \eqref{eq:Pinf}, we conclude that 
\begin{align*}
	F(x) \ge 2^{-j_x} \frac{\mathbf{1}_{E_{j_x}}(x)}{1+\rho_v^{\Phi}(\mathbf{1}_{E_{j_x}}\,u)} 
	= \frac{2^{-j_x}}{1+\rho_v^{\Phi}(\mathbf{1}_{E_{j_x}}u)} >0. 
\end{align*}
This and the  fact that $\mu(E_0')=0$ readily implies that $h_{\varepsilon}\ge \varepsilon\,\theta\,F>0$ $\mu$-a.e. On the other hand, since $\mu(E_0')=0$ and $\Phi$ is convex, one has 
\begin{align*}
	\rho_v^{\Phi}(h_{\varepsilon}\,u)
	&\le (1-\varepsilon)\, \rho_v^{\Phi}(h\,u) + \varepsilon\,\theta\, \rho_v^{\Phi}(F\,u)
	\\
	&=(1-\varepsilon) \rho_v^{\Phi}(h\,u) + \varepsilon\,\theta\,  \rho_v^{\Phi}(F\,u\, \mathbf{1}_{E'_0}) 
	+ \varepsilon\,\theta\,  \rho_v^{\Phi}(F\,u\, \mathbf{1}_{\Sigma \setminus E'_0})
	\\
	&\le (1-\varepsilon) \rho_v^{\Phi}(h\,u) + \varepsilon\,\theta\,  \sum_{j=j_0}^{\infty} 2^{-j} 
	\frac{ \rho_v^{\Phi}(\mathbf{1}_{E_j}\,u)}{1+\rho_v^{\Phi}(\mathbf{1}_{E_j}\,u)}
	\\
	&\le
	(1-\varepsilon) \rho_v^{\Phi}(h\,u) + \varepsilon\,\theta
	\\
		&\le
	\rho_v^{\Phi}(h\,u). 
\end{align*}
This shows \eqref{eq:Phhu} and the proof is complete. 
\end{proof}

\section{Extrapolation on weighted Banach function spaces}\label{sec:BFSextr}

This section is devoted to establishing a variety of  extrapolation theorems on the general weighted Banach function spaces introduced above. We begin with the so-called $A_p$ extrapolation. Hereafter, a family of extrapolation pairs $\F$ is a collection of pairs $(f, g)$ of nonnegative measurable functions.
\begin{theorem}\label{thm:BFSAp}
Suppose that $(\Sigma, \mu)$ is a non-atomic $\sigma$-finite measure space with $\mu(\Sigma)>0$. Let $\B$ be a basis and let $\F$ be a family of extrapolation pairs. Let $u$ and $v$ be $\B$-weights on $(\Sigma, \mu)$ such that $v(B)<\infty$ for every $B\in\B$, and let $\X_v$ be a Banach function space over $(\Sigma, v)$. Let $M_{\B}$ denote the Hardy-Littlewood maximal function on $(\Sigma,\mu)$ associated with $\B$ and let $M'_{\B, v}h:=M_{\B}(h\,v)/v$ for each $h \in\M_{\mu}$. Then, the following hold:	
\begin{list}{\textup{(\theenumi)}}{\usecounter{enumi}\leftmargin=1cm \labelwidth=1cm \itemsep=0.2cm \topsep=.2cm \renewcommand{\theenumi}{\alph{enumi}}}
	
	\item Let $p_0\in (1,\infty)$, and assume that there are $\mathcal{N}_1,\mathcal{N}_2<\infty$ so that 
	\begin{align}
	\label{eqn:thm:BFSAp:M} \|(M_{\B}h)\,u\|_{\X_v} &\leq \mathcal{N}_1 \|h\,u\|_{\X_v}, \quad \forall\,h \in \M_{\mu}, 
	\\
	\label{eqn:thm:BFSAp:Mprime} \|(M'_{\B, v}h)\,u^{-1} \|_{\X'_v} &\leq \mathcal{N}_2 \|h\,u^{-1}\|_{\X'_v}, \quad \forall\,h \in\M_{\mu}.
	\end{align}
	If for every $w \in A_{p_0, \B}$, one has
	\begin{equation}\label{eq:uXv-3}
	\|f\|_{L^{p_0}(\Sigma,w)} \leq \Psi([w]_{A_{p_0, \B}})\|g\|_{L^{p_0}(\Sigma,w)}, \qquad(f,g) \in \F, 
	\end{equation}  
	where $\Psi: [1, \infty) \to [1, \infty)$ is a non-decreasing function, then 
	\begin{equation}\label{eq:uXv-4}
	\|f\,u\|_{\X_v} \leq C_0\, \|g\,u\|_{\X_v}, \qquad(f,g) \in \F,
	\end{equation}
	and 
	\begin{equation}\label{eq:uXv-54-alt}
	\bigg\|\Big(\sum_j f_j^{p_0}\Big)^{\frac1{p_0}}u\bigg\|_{\X_v} 
	\leq C_0 \bigg\|\Big(\sum_j g_j^{p_0}\Big)^{\frac1{p_0}}u\bigg\|_{\X_v}, \quad \{(f_j, g_j)\}_j \subset \F, 
	\end{equation}
	where $C_0=2^{3+\frac{2}{p'_0}}\,\Psi(2^{p_0}\,\mathcal{N}_1^{p_0-1}\, \mathcal{N}_2)$.

	\item 	Assume that there is $\mathcal{N}_2<\infty$ so that \eqref{eqn:thm:BFSAp:Mprime} holds. If for every $w \in A_{1, \B}$, one has
	\begin{equation}\label{eq:uXv-3:1}
	\|f\|_{L^{1}(\Sigma,w)} \leq \Psi([w]_{A_{1, \B}})\|g\|_{L^{1}(\Sigma,w)}, \qquad(f,g) \in \F, 
	\end{equation}  
	where $\Psi: [1, \infty) \to [1, \infty)$ is a non-decreasing function, then 
	\begin{equation}\label{eq:uXv-4:1}
	\|f\,u\|_{\X_v} \leq 8\, \Psi(2\,\mathcal{N}_2)\, \|g\,u\|_{\X_v}, \qquad(f,g) \in \F. 
	\end{equation}
	and 
	\begin{equation}\label{eq:uXv-54-alt:1}
	\bigg\|\Big(\sum_j f_j\Big)\, u\bigg\|_{\X_v} 
	\leq 
	8\, \Psi(2\,\mathcal{N}_2)\, \bigg\|\Big(\sum_j g_j\Big)\,u\bigg\|_{\X_v}, \qquad \{(f_j, g_j)\}_j \subset \F. 
	\end{equation}
\end{list}
Moreover, in any of the two scenarios if $\B$ is additionally assumed to be a Muckenhoupt basis, then \eqref{eq:uXv-3} or \eqref{eq:uXv-3:1} imply that for every $q \in (1, \infty)$, 
\begin{equation}\label{eq:uXv-5}
\bigg\|\Big(\sum_j f_j^q\Big)^{\frac1q} u\bigg\|_{\X_v} 
\leq C\, \bigg\|\Big(\sum_j g_j^q\Big)^{\frac1q}u\bigg\|_{\X_v}, \qquad \{(f_j, g_j)\}_j \subset \F. 
\end{equation}
\end{theorem}

\begin{remark}
Theorem \ref{thm:BFSAp} generalizes \cite[Theorems~3.9 and 4.10]{CMP}, \cite[Theorem~2.6]{CW}, and \cite[Theorem 8.2]{MMMMM20}. Indeed, if $\X_v=L^p(\Rn, v)$ with $1<p<\infty$, then both \eqref{eqn:thm:BFSAp:M} and \eqref{eqn:thm:BFSAp:Mprime} are equivalent to $u^p v \in A_{p, \B}$. On the other hand, if $\X_v$ is a rearrangement invariant Banach function space with $1<p_{\X}\le q_{\X}<\infty$ and $u \equiv 1$, then $v \in A_{p_{\X}, \B}$ implies both  \eqref{eqn:thm:BFSAp:M} and \eqref{eqn:thm:BFSAp:Mprime} hold (cf. \cite[Lemma~4.12]{CMP}) provided that $\B$ is open  for the Muckenhoupt classes. Finally, if we take $\X_v=L^{p(\cdot)}(\Rn,\Ln)$, then \eqref{eqn:thm:BFSAp:M} and \eqref{eqn:thm:BFSAp:Mprime} agree with the condition that $(p(\cdot), u)$ is an $M$-pair (cf. \cite[p.~1209]{CW}).  
\end{remark} 

\begin{remark}\label{remark:rescale}
	As in \cite[Section 3.3]{CMP} one can easily rescale in the previous result. To be more precise, suppose that for some $r>0$ and $p_0\in[r,\infty]$ there holds
	\begin{equation}\label{eq:uXv-3:rescale}
	\|f\|_{L^{p_0}(\Sigma,w)} \leq \Psi([w]_{A_{p_0/r, \B}})\|g\|_{L^{p_0}(\Sigma,w)}, \qquad(f,g) \in \F, 
	\end{equation}  
	for all $w\in A_{p_0/r,\B}$, and where $\Psi: [1, \infty) \to [1, \infty)$ is a non-decreasing function. Then, one can rewrite \eqref{eq:uXv-3:rescale} as
	\begin{equation}\label{eq:uXv-3:rescale:alt}
	\|f^r\|_{L^{p_0/r}(\Sigma,w)} \leq \Psi([w]_{A_{p_0/r, \B}})\|g^r\|_{L^{p_0/r}(\Sigma,w)}, \qquad(f,g) \in \F.
	\end{equation}  
	Thus we can apply Theorem~\ref{thm:BFSAp} to the previous expression (that is, to the family $\F_r$ of pairs $(f^r,g^r)$ with $(f,g)\in\F$) to easily obtain, using the notation  introduced in \eqref{Xr}, 
	\begin{equation}\label{eq:uXv-4:rescale}
	\|f\,u\|_{\X_v^r} \lesssim \|g\,u\|_{\X_v^r}, \qquad(f,g) \in \F. 
	\end{equation}
	provided
	\begin{align}
	\label{thm:BFSAp:M:rescale} \norm{(M_{\B}h)\,u^r}_{\X_v} &\lesssim \|h\,u^r\|_{\X_v}, \quad \forall\,h \in\M_{\mu}, 
	\\
	\label{thm:BFSAp:Mprime:rescale} \|(M'_{\B, v}h)\,u^{-r} \|_{\X'_v} &\lesssim \|h\,u^{-r}\|_{\X'_v}, \quad \forall\,h \in\M_{\mu}.
	\end{align}
	when $r<p_0$ and assuming only \eqref{thm:BFSAp:Mprime:rescale} when $r=p_0$. Further details are left to the interested reader. 	
\end{remark}

The proof of the previous result will be based on the following proposition which is interesting on its own right.
\begin{proposition}\label{prop:extrapol-maintool}
	Suppose that $(\Sigma, \mu)$ is a non-atomic $\sigma$-finite measure space with $\mu(\Sigma)>0$. Let $\B$ be a basis,  $u$ and $v$ be $\B$-weights on $(\Sigma, \mu)$ such that $v(B)<\infty$ for every $B\in\B$, and let $\X_v$ be a Banach function space over $(\Sigma, v)$. Let $M_{\B}$ denote the Hardy-Littlewood maximal function on $(\Sigma,\mu)$ associated with $\B$ and let $M'_{\B, v}h := M_{\B}(h\,v)/v$ for each $h \in\M_{\mu}$. 
	Then, for every $f,g\in \M_\mu$ so that $\|f\,u\|_{\X_v}, \|g\,u\|_{\X_v}<\infty$ the following hold:	
	\begin{list}{\textup{(\theenumi)}}{\usecounter{enumi}\leftmargin=1cm \labelwidth=1cm \itemsep=0.2cm \topsep=.2cm \renewcommand{\theenumi}{\alph{enumi}}}

	\item Let $p_0\in (1,\infty)$, and assume that there are $\mathcal{N}_1,\mathcal{N}_2<\infty$ so that both \eqref{eqn:thm:BFSAp:M} and \eqref{eqn:thm:BFSAp:Mprime} hold. Then, there exists a $\B$-weight $w=w(f,g)\in A_{p_0,\B}$ satisfying $[w]_{A_{p_0,\B}}\le 2^{p_0}\,\mathcal{N}_1^{p_0-1} \,\mathcal{N}_2$ and such that
	\begin{equation}\label{eq:embedding:f-g}
	\|f\,u\|_{\X_v} \leq 2^{1+\frac{4}{p'_0}}\,\|f\|_{L^{p_0}(\Sigma,w)}  
	\quad\text{ and }\quad 
	\|g\|_{L^{p_0}(\Sigma,w)}\le  2^{\frac2{p_0}}\,\|g\,u\|_{\X_v}. 	
	\end{equation}
	In particular, there exists a $\B$-weight $w=w(f)\in A_{p_0,\B}$ such that $[w]_{A_{p_0,\B}}\le 2^{p_0}\,\mathcal{N}_1^{p_0-1} \,\mathcal{N}_2$ and 
	\begin{equation}\label{eq:embedding:f}
	2^{-1-\frac{4}{p'_0}}\, \|f\,u\|_{\X_v}\leq  \|f\|_{L^{p_0}(\Sigma,w)}  \le 2^{\frac2{p_0}}\, \|f\,u\|_{\X_v}. 	
	\end{equation}
	
	\item Assume that there is $\mathcal{N}_2<\infty$ so that \eqref{eqn:thm:BFSAp:M} holds. Then, there exists a Muckenhoupt weight $w=w(f,g)\in A_{1,\B}$ satisfying $[w]_{A_{p_0,\B}}\le 2\,\mathcal{N}_2$ and such that
	\begin{equation}\label{eq:embedding:f-g:1}
	\|f\,u\|_{\X_v} \leq 2\,\|f\|_{L^{1}(\Sigma,w)}  
	\qquad\text{and}\qquad 
	\|g\|_{L^{1}(\Sigma,w)}\le  4\,\|g\,u\|_{\X_v}. 	
	\end{equation}
	In particular, there exists a $\B$-weight $w=w(f)\in A_{1,\B}$ satisfying $[w]_{A_{1,\B}}\le 2\,\mathcal{N}_2$ and such that
	\begin{equation}\label{eq:embedding:f:1}
	2^{-1}\, \|f\,u\|_{\X_v}\leq  \|f\|_{L^{1}(\Sigma,w)}  \le 4\, \|f\,u\|_{\X_v}. 	
	\end{equation}
		\end{list}
	\end{proposition}

Assuming this result momentarily we can easily prove Theorem \ref{thm:BFSAp}:

\begin{proof}[Proof of Theorem~\ref{thm:BFSAp}]
Fix $(f, g) \in \F$, and without loss of generality, we may assume that $\|g\,u\|_{\X_v}<\infty$, otherwise there is nothing to prove. We claim that  $f<\infty$ $\mu$-a.e. Otherwise, there exists a measurable set $E \subset \Sigma$ with $\mu(E)>0$ such that $f=\infty$ on $E$. In view of \eqref{eq:uXv-3} and the fact that $\B$-weights are $\mu$-a.e positive on $\Sigma$, it follows that 
\begin{align}\label{eq:gAi}
\|g\|_{L^{p_0}(\Sigma, w)}=\infty \quad\text{ for every } w \in A_{p_0, \B},
\end{align} 
and this clearly contradicts \eqref{eq:embedding:f} in Proposition~\ref{prop:extrapol-maintool} applied to $g$.

To proceed we recall that $(\Sigma, \mu)$ is $\sigma$-finite, hence there exists an increasing sequence of $\mu$-measurable sets  $\{\Sigma_j\}_{j=1}^{\infty}$ such that $\Sigma=\bigcup_{j=1}^{\infty} \Sigma_j$ and $\mu(\Sigma_j)<\infty$ for each $j$. For every $N \ge 1$, we define 
\[
f_N = f\, \mathbf{1}_{\Sigma'_N} := f\, \mathbf{1}_{\{x\in \Sigma_N: f(x)\leq N, u(x) \le N, v(x) \le N\}}.
\] 
Note that $v(\Sigma'_N) \le N \mu(\Sigma_N)<\infty$. Since $\X_v$ is a Banach function space over $(\Sigma, v)$, thanks to the property (vi) in Definition \ref{def:BFS}, one has  
\begin{align*}
\|f_N\, u\|_{\X_v} \le N^2 \|\mathbf{1}_{\Sigma'_N}\|_{\X_v}<\infty. 
\end{align*}
Consider first the case $p_0>1$. Then applying Proposition~\ref{prop:extrapol-maintool} part (a), we can find a weight $w=w(f_N, g) \in A_{p_0, \B}$ with $[w]_{A_{p_0, \B}} \le 2^{p_0}\mathcal{N}_1^{p_0-1} \mathcal{N}_2$ such that 
\begin{equation}\label{eq:fNu}
\|f_N\,u\|_{\X_v} \le 2^{1+\frac{4}{p'_0}} \|f_N\|_{L^{p_0}(\Sigma, w)} \quad\text{and}\quad 
\|g\|_{L^{p_0}(\Sigma, w)} \le 2^{\frac{2}{p_0}} \|g\,u\|_{\X_v}.  
\end{equation}
This, together with \eqref{eq:uXv-3}, yields 
\begin{multline}\label{eq:fNugu}
\|f_Nu\|_{\X_v} \le 2^{1+\frac{4}{p'_0}} \|f_N\|_{L^{p_0}(\Sigma, w)} 
\le 2^{1+\frac{4}{p'_0}} \|f\|_{L^{p_0}(\Sigma, w)} 
\\
\le 2^{1+\frac{4}{p'_0}} \Psi([w]_{A_{p_0, \B}}) \|g\|_{L^{p_0}(\Sigma, w)} 
\le 2^{3+\frac{2}{p'_0}} \Psi(2^{p_0}\mathcal{N}_1^{p_0} \mathcal{N}_2) \|g\,u\|_{\X_v}.  
\end{multline}
Recalling that $f<\infty$ $\mu$-a.e.~and that $u, v$ are $\B$-weights we have $f_N\,u \nearrow f\,u$ as $N \to \infty$ $\mu$-a.e., hence also $v$-a.e., since $v$ is a $\B$-weight. Therefore, \eqref{eq:fNugu} and property (v) in Definition \ref{def:BFS} immediately give \eqref{eq:uXv-4}, as desired. The case $p_0=1$ follows using the same argument but invoking Proposition~\ref{prop:extrapol-maintool} part (b), details are left to the interested reader.

To complete the proof we need to justify the vector-valued inequalities \eqref{eq:uXv-54-alt}, \eqref{eq:uXv-54-alt:1}, and \eqref{eq:uXv-5}. For any $q \in (1, \infty)$, introduce the new family $\F_{\ell^q}$ of extrapolation pairs $(F, G)$, where 
\begin{align}\label{def:F-vv}
\F_{\ell^q}:=\Big\{(F,G)=\Big(\Big(\sum_{j} f_j^q \Big)^{\frac1q},\Big(\sum_{j} g_j^q \Big)^{\frac1q}\Big): \{(f_j, g_j)\}_j \subset \F\Big\}. 
\end{align}
Note that for every $(F,G)\in\F_{\ell^{p_0}}$ and for every $w\in A_{p_0,\B}$, one easily gets from \eqref{eq:uXv-3} if $p_0>1$ or \eqref{eq:uXv-3:1} if $p_0=1$
\begin{align*}
\|F\|_{L^{p_0}(\Sigma, w)}^{p_0} 
= \sum_{j} 
\int_{\Sigma} f_j^{p_0} w\, d\mu
\le 
 \Psi([w]_{A_{p_0, \B}})^{p_0}\,
\sum_{j} \int_{\Sigma} g_j^q w\, d\mu 
= 
 \Psi([w]_{A_{p_0, \B}})^{p_0} \|G\|_{L^{p_0}(\Sigma, w)}^{p_0}. 
\end{align*}
Thus, \eqref{eq:uXv-54-alt} follows from \eqref{eq:uXv-4} when $p_0>1$ and \eqref{eq:uXv-54-alt:1} follows from \eqref{eq:uXv-4:1} when $p_0=1$, either of them applied to $\F_{\ell^{p_0}}$: for every $(F,G)\in\F_{\ell^{p_0}}$,
\begin{multline*}
\bigg\|\Big(\sum_j f_j^{p_0}\Big)^{\frac1{p_0}} u\bigg\|_{\X_v} 
=\|F\,u\|_{\X_v} \leq 2^{3+\frac{2}{p'_0}}\Psi(2^{p_0}\mathcal{N}_1^{p_0-1} \mathcal{N}_2)\, \|G\,u\|_{\X_v}
\\
=2^{3+\frac{2}{p'_0}}\Psi(2^{p_0}\mathcal{N}_1^{p_0-1} \mathcal{N}_2)\, \bigg\|\Big(\sum_j g_j^{p_0}\Big)^{\frac1{p_0}}u\bigg\|_{\X_v}.
\end{multline*}

We are left with showing \eqref{eq:uXv-5} under the additional assumption that $\B$ is a Muckenhoupt basis. The main point is that such a condition allows us to apply Rubio de Francia extrapolation theorem in the present context of Muckenhoupt basis (cf.~\cite[Theorem~3.9]{CMP}) to replace $p_0$ in \eqref{eq:uXv-3} by any $q\in (1,\infty)$. More specifically, fix $q\in (1,\infty)$ and note that \eqref{eqn:thm:BFSAp:M} and \eqref{eqn:thm:BFSAp:Mprime} hold with $u\equiv 1$, $v\in A_{q,\B}$, $\X_v=L^{q}(\Sigma, v)$, precisely because  $\B$ is a Muckenhoupt basis and hence $M_\B$ is bounded in $\X_v=L^{q}(\Sigma, v)$ and in $\X_v=L^{q'}(\Sigma, v^{1-q'})$ for every $v\in A_{q,\B}$. As such \eqref{eq:uXv-4} in this particular case implies a version of  \eqref{eq:uXv-3} with $q$ in place of $p_0$. This becomes our new initial assumption and we readily obtain 
\eqref{eq:uXv-54-alt} with $q$ in place of $p_0$, and this is eventually \eqref{eq:uXv-5}, completing the proof. 
\end{proof}

\begin{proof}[Proof of Proposition~\ref{prop:extrapol-maintool}]
	We fist observe that \eqref{eq:embedding:f-g} (resp.~\eqref{eq:embedding:f-g:1}) with $g:=f$ readily implies \eqref{eq:embedding:f} (resp.~\eqref{eq:embedding:f-g:1}). We then focus on obtaining \eqref{eq:embedding:f-g} and \eqref{eq:embedding:f-g:1}. With this goal in mind, we generalize some of the ideas of \cite[Theorem~4.10]{CMP}. Fix $f, g\in\M_\mu$ with $\|f\,u\|_{\X_v}<\infty$ and  $\|g\,u\|_{\X_v}<\infty$. We start with (a), that is, we fix $p_0>1$ and assume \eqref{eqn:thm:BFSAp:M} and \eqref{eqn:thm:BFSAp:Mprime} for some finite constants $\mathcal{N}_1$ and $\mathcal{N}_2$. For any $h_1,h_2\in\M_\mu$ with $h_1, h_2>0$ $\mu$-a.e., $h_1\, u \in \X_v$ and $h_2\, u^{-1} \in \X'_v$, we define the Rubio de Francia iteration algorithms as:
	\begin{equation*}
	\mathcal{R}h_1 :=\sum_{k=0}^{\infty}\frac{M_{\B}^k h_1}{2^k\, \mathcal{N}_1^k} \qquad\text{and}\qquad
	\mathcal{R}'h_2 :=\sum_{k=0}^{\infty}\frac{(M'_{\B, v})^k h_2}{2^k\, \mathcal{N}_2^k},
	\end{equation*} 
	where $M_{\B}^0$ and $(M'_{\B, v})^0$ denote the identity operator, while for every $k\ge 1$, we write $M_{\B}^k=M_{\B}\circ \dots\circ M_{\B}$ for the $k$-th iteration of $M_{\B}$ and	$(M'_{\B,v})^k=M'_{\B,v}\circ \dots\circ M'_{\B,v}$ for the $k$-th iteration of $M'_{\B,v}$.
	We claim that the following hold: 
	\begin{alignat}{2}
	h_1 &\leq \mathcal{R}h_1\qquad\qquad &	h_2 &\leq \mathcal{R}' h_2, 
	\label{RdF1}
	\\[4pt]
	\norm{(\mathcal{R}h_1)\,u}_{\X_v} &\leq 2\,\norm{h_1\, u}_{\X_v}, \qquad\qquad &		\norm{(\mathcal{R}'h_2)\,u^{-1}}_{\X'_v} &\leq 2\,\norm{h_2\, u^{-1}}_{\X'_v},
	\label{RdF2}
	\\[4pt]
	\left[\mathcal{R}h_1\right]_{A_{1, \B}}&\leq 2\, \mathcal{N}_1, \qquad\qquad &			[(\mathcal{R}' h_2)\, v]_{A_{1, \B}}&\leq 2\, \mathcal{N}_2.
	\label{RdF3}
	\end{alignat}
	Indeed, \eqref{RdF1} and \eqref{RdF2} are immediate consequences of the definitions, \eqref{eqn:thm:BFSAp:M} and \eqref{eqn:thm:BFSAp:Mprime}. 
	On the other hand, using that $h_1, h_2$ are positive almost everywhere, we can obtain from \eqref{RdF1}
	\begin{equation*}
	\mathcal{R}h_1(x) \ge h_1(x)>0, \quad \mu \text{-a.e.~} x \in \Sigma,
	\quad\text{and}\quad
	\mathcal{R}'h_2(x) \ge h_2(x)>0, \quad \mu \text{-a.e.~} x \in \Sigma.
	\end{equation*} 
	Besides, since by assumption $v(B)<\infty$ for every $B\in\B$ and $\X_v$ is a Banach function space over $(\Sigma, v\,d\mu)$, by the property (vi) in Definition \ref{def:BFS}, and \eqref{RdF2}, one has for every $B\in\B$
	\begin{equation*}
	\int_{B} (\mathcal{R}h_1)\,u\, v\,d\mu 
	=
	\int_{B} (\mathcal{R}h_1)\,u\,dv
	\le C_B \|(\mathcal{R}h_1)\,u\|_{\X_v} 
	\le 
	2
	C_B \|h_1 u\|_{\X_v}<\infty. 
	\end{equation*}
	This in turn implies that $(\mathcal{R}h_1)\,u\,v<\infty$ $\mu$-a.e.~in $\Sigma_\B$. This eventually shows that $\mathcal{R}h_1$ is a $\B$-weight. Likewise, $\mathcal{R}'h_2$ is also a $\B$-weight and so is $(\mathcal{R}'h_2) \, v$. Moreover,
	\[
	M_{\B}(\mathcal{R} h_1) 
	\leq 
	\sum_{k=0}^{\infty} \frac{M_{\B}^{k+1} h_1}{2^k \mathcal{N}_1^k} 
	\le  
	2\,\mathcal{N}_1\, \mathcal{R} h_1
	\quad\text{ and }\quad 	
	M'_{\B,v}(\mathcal{R}' h_2) 
	\leq 
	\sum_{k=0}^{\infty} \frac{(M'_{\B, v})^{k+1} h_2}{2^k \mathcal{N}_2^k} 
	\leq 2\,\mathcal{N}_2\, \mathcal{R}'h_2.
		\]
	These readily imply \eqref{RdF3}. 
	
	To proceed, we claim that it suffices to consider the case $\|f\,u\|_{\X_v}>0$ and $\|g\,u\|_{\X_v}>0$. Indeed, if $\|f\,u\|_{\X_v}=0$, or $\|g\,u\|_{\X_v}=0$, or both, we just use the claim with $f$, or $g$, or both, replaced by $\mathbf{1}_B$ for some fixed $B \in \B$. Note that $\|\mathbf{1}_B u\|_{\X_v}>0$ by \eqref{eq:BuXv} in Lemma~\ref{lem:AXprev}. Assume then that  $0<\|f\,u\|_{\X_v}, \|g\,u\|_{\X_v}<\infty$.  Set $h_1:= \frac{g}{\norm{gu}_{\X_v}}$ so that  $h_1\ge 0$ $\mu$-a.e.~and $\|h_1 u\|_{\X_v}=1$. In view of \eqref{eq:fX-norm}, there exists a non-negative function $h \in \X'_v$ with $\|h\|_{\X'_v} \leq 1$ such that 
	\begin{align}\label{eq:fhu}
	\|f\,u\|_{\X_v} \le 2 \int_{\Sigma} f\, h\,u\,dv= \int_{\Sigma} f\, h_2\,v\, d\mu
	\end{align}
	where $h_2:=h\,u$ satisfies $h_2\,u^{-1}\in \X'_v$ with $\|h_2\,u^{-1}\|_{\X'_v} \leq 1$. Note that since $\|f\,u\|_{\X_v}>0$ it follows that $\mu(\{h>0\})>0$, hence $h_2$ is not zero $\mu$-a.e.~in $\Sigma$.  Applying Lemma \ref{lem:hh} with $\varepsilon=1$ to $h_1$ and $h_2$, we can find two functions $\widetilde{h}_1,\widetilde{h}_2> 0$ $\mu$-a.e.~such that 
	\begin{align}\label{eq:hhh}
	h_1 \le \widetilde{h}_1, \quad h_2 \le \widetilde{h}_2,\quad 
	\|\widetilde{h}_1\,u\|_{\X_v} \le 2 \quad\text{and}\quad \|\widetilde{h}_2\,u^{-1}\|_{\X'_v} \le 2.  
	\end{align} 
	Set
	\begin{equation*}
	w:=(\mathcal{R}\widetilde{h}_1)^{1-p_0} (\mathcal{R}'\widetilde{h}_2)\, v. 
	\end{equation*}
	From \eqref{RdF3} and \eqref{eq:Ap-product}, we deduce that $w \in A_{p_0, \B}$ and 
	\begin{equation*}
	[w]_{A_{p_0}} \leq [\mathcal{R}\widetilde{h}_1]_{A_{1, \B}}^{p_0-1} [(\mathcal{R}'\widetilde{h}_2)\,v]_{A_{1, \B}}
	\leq 2^{p_0} \mathcal{N}_1^{p_0-1} \mathcal{N}_2.
	\end{equation*}
	Also, using \eqref{eq:Holder} and \eqref{RdF2} we obtain 
	\begin{multline}\label{eq:RgRh}
	\int_{\Sigma} \mathcal{R}\widetilde{h}_1\, \mathcal{R}'\widetilde{h}_2\,v\,d\mu
	=
	\int_{\Sigma} (\mathcal{R}\widetilde{h}_1)\,u\, (\mathcal{R}'\widetilde{h}_2)\,u^{-1}\,dv
	\\
	\leq \|(\mathcal{R}\widetilde{h}_1)\,u\|_{\X_v} \|(\mathcal{R}'\widetilde{h}_2)\,u^{-1}\|_{\X'_v}  
	\leq 4\,\|\widetilde{h}_1\,u\|_{\X_v}\, \|\widetilde{h}_2\,u^{-1}\|_{\X'_v} \leq 16
	\end{multline}
	and
	\begin{multline}\label{eq:RgRh*}
	\int_{\Sigma} g\, (\mathcal{R}'\widetilde{h}_2)\,v\,d\mu
	=
	\int_{\Sigma} g\,u\, (\mathcal{R}'\widetilde{h}_2)\,u^{-1}\,dv
	\\
	\leq \|g\,u\|_{\X_v}\, \|(\mathcal{R}'\widetilde{h}_2)\,u^{-1}\|_{\X'_v}  
	\leq 2\,\|g\,u\|_{\X_v}\,\|\widetilde{h}_2\,u^{-1}\|_{\X'_v} \leq 4\,\|g\,u\|_{\X_v}.
	\end{multline}
	Then, by \eqref{eq:fhu}, \eqref{eq:hhh}, \eqref{RdF1}, Hölder's inequality, and \eqref{eq:RgRh}, 
	\begin{multline*}
	2^{-1} \|f\,u\|_{\X_v} \le \int_{\Sigma} f\, \widetilde{h}_2\,v\, d\mu 
	\le 
	\int_{\Sigma} f\, (\mathcal{R}\widetilde{h}_1)^{-\frac{1}{p'_0}}\, 
	(\mathcal{R}\widetilde{h}_1)^{\frac{1}{p'_0}}\, (\mathcal{R}'\widetilde{h}_2)\, v\, d\mu
	\\[4pt]
	\leq 
	\left(\int_{\Sigma} f^{p_0}\,(\mathcal{R}\widetilde{h}_1)^{1-p_0}\, (\mathcal{R}'\widetilde{h}_2)\, v\,d\mu\right)^{\frac{1}{p_0}} 
	\left(\int_{\Sigma} (\mathcal{R}\widetilde{h}_1)\,(\mathcal{R}'(\widetilde{h}_2)\,v\,d\mu\right)^{\frac{1}{p_0'}} 
	\le 
	16^{\frac{1}{p'_0}} \|f\|_{L^{p_0}(\Sigma, w)}.
	\end{multline*}
	This shows the first inequality in \eqref{eq:embedding:f-g}. On the other hand, $g/\norm{gu}_{\X_v}=h_1\le \widetilde{h}_1\le \mathcal{R}\widetilde{h}_1$ by \eqref{eq:hhh},  \eqref{RdF1}. Thus, \eqref{eq:RgRh*} yields 
	\begin{multline*}
	\|g\|_{L^{p_0}(\Sigma,w_0)} 
	=\left(\int_{\Sigma} g^{p_0}\,(\mathcal{R}\widetilde{h}_1)^{1-p_0}\,
	(\mathcal{R}' \widetilde{h}_2)\,v\, d\mu\right)^{\frac{1}{p_0}}
	\\[4pt]
	\leq \|g\,u\|_{\X_v}^{\frac{1}{p_0'}} \left(\int_{\Sigma} g\, (\mathcal{R}'\widetilde{h}_2)\,v\,d\mu\right)^{\frac{1}{p_0}}
	\le 2^{\frac{2}{p_0}} \norm{gu}_{\X_v}.  
	\end{multline*} 
	This shows the second estimate in \eqref{eq:embedding:f-g} and completes the proof of the case $p_0>1$.
	
	Let us next deal with (b), that is, we consider the case $p_0=1$ and only assume that there exists $\mathcal{N}_2<\infty$ so that \eqref{eqn:thm:BFSAp:Mprime} holds. We follow the proof of the previous case and this time we do not use $\mathcal{R}$, that is, \eqref{eqn:thm:BFSAp:M} needs not to be assumed (hence $\mathcal{N}_1$ could be infinity).  With the notation above, let us set $w:=\mathcal{R}'(\widetilde{h}_2)\,v$. The second estimate in \eqref{RdF3} implies $w \in A_{1, \B}$ and $[w]_{A_{1, \B}} \le 2\,\mathcal{N}_2$. Additionally, \eqref{eq:fhu} and the second estimate in \eqref{eq:hhh} give
	\begin{align*}
	2^{-1} \|f\,u\|_{\X_v} 
	\le 
	\int_{\Sigma} f\, \widetilde{h}_2\,v\, d\mu 
	\le \int_{\Rn} f\, (\mathcal{R}'\widetilde{h}_2)\,v\, d\mu = \|f\|_{L^1(\Sigma, w)}. 
	\end{align*}
	On the other hand, \eqref{eq:Holder}, the second estimate in \eqref{RdF2}, and the last in \eqref{eq:hhh} readily imply
	\begin{multline*}
	\|g\|_{L^1(\Sigma, w)} 
	= 
	\int_{\Sigma} g\, (\mathcal{R}' \widetilde{h}_2)\,v\, d\mu
	=
	\int_{\Sigma} g\,u\, (\mathcal{R}' \widetilde{h}_2)\,u^{-1}\,dv
	\\
	\le 
	\|g\,u\|_{\X_v}\, \|(\mathcal{R}' \widetilde{h}_2)\,u^{-1}\|_{\X'_v} 
	\le 
	2\, \|g\,u\|_{\X_v} \|\widetilde{h}_2\,u^{-1}\|_{\X'_v} 
	\le 4\, \|g\,u\|_{\X_v}. 
	\end{multline*}
	These prove \eqref{eq:embedding:f-g:1}. 
\end{proof}

The next goal is to prove an extrapolation theorem for $A_{\infty, \B}$ weights. 

\begin{theorem}\label{thm:BFSAi}
Suppose that $(\Sigma, \mu)$ is a non-atomic $\sigma$-finite measure space with $\mu(\Sigma)>0$. Let $\B$ be a Muckenhoupt basis and let $\F$ be a family of extrapolation pairs. Let $u$ and $v$ be $\B$-weights on $(\Sigma, \mu)$ such that $v(B)<\infty$ for every $B\in\B$, and let $\X_v$ be a Banach function space over $(\Sigma, v)$. Let $M_{\B}$ denote the Hardy-Littlewood maximal function on $(\Sigma,\mu)$ associated with $\B$ and let $M'_{\B, v}h := M_{\B}(h\,v)/v$ for each $h \in \M_{\mu}$. Assume that there exists $\mathcal{N}<\infty$ such that
\begin{align}\label{eq:BFSAi-1} 
\|(M'_{\B, v} h)\, u^{-1}\|_{\X'_v} \le \mathcal{N} \|h\, u^{-1}\|_{\X'_v},\quad\forall\,h \in \M_{\mu}. 
\end{align} 
If for some $p_0 \in (0,\infty)$ and for every $w \in A_{\infty, \B}$, 
\begin{equation}\label{eq:BFSAi-2}
\|f\|_{L^{p_0}(\Sigma,w)} \leq \Psi([w]_{A_{\infty,\B}}) \|g\|_{L^{p_0}(\Sigma,w)}, \quad(f,g) \in \F, 
\end{equation} 
where $\Psi:[1,\infty) \to [1,\infty)$ is a non-decreasing function, then for every $p\in (0, \infty)$, 
\begin{equation}\label{eq:BFSAi-3}
\|f^p\,u\|_{\X_v} \leq C\, \|g^p\,u\|_{\X_v},\qquad(f,g) \in \F, 
\end{equation}
and for every $q \in (0, \infty)$, 
\begin{equation}\label{eq:BFSAi-4}
\bigg\|\Big(\sum_j f_j^q\Big)^{\frac{p}{q}} u\bigg\|_{\X_v} 
\leq C\, \bigg\|\Big(\sum_j g_j^q\Big)^{\frac{p}{q}}\bigg\|_{\X_v}, \qquad \{(f_j, g_j)\}_j \subset \F. 
\end{equation}
\end{theorem}

\begin{proof} 
The first thing that we are going to show is that \eqref{eq:BFSAi-2} holds for all $p\in (0,\infty)$ and $w\in A_{\infty,\B}$. This is an extension of \cite{CMP04}, \cite[Corollary~3.15]{CMP} to our current setting. For completeness we include the argument which follows easily from Theorem~\ref{thm:BFSAp}. To see this take an arbitrary $p\in (0,\infty)$ and $w\in A_{\infty, \B}$. The definition of $A_{\infty, \B}$ and the fact that the $A_{q,\B}$ are nested imply that we can pick $s>\max\{p/p_0,1\}$ so that $w\in A_{s,\B}$. Select then $r=p_0\,s/p\in (1,\infty)$ and note that  \eqref{eq:BFSAi-2} implies in particular that
\begin{equation}\label{qwerwege}
\|f^{p_0/r}\|_{L^{r}(\Sigma,w_0)} \leq \Psi([w]_{A_{r,\B}}) \|g^{p_0/r}\|_{L^{r}(\Sigma,w_0)}, \quad(f,g) \in \F,
\end{equation}
for every $w_0\in A_{r,\B}$. With this in hand we are going to invoke Rubio de Francia extrapolation theorem in this context. More precisely,  we first observe that \eqref{eqn:thm:BFSAp:M} and \eqref{eqn:thm:BFSAp:Mprime} hold with $u\equiv 1$, $v\in A_{r,\B}$, $\X_v=L^{r}(\Sigma, v)$, because  $\B$ is a Muckenhoupt basis and hence $M_\B$ is bounded in $\X_v=L^{r}(\Sigma, v)$ and in $\X_v=L^{r'}(\Sigma, v^{1-r'})$ for every $v\in A_{r,\B}$. Hence, Theorem~\ref{thm:BFSAp} part (a) applied in this scenario with \eqref{qwerwege} playing the role of \eqref{eq:uXv-3} (and where it may be convenient to think of the pairs $(f^{p_0/r},g^{p_0/r})$ as the elements of a new family $\F_{p_0/r}$) yields 
\begin{equation}\label{qwerwege:1}
\|f^{p_0/r}\|_{L^{q}(\Sigma,w_0)} \lesssim \|g^{p_0/r}\|_{L^{q}(\Sigma,w_0)}, \quad(f,g) \in \F,
\end{equation}
for every $w_0\in A_{q,\B}$,  since $q>1$. Taking in particular $q=s>1$, and $w_0=w \in A_{s,\B}$, our choice of $r$ readily implies 
\begin{equation}\label{qwerwege:22}
\|f\|_{L^{p}(\Sigma,w)} \lesssim \|g\|_{L^{p}(\Sigma,w)}, \quad(f,g) \in \F,
\end{equation}
where we recall that $p\in (0,\infty)$ and $w\in A_{\infty,\B}$ are arbitrary.

We are now ready to establish the desired estimates. Fixed $p\in (0,\infty)$,  we have by \eqref{qwerwege:22} that
\begin{equation}\label{qwerwege:2}
\|f^p\|_{L^{1}(\Sigma,w)} \leq C\, \|g^p\|_{L^{1}(\Sigma,w)}, \quad(f,g) \in \F,
\end{equation}
for every $w\in A_{\infty,\B}$, hence in particular for every $w\in A_{1,\B}$. We are now ready to invoke Theorem~\ref{thm:PhiAp} part (b) (the reader may  find convenient to introduce the family $\F_p$ consisting of the pairs $(f^p, g^p)$ with $(f,g)\in\F$) to conclude that \eqref{eq:uXv-4:1} yields \eqref{eq:BFSAi-3} as desired:
\[
\|f^p\,u\|_{\X_v} \lesssim \|g^p\,u\|_{\X_v}, \qquad(f,g) \in \F. 
\]

To obtain \eqref{eq:BFSAi-4} we observe that, with the notation introduced in \eqref{def:F-vv}, one has that \eqref{qwerwege:22} implies for any $0<q<\infty$
\begin{equation}\label{qwerwege:3}
\|F\|_{L^{q}(\Sigma,w)}^q 
= \sum_{j} 
\int_{\Sigma} f_j^{q}\, w\, d\mu
\lesssim  
\sum_{j} \int_{\Sigma} g_j^q\, w\, d\mu 
= 
\|G\|_{L^{q}(\Sigma, w)}^q, \quad(F,G) \in \F_{\ell^q},
\end{equation}
for every $w\in A_{\infty,\B}$. The same argument we have used to show that \eqref{qwerwege:2} yields \eqref{eq:BFSAi-3} can be then repeated to see that 
\eqref{qwerwege:3} implies \eqref{eq:BFSAi-4}. This completes the proof. 
\end{proof}

We formulate the limited range extrapolation on rearrangement invariant Banach function spaces as follows.  
\begin{theorem}\label{thm:lim-RIBFS}
	Suppose that $(\Sigma, \mu)$ is a non-atomic $\sigma$-finite measure space with $\mu(\Sigma)>0$. Let $\B$ be a Muckenhoupt basis so that $A_{\infty, \B}\subset\bigcup_{s>1} RH_{s, \B}$, let $\F$ be a family of extrapolation pairs, and let $1 \le \p_{-}< \p_{+}\le \infty$. Assume that for some $\p_0 \in [\p_-, \p_+]$ and for every $w_0 \in A_{\p_0/{\p_{-}}, \B} \cap RH_{(\p_+/\p_0)', \B}$, there holds
	\begin{equation}\label{eq:lim-3}
		\|f\|_{L^{\p_0}(\Sigma,w_0)} \leq C \|g\|_{L^{\p_0}(\Sigma,w_0)},\qquad(f,g) \in \F. 
	\end{equation}
	If $\X$ is a rearrangement invariant Banach function space over $(\Sigma, \mu)$ such that $\X^{\frac1r}$ is a Banach function space for some $r>\p_-$
	and 
	$\p_{-}<p_{\X} \le q_{\X}< \p_{+}$, 
	then 
	\begin{equation}\label{eq:lim-4}
		\|f\,u\|_{\X} \leq C \|g\,u\|_{\X}, \qquad(f, g) \in \F, 
	\end{equation}
	for every 
	\begin{equation}\label{eq:upuq} 	
		u^{p_{\X}} \in A_{p_{\X}/\p_{-}, \B} \cap RH_{(\p_{+}/p_{\X})', \B} \, \text{ and }\,  u^{q_{\X}} \in A_{q_{\X}/\p_{-}, \B} \cap RH_{(\p_{+}/q_{\X})',\B}. 
	\end{equation}
\end{theorem}

\begin{remark}
	The classical limited range extrapolation \cite[Theorem~4.9]{AM1} and \cite[Theorem~3.31]{CMP} is a particular case of Theorem~\ref{thm:lim-RIBFS} by simply taking $\X=L^p(\Sigma,\mu)$ with $\p_-<p<\p_+$. Note that with $r=p>\p_{-}$ we have $\X^{\frac1r}= L^{1}(\Sigma,\mu)$ is a Banach function space. Note that \eqref{eq:lim-4} and \eqref{eq:upuq} can be rewritten as 
	\[
	\|f\,u\|_{L^p(\Sigma,\mu)}
	\le
	C\, \|g\,u\|_{L^p(\Sigma,\mu)}, \qquad(f, g) \in \F, 
	\]
	for all $u^{p} \in A_{p/\p_{-}, \B} \cap RH_{(\p_{+}/p)', \B}$. 
\end{remark}

\begin{proof}
	We first observe that it suffices to consider the case $\p_{-}=1$. The general case follows by rescaling. Indeed,  we just need to set  $\widetilde{\p}_{-}=1$, $\widetilde{\p}_{+}=\p_{+}/\p_{-}$, $\widetilde{\p}_0=\p_0/\p_{-}$ and consider the new family $\widetilde{\F}$ consisting of the pairs $(\widetilde{f}, \widetilde{g}) =(f^{\p_{-}}, g^{\p_{-}})$ with $(f, g) \in \F$. 
	
	Assume from now on that $\p_{-}=1$. We may also assume that $\p_+<\infty$, otherwise the desired conclusion follows from Theorem~\ref{thm:BFSAp}. We claim that under the hypothesis \eqref{eq:lim-3}, by the same techniques as in \cite[Theorem~3.31]{CMP}, and using \eqref{eq:JN}  and that $\B$ is Muckenhoupt basis, one has for every $p_0 \in (1, \p_{+})$ and for every $w_0 \in A_{p_0, \B} \cap RH_{(\p_+/p_0)', \B}$,
	\begin{equation}\label{eq:ARH-2}
		\|f\|_{L^{p_0}(\Sigma, w_0)} \leq C \|g\|_{L^{p_0}(\Sigma, w_0)}, \quad (f,g) \in \F.
	\end{equation}

	Fix $\X$ as in the statement and let $u$ satisfy \eqref{eq:upuq}. Set $p:=p_{\X}$ and $q:=q_{\X}$. By Lemma~\ref{lem:RHApq} we have 
	\begin{equation}\label{43qfawfeg}
		u \in A_{p, q(\p_+/q)', \B}. 
	\end{equation}
	Let $t=p\,\p_+/q$ and note that $p<t<\p_+$. Pick $p_*$ with $1<p_*<\min\{t, 1+\frac{(r-1)\,(t-1)}{p-1}\}$. 
	Write   $\tau_{p_*}:=\big(\frac{\p_+}{p_*}\big)' (p_*-1)+1$ and define
	\begin{equation}\label{eq:sss1}
		s:=1+\frac{(p_*-1)\,(p-1)}{t-1},
		\quad
		\alpha_1:=\frac{t'}{(t/p)'}=\frac{t-p}{t-1}
		,
	\end{equation}
	\begin{equation}\label{eq:sss2}
		\beta_1:=\frac{p}{t-1},\quad
		\alpha_2:=\Big(\frac{\p_+}{p_*}\Big)',\quad\beta_2:=s\alpha_2 - \beta_1(\tau_{p_*}-1).
	\end{equation}
	Observe that one has  $1<s<\min\{p_*, p,r\}$, hence $\Y:=\X^{\frac1s}$ is a rearrangement invariant Banach function space. Easy computations yield
	\[
	\frac{\alpha_1+\beta_1}{\alpha_1}
	=
	\Big(\frac{t}{p} \Big)'
	=
	\Big(\frac{\p_+}{q} \Big)'
	\qquad\text{and}\qquad
	(\alpha_1+\beta_1)\,\Big(\frac{p}{\alpha_1} \Big)'
	=
	\frac{t}{t-1}\bigg(\frac{p\,(t-1)}{(t-p)} \bigg)'
	=
	p'.
	\]
	Hence, \eqref{43qfawfeg} is equivalent to $u^{\alpha_1+\beta_1}\in A_{\frac{p}{\alpha_1}, \frac{q}{\alpha_1},\B}$. This, the fact that $\X^{\frac{1}{\alpha_1}}$ is a Banach function space since $\alpha_1<1$ (this happens because $p>1$), and \eqref{eq:intro-CMP-ri:v=1} readily give that there exists $\mathcal{N}_1\in [1,\infty)$ such that 
	\begin{align}
		\label{eq:lim-M1} \|(M_{\B}h)\,u^{\alpha_1+\beta_1}\|_{\X^{\frac{1}{\alpha_1}}} 
		&\le \mathcal{N}_1 \|h\,u^{\alpha_1+\beta_1}\|_{\X^{\frac{1}{\alpha_1}}}, \quad\forall h \in \M_{\mu}. 
	\end{align}

	On the other hand, we also have
	\[
	\frac1{(p/s)'}
	=
	1-\frac{s}{p}
	=
	\frac1{p'}-\frac{p_*-1}{(t-1)\,p'}
	=
	\frac1{p'}\,\frac{t-p_*}{t-1}
	\]
	and
	\[
	\frac{\beta_2}{\alpha_2}
	=
	s-\beta_1\,\frac{\tau_{p_*}-1}{\alpha_2}
	=
	s-\beta_1\,(p_*-1)
	=
	1-\frac{p_*-1}{t-1}
	=
	\frac{t-p_*}{t-1}.
	\]
	Thus, 
	\begin{equation}\label{52fgwegve:1}
		\frac{\beta_2}{\alpha_2}\,\Big(\frac{p}{s}\Big)'
		=
		p'.
	\end{equation}
	To continue, we observe that 
	\[
	p_*-s
	=
	(p_*-1)\,\bigg(1-\frac{p-1}{t-1}\bigg)
	=
	\frac{(p_*-1)\,(t-p)}{t-1}
	=
	\frac{(p_*-1)\,(\p_+-q)\,p}{q\,(t-1)},
	\]
	which implies
	\begin{align*}
		\frac{s\,\p_+-p_*\,q}{\p_+-q}
		=
		s-\frac{(p_*-s)\,q}{\p_+-q}
		=
		s-\frac{(p_*-1)\,p}{t-1}
		=
		1-\frac{p_*-1}{t-1}
		=
		\frac{t-p_*}{t-1}
		=
		\frac{\beta_2}{\alpha_2}.
	\end{align*}
	Consequently, 
	\begin{align}\label{52fgwegve:2}
		\frac{1}{\beta_2\,\Big(\frac{(q/s)'}{\alpha_2}\Big)'}
		=
		\frac{\alpha_2}{\beta_2}\,
		\bigg[\frac1{\alpha_2}-\frac{1}{(q/s)'}\bigg]
		=
		\frac{\alpha_2}{\beta_2}\,
		\Big(\frac{s}{q}-\frac{p_*}{\p_+}\Big)
		=
		\frac{\alpha_2}{\beta_2}\,
		\frac{s\,\p_+-p_*\,q}{q\,\p_+}
		=
		\frac1{q\,(\p_+/q)'}.
	\end{align}
	Collecting \eqref{52fgwegve:1} and \eqref{52fgwegve:2} we see that \eqref{43qfawfeg} is equivalent to $u^{-\beta_2}\in A_{p_{\tilde{\Y}}, q_{\tilde{\Y}},\B}$ where $\tilde{\Y}=(\Y')^{\frac{1}{\alpha_2}}$, which is a quasi-Banach function space whose Boyd indices satisfy
	\[
	p_{\tilde{\Y}}
	=
	\frac{p_{\Y'}}{\alpha_2}
	=
	\frac{(q_{\Y})'}{\alpha_2}
	=
	\frac{(q/s)'}{\alpha_2},
	\qquad
	q_{\tilde{\Y}}
	=
	\frac{q_{\Y'}}{\alpha_2}
	=
	\frac{(p_{\Y})'}{\alpha_2}
	=
	\frac{(p/s)'}{\alpha_2}. 
	\]
	Thus, \eqref{eq:intro-CMP-ri:v=1} readily yields that there exists $\mathcal{N}_2\in [1,\infty)$ such that 
	\begin{align}
		\label{eq:lim-M2} \|(M_{\B}h)\,u^{-\beta_2}\|_{(\Y')^{\frac{1}{\alpha_2}}} 
		&\le \mathcal{N}_2 \|h\,u^{-\beta_2}\|_{(\Y')^{\frac{1}{\alpha_2}}}, \quad\forall h \in \M_{\mu}.  
	\end{align}
	We would like to note that strictly speaking one cannot invoke \eqref{eq:intro-CMP-ri:v=1} right away since we do not know whether $(\Y')^{\frac{1}{\alpha_2}}$ is a Banach function space because $\alpha_2>1$. Nonetheless, we can overcome this difficulty by working with the sublinear operator $M_{\B,\alpha_2} h:=M_\B(|h|^{\alpha_2})^{\frac1{\alpha_2}}$ which is bounded in $L^\gamma(\Sigma, v)$ whenever $\gamma>\alpha_2$ and $v\in A_{\gamma/\alpha_2,\B}$. This and the ideas used in Example~\ref{ex:RIBFS} (see \eqref{eq:RIBFS-M:u-v}--\eqref{eq:RIBFS-2}) show that $\|(M_{\B,\alpha_2} h)\,u^{-\frac{\beta_2}{\alpha_2}}\|_{\Y'}\lesssim \|h\,u^{-\frac{\beta_2}{\alpha_2}}\|_{\Y'}$ for every $h \in \M_{\mu}$. This immediately implies \eqref{eq:lim-M2}.

	Once we have obtained \eqref{eq:lim-M1} and \eqref{eq:lim-M2}, for any $h_1, h_2 \in \M_\mu$ with $h_1, h_2>0$ $\mu$-a.e., $h_1\, u^{\alpha_1+\beta_1} \in \X^{\frac{1}{\alpha_1}}$ and $h_2\, u^{-\beta_2} \in (\Y')^{\frac{1}{\alpha_2}}$, we define the Rubio de Francia iteration algorithms as:
	\begin{equation*}
		\mathcal{R}_1 h_1 :=\sum_{k=0}^{\infty}\frac{M_\B^k h_1}{2^k\, \mathcal{N}_1^k} \qquad\text{and}\qquad
		\mathcal{R}_2 h_2 :=\sum_{k=0}^{\infty}\frac{M_\B^k h_2}{2^k\, \mathcal{N}_2^k}.
	\end{equation*} 
	We claim that
	\begin{align}
		\label{eq:lim-hRh-1} h_1 \leq \mathcal{R}_1 h_1,\quad 
		\|(\mathcal{R}_1 h_1)\, u^{\alpha_1+\beta_1}\|_{\X^{\frac{1}{\alpha_1}}} 
		\leq 2\, \|h_1\, u^{\alpha_1+\beta_1}\|_{\X^{\frac{1}{\alpha_1}}}, 
		\quad [\mathcal{R}_1 h_1]_{A_{1,\B}} \leq 2\, \mathcal{N}_1,  
		\\ 
		\label{eq:lim-hRh-2} h_2 \leq \mathcal{R}_2 h_2,\quad 
		\|(\mathcal{R}_2 h_2)\, u^{-\beta_2}\|_{(\Y')^{\frac{1}{\alpha_2}}} 
		\leq C_0\, \|h_2\, u^{-\beta_2}\|_{(\Y')^{\frac{1}{\alpha_2}}}, 
		\quad [\mathcal{R}_2 h_2]_{A_{1,\B}} \leq 2\, \mathcal{N}_2, 
	\end{align}
	with $C_0=(1-2^{-\frac{1}{\alpha_2}})^{-\alpha_2}$. 
	The first and last conditions in both \eqref{eq:lim-hRh-1} and \eqref{eq:lim-hRh-2} are obtained as before and we omit the details. The second condition in \eqref{eq:lim-hRh-1} can be proved as before upon observing that $\alpha_1<1$ (since $p>1$) and hence $\X^{\frac{1}{\alpha_1}}$ is a Banach function space. The second condition in \eqref{eq:lim-hRh-2} needs some extra work since $(\Y')^{\frac{1}{\alpha_2}}$ may not a Banach function spaces as $\alpha_2>1$. In any case, the fact that $\alpha_2>1$ implies
	\begin{multline*}
		\|(\mathcal{R}_2 h_2)\, u^{-\beta_2}\|_{(\Y')^{\frac{1}{\alpha_2}}}^{\frac{1}{\alpha_2}}  
		=
		\|(\mathcal{R}_2 h_2)^{\frac{1}{\alpha_2}}\,  u^{-\frac{\beta_2}{\alpha_2}}\|_{\Y'}
		\le
		\Big\| \sum_{k=0}^{\infty}\frac{(M^k h_2)^{\frac{1}{\alpha_2}}}{(2^k\, \mathcal{N}_2^k)^{\frac{1}{\alpha_2}}} \,  u^{-\frac{\beta_2}{\alpha_2}}\Big\|_{\Y'}
		\\
		\le
		\sum_{k=0}^{\infty}\frac{\big\|(M^k h_2)^{\frac{1}{\alpha_2}}\,u^{-\frac{\beta_2}{\alpha_2}}\big\|_{\Y'}}{(2^k\, \mathcal{N}_2^k)^{\frac{1}{\alpha_2}}} 
		=
		\sum_{k=0}^{\infty}\frac{\|(M^k h_2)\,u^{-\beta_2}\|_{(\Y')^{\frac{1}{\alpha_2}}}^{\frac{1}{\alpha_2}}  }{(2^k\, \mathcal{N}_2^k)^{\frac{1}{\alpha_2}}} 
		\\
		\le
		\|h_2\,u^{-\beta_2}\|_{(\Y')^{\frac{1}{\alpha_2}}}^{\frac{1}{\alpha_2}} \,
		\sum_{k=0}^{\infty}2^{-\frac{k}{\alpha_2}}
		=
		(1-2^{-\frac{1}{\alpha_2}})^{-1}\,\|h_2\,u^{-\beta_2}\|_{(\Y')^{\frac{1}{\alpha_2}}}^{\frac{1}{\alpha_2}} .
	\end{multline*}
	We would like to observe that in the present scenario $\X^{\frac{1}{\alpha_1}}$ is a Banach function space, but if this were not the case one could still prove the second condition in \eqref{eq:lim-hRh-1} with a different constant, much as we have done with the second condition in \eqref{eq:lim-hRh-2}. This is not relevant in this proof but may be needed in the proof of Theorem~\ref{thm:LRE}. 
	
	We are now ready to turn to the proof of \eqref{eq:lim-4}. We first observe that $1<p_*<t\le \p_+$, since $1< p\le q$. Hence \eqref{eq:ARH-2} holds. Considering then the approach used in Theorem \ref{thm:BFSAp}, we only need to prove that for every $f, g\in \M_{\mu}$ with $0<\|f\,u\|_{\X}<\infty$ and $0<\|g\,u\|_{\X}<\infty$, there exist a constant $C>0$ and a $\B$-weight $w=w(f, g) \in A_{p_*, \B} \cap RH_{(\p_+/p_*)', \B}$ such that 
	\begin{align}\label{eq:lim-emb}
		\|f\,u\|_{\X} \le C\, \|f\|_{L^{p_*}(\Sigma, w)} \quad\text{and}\quad 
		\|g\|_{L^{p_*}(\Sigma, w)} \le C\, \|g\,u\|_{\X}. 
	\end{align}   
	To show \eqref{eq:lim-emb}, fix $f, g \in \M_{\mu}$ with $0<\|f\,u\|_{\X}<\infty$ and $0<\|g\,u\|_{\X}<\infty$. Set
	\begin{equation}\label{eq:fg}
		h_1(x):=\frac{g(x)}{\|g\,u\|_{\X}},\quad x \in \Sigma.
	\end{equation}
	Recall the definition of $s$ in \eqref{eq:sss1} and let $r_*:=p_*/s$, which satisfies $r_*>1$ as $s<p_*$. Using the fact that as observed above $\Y=\X^{\frac1s}$ is a Banach function space, by \eqref{eq:fX-norm}, there exists a non-negative function $h_2 \in \Y'$ with $\|h_2\|_{\Y'} \le 1$ such that 
	\begin{align}\label{eq:lim-fu}
		\|f\,u\|_{\X}^s = \|f^s\, u^s\|_{\Y} \le 2 \int_{\Sigma} f^s\, u^s\, h_2 \, d\mu.  
	\end{align}
	Applying Lemma \ref{lem:hh} with $\varepsilon=1$ to $h_1$ and $h_2$, we can find two functions $\widetilde{h}_1,\widetilde{h}_2> 0$ $\mu$-a.e.~such that 
	\begin{align}\label{eq:hhXY}
		h_1 \le \widetilde{h}_1, \quad h_2 \le \widetilde{h}_2,\quad 
		\|\widetilde{h}_1\,u\|_{\X} \le 2, \quad\text{and}\quad \|\widetilde{h}_2\|_{\Y'} \le 2.  
	\end{align}

	Now we define 
	\begin{align}\label{eq:HRv}
		H_1 := \mathcal{R}_1(\widetilde{h}_1^{\alpha_1} u^{-\beta_1})^{\frac{1}{\alpha_1}} u^{\frac{\beta_1}{\alpha_1}} \quad\text{and}\quad 
		H_2 := \mathcal{R}_2(\widetilde{h}_2^{\alpha_2} u^{\beta_2})^{\frac{1}{\alpha_2}} u^{-\frac{\beta_2}{\alpha_2}}. 
	\end{align}
	It follows from \eqref{eq:lim-hRh-1}, \eqref{eq:lim-hRh-2}, and \eqref{eq:hhXY} that 
	\begin{align}
		\label{eq:lim-RdF-1} \widetilde{h}_1 \le H_1, \qquad 
		\|H_1\,u\|_{\X} \le 2^{1+\frac{1}{\alpha_1}},  
		\qquad [H_1^{\alpha_1} u^{-\beta_1}]_{A_1} \le 2\,\mathcal{N}_1, 
	\end{align} 
and
		\begin{align}
				\label{eq:lim-RdF-2} \widetilde{h}_2 \le H_2, \qquad 
		\|H_2\|_{\Y'} \le 2\,C_0^{\frac{1}{\alpha_2}},  
		\qquad\,  [H_2^{\alpha_2} u^{\beta_2}]_{A_1} \le 2\,\mathcal{N}_2.  
	\end{align} 
	Then invoking \eqref{eq:Holder}, \eqref{eq:hhXY}, \eqref{eq:lim-RdF-1} and \eqref{eq:lim-RdF-2}, we obtain 
	\begin{align}\label{eq:lim-H}
		\mathcal{I} :&= \int_{\Sigma} H_1^s\, H_2\, u^s \, d\mu 
		\leq \|H_1^s\, u^s\|_{\Y}\, \|H_2\|_{\Y'}
		=\|H_1\, u\|_{\X}^{s} \|H_2\|_{\Y'}
		\leq 2^{1+s+\frac{s}{\alpha_1}}\,C_0^{\frac{1}{\alpha_2}}.
	\end{align}
	Set $w:=H_1^{-s\,(r_*-1)}\, H_2\, u^s$.  Due to \eqref{eq:sss1} and \eqref{eq:sss2} one has 
	\[
	\frac{s\,(r_*-1)}{p_*-1}
	=
	\frac{p_*-s}{p_*-1}
	=
	\Big(1-\frac{p-1}{t-1}\Big)
	=
	\frac{t-p}{t-1}
	=
	\alpha_1
	\]
	and 
	\begin{equation*}
		w^{(\p_{+}/{p_*})'} = (H_1^{\alpha_1}\, u^{-\beta_1})^{1-\tau_{p_*}}\, (H_2^{\alpha_2}\, u^{\beta_2}). 
	\end{equation*}
	This together with \eqref{eq:Ap-product}, \eqref{eq:lim-RdF-1} and \eqref{eq:lim-RdF-2} implies that 
	\begin{align*}
		[w^{(\p_{+}/{p_*})'}]_{A_{\tau_{p_*}, \B}} 
		\le [H_1^{\alpha_1}\, u^{-\beta_1}]_{A_{1, \B}}^{\tau_{p_*}-1}\, [H_2^{\alpha_2}\, u^{\beta_2}]_{A_{1, \B}} 
		\le 2^{\tau_{p_*}}\, \mathcal{N}_1^{\tau_{p_*}-1}\, \mathcal{N}_2. 
	\end{align*}
	Hence, this and \eqref{eq:JN} give $w \in A_{p_*, \B} \cap RH_{(\p_+/p_*)', \B}$. 
	
	On the other hand, \eqref{eq:lim-fu}, \eqref{eq:hhXY}, \eqref{eq:lim-RdF-1}, \eqref{eq:lim-RdF-2}, Hölder's inequality, and \eqref{eq:lim-H} imply 
	\begin{multline*}
		\|f\,u\|_{\X}^s 
		\le 
		2\, \int_{\Sigma} f^s\, H_2\, u^s \, d\mu 
		= 2\,\int_{\Sigma} f^s\, H_1^{-\frac{s}{r'_*}}\, H_1^{\frac{s}{r'_*}}\, H_2\, u^s \, d\mu
		\\
		\leq 2\, \mathcal{I}^{\frac{1}{r'_0}}\, \left(\int_{\Sigma} f^{p_*}\, H_1^{-s(r_*-1)}\, H_2\, u^s \, d\mu \right)^{\frac{1}{r_*}}
		\le 2^{1+\frac{1}{r'_*}(1+s+\frac{s}{\alpha_1})}\,C_0^{\frac{1}{r_*'\,\alpha_2}}\, \|f\|_{L^{p_*}(\Sigma,\, w)}^s. 
	\end{multline*}
	Furthermore, by \eqref{eq:fg}, \eqref{eq:hhXY}, \eqref{eq:lim-RdF-1} there holds $g \leq H_1 \|g\,u\|_{\X}$. This and \eqref{eq:lim-H} yield
	\begin{align*}
		\|g\|_{L^{p_*}(\Sigma, w)}^{p_*} 
		= \int_{\Sigma} g^{p_*}\, H_1^{-s(r_*-1)}\, H_2\, u^s \, d\mu
		\leq \|g\,u\|_{\X}^{p_*} \int_{\Sigma} H_1^s\, H_2\, u^s \, d\mu
		\le  2^{1+s+\frac{s}{\alpha_1}}\,C_0^{\frac{1}{\alpha_2}}\, \|g\,u\|_{\X}^{p_*}.
	\end{align*}
	We have therefore established \eqref{eq:lim-emb} and the proof is complete. 
\end{proof}

Using the ideas in the proof of Theorem \ref{thm:lim-RIBFS}, we conclude the following general limited range extrapolation on Banach function spaces. The detailed proof is left to the interested reader. 
\begin{theorem}\label{thm:LRE}
	Suppose that $(\Sigma, \mu)$ is a non-atomic $\sigma$-finite measure space with $\mu(\Sigma)>0$. Let $\B$ be a Muckenhoupt basis and $\mathcal{F}$ be a family of extrapolation pairs.  Fix $1 \le \p_{-}< \p_{+}<\infty$ and $\p_{-} \le \p_0 \le \p_{+}$. 
	Let $u$ and $v$ be $\B$-weights on $(\Sigma, \mu)$ such that $v(B)<\infty$ for every $B\in\B$, and let $\X_v$ be a Banach function space over $(\Sigma, v)$. Let $M_{\B}$ denote the Hardy-Littlewood maximal function on $(\Sigma,\mu)$ associated with $\B$ and let $M'_{\B, v}h:=M_{\B}(h\,v)/v$ for each $h \in\M_{\mu}$.  Assume that there exist $p_{*} \in (\p_{-}, \p_{+})$, $s \in (0, p_{*})$ and $\mathcal{N}_1, \mathcal{N}_2<\infty$ such that $\Y_v:=\X_v^{\frac1s}$ is a Banach function space, and 
	\begin{align}
		\label{eq:LRE-1} \|(M_{\B}h)\,u^{\alpha_1+\beta_1}\|_{\X_v^{\frac{1}{\alpha_1}}} 
		&\le \mathcal{N}_1 \|h\,u^{\alpha_1+\beta_1}\|_{\X_v^{\frac{1}{\alpha_1}}}, \quad\forall h \in \M_{\mu}, 
		\\
		\label{eq:LRE-2} \|(M_{\B,v}' h)\,u^{-\beta_2}\|_{(\Y_v')^{\frac{1}{\alpha_2}}} 
		&\le \mathcal{N}_2 \|h\,u^{-\beta_2}\|_{(\Y_v')^{\frac{1}{\alpha_2}}}, \quad\forall h \in \M_{\mu}, 
	\end{align}
	where 
	\begin{align}
		\label{eq:LRE-aa} &\alpha_1 = \p_{-} \left(\frac{p_{*}-s}{p_{*}-\p_{-}}\right), \quad
		\alpha_2 = \left(\frac{\p_{+}}{p_{*}}\right)', \quad \beta_1 \in \R,
		\\
		\label{eq:LRE-bb} &\beta_2=s\alpha_2-\beta_1(\tau_{p_{*}}-1), \quad
		\tau_{p_{*}}=\left(\frac{\p_{+}}{p_{*}}\right)'\left(\frac{p_{*}}{\p_{-}}-1\right)+1. 
	\end{align} 
	If for every $w_0 \in A_{\p_0/{\p_{-}}, \B} \cap RH_{(\p_{+}/\p_0)', \B}$, 
	\begin{equation}\label{eq:LRE-3}
		\|f\|_{L^{\p_0}(\Sigma,w_0)} \leq C \|g\|_{L^{\p_0}(\Sigma,w_0)},\qquad(f,g) \in \F, 
	\end{equation}
	then 
	\begin{equation}\label{eq:LRE-4}
		\|f\,u\|_{\X_v} \leq C \|g\,u\|_{\X_v}, \qquad(f, g) \in \F. 
	\end{equation}
\end{theorem}

\begin{remark}
	Theorem \ref{thm:LRE} refines the limited range extrapolation on the variable Lebesgue space $L^{p(\cdot)}(\Rn, \Ln)$ established in \cite[Proposition~5.8]{CW}. Indeed, let $p(\cdot)$ be a measurable function with $\p_{-}<p_{-} \le p_{+}<\p_{+}$.  Let $\X=L^{p(\cdot)}(\Rn, \Ln)$ and let $\B$ be the collection of all balls in $\Rn$. Then 
	\begin{align*}
		\X^{\frac{1}{\alpha_1}} = L^{\frac{p(\cdot)}{\alpha_1}}(\Rn, \Ln), \quad 
		\Y=L^{\frac{p(\cdot)}{s}}(\Rn, \Ln),\quad 
		(\Y')^{\frac{1}{\alpha_2}}=L^{ \left(\frac{p(\cdot)}{s} \right)' \frac{1}{\alpha_2}}(\Rn, \Ln). 
	\end{align*}
	Note that $\Y$ is a Banach function space if $p_{-}>s$. Thus, in this setting, our assumption becomes that there exist $p_{*} \in (\p_{-}, \p_{+})$, $0<s<\min\{p_{*}, \p_{-}\}$ and $\mathcal{N}_1, \mathcal{N}_2<\infty$ such that 
	\begin{equation}\label{eq:LRE-12}
		\text{both \eqref{eq:LRE-1} and \eqref{eq:LRE-2} hold.}
	\end{equation} 
	If we assume in addition that $p(\cdot) \in {\rm LH}$, then \eqref{eq:LRE-12} holds provided that 
	$u^{\alpha_1+\beta_1} \in A_{p(\cdot)/\alpha_1, \B}$, $u^{-\beta_2} \in A_{(p(\cdot)/s)'/\alpha_2, \B}$, ans $s$ satisfies 
	\begin{align*}
		\max\bigg\{p_{-} - p_{*} \Big(\frac{p_{-}}{\p_{-}}-1\Big), \, \frac{p_{*}p_{+}}{\p_{+}}\bigg\} 
		< s < \min\{p_{*}, \, p_{-}\}, 
	\end{align*}
	see \cite[p.1230]{CW} for details. Consequently, we obtain the limited range extrapolation on the variable Lebesgue spaces. 
\end{remark}

\section{Extrapolation on modular spaces}\label{sec:modular}

In this section we establish a variety of  extrapolation theorems on modular spaces. We begin with the so-called $A_p$ extrapolation. As before, a family of extrapolation pairs $\F$ is a collection of pairs $(f, g)$ of nonnegative measurable functions. Our main result generalizes \cite[Theorems~4.15]{CMP}:

\begin{theorem}\label{thm:PhiAp}
	Suppose that $(\Sigma, \mu)$ is a non-atomic $\sigma$-finite measure space with $\mu(\Sigma)>0$. Let $\B$ be a basis and let $\F$ be a family of extrapolation pairs. Let $u$ and $v$ be $\B$-weights on $(\Sigma, \mu)$ such that $v(B)<\infty$ for every $B\in\B$, and let $\Phi$ be a Young function such that $\Phi\in\Delta_2$ \textup{(}equivalently, $I_\Phi<\infty$\textup{)} with a doubling constant $C_{\Phi}$. Let $M_{\B}$ denote the Hardy-Littlewood maximal function on $(\Sigma,\mu)$ associated with $\B$ and let $M'_{\B, v}h:=M_{\B}(h\,v)/v$ for each $h \in\M_{\mu}$. Then, the following hold:	

\begin{list}{\textup{(\theenumi)}}{\usecounter{enumi}\leftmargin=1cm \labelwidth=1cm \itemsep=0.2cm \topsep=.2cm \renewcommand{\theenumi}{\alph{enumi}}}
	
	\item Let $p_0\in (1,\infty)$, and assume that there are $\mathcal{N}_1,\mathcal{N}_2<\infty$ so that 
	\begin{align}
		\label{eq:Phifu-1} \int_{\Sigma} \Phi((M_{\B}h)\, u)\, v\, d\mu 
		&\le \mathcal{N}_1 \int_{\Sigma} \Phi(|h|\, u)\,v\, d\mu,\quad\forall\,h \in \M_{\mu}, 
		\\
		\label{eq:Phifu-2} \int_{\Sigma} \overline{\Phi}((M'_{\B, v}h)\, u^{-1})\, v\, d\mu 
		&\le \mathcal{N}_2 \int_{\Sigma} \overline{\Phi}(|h|\, u^{-1})\,v\, d\mu,\quad\forall\,h \in \M_{\mu}. 
	\end{align} 
	If for every $w \in A_{p_0, \B}$, one has 
\begin{equation}\label{eq:Phifu-3}
	\|f\|_{L^{p_0}(\Sigma,w)} \leq \Psi([w]_{A_{p_0, \B}}) \|g\|_{L^{p_0}(\Sigma,w)}, \quad(f,g) \in \F, 
\end{equation} 
	where $\Psi: [1, \infty) \to [1, \infty)$ is a non-decreasing function, then 
	\begin{equation}\label{eq:Phifu-4}
		\int_{\Sigma}\Phi(f\,u)\,v\,d\mu
		\leq 
			C_1\, \int_{\Sigma}\Phi(g\,u)\,v\,d\mu,\qquad(f,g) \in \F,
	\end{equation}
	and 
	\begin{equation}\label{eq:Phifu-4-vv}
		\int_{\Sigma}\Phi \bigg(\Big(\sum_j f_j^{p_0} \Big)^{\frac1{p_0}}u \bigg)\,v\,d\mu 
		\leq 
		  C_1 \int_{\Sigma}\Phi \bigg(\Big(\sum_j g_j^{p_0}\Big)^{\frac{1}{p_0}}u \bigg) v\,d\mu, \quad \{(f_j, g_j)\} \subset \F, 
	\end{equation}
	where $C_1=C_{\Phi}\, \max\{C_0, C_0^{2\,I_\Phi}\}$ and $C_0=	2^{3+\frac2{p_0'}}\,\Psi(2\,\mathcal{N}_1^{p_0-1}\, \mathcal{N}_2)$.

	\item 	Assume that there is $\mathcal{N}_2<\infty$ so that \eqref{eq:Phifu-2} holds. If for every $w \in A_{1, \B}$, one has
	\begin{equation}\label{eq:Phifu-3:1}
		\|f\|_{L^{1}(\Sigma,w)} \leq \Psi([w]_{A_{1, \B}})\|g\|_{L^{1}(\Sigma,w)}, \qquad(f,g) \in \F, 
	\end{equation}  
	where $\Psi: [1, \infty) \to [1, \infty)$ is a non-decreasing function, then 
	\begin{equation}\label{eq:Phifu-4:1}
		\int_{\Sigma}\Phi(f\,u)\,v\,d\mu
		\leq C_1 \int_{\Sigma}\Phi(g\,u)\,v\,d\mu,\qquad(f,g) \in \F,
	\end{equation}
	and 
	\begin{equation}\label{eq:Phifu-4-vv:1}
		\int_{\Sigma}\Phi \bigg(\Big(\sum_j f_j \Big)\,u \bigg)\,v\,d\mu 
		\leq C_1 \int_{\Sigma}\Phi \bigg(\Big(\sum_j g_j \Big)\,u \bigg)\,v\,d\mu,\quad \{(f_j, g_j)\} \subset \F, 
	\end{equation} 
	where $C_1=C_{\Phi}\, \max\{C_0, C_0^{2\,I_\Phi}\}$ and $C_0=8\,\Psi(2\,\mathcal{N}_2)$. 
\end{list}
	Moreover, in any of the two scenarios if $\B$ is additionally assumed to be a Muckenhoupt basis, then \eqref{eq:Phifu-1} or \eqref{eq:Phifu-2} imply that for every $q \in (1, \infty)$, 
\begin{equation}\label{eq:Phifu-5}
\int_{\Sigma}\Phi \bigg(\Big(\sum_j f_j^q \Big)^{\frac1q}\,u \bigg)\,v\,d\mu 
\leq C\, \int_{\Sigma}\Phi \bigg(\Big(\sum_j g_j^q \Big)^{\frac1q}\,u \bigg)\,v\,d\mu,\quad \{(f_j, g_j)\} \subset \F. 
\end{equation}
\end{theorem}

\begin{remark}\label{remark:rescale:mod}
	As in Remark \ref{remark:rescale} one can easily rescale in the previous result. To be more precise, suppose that for some $r>0$ and $p_0\in[r,\infty]$ there holds
	\begin{equation}\label{eq:uXv-3:rescale:mod}
		\|f\|_{L^{p_0}(\Sigma,w)} \leq \Psi([w]_{A_{p_0/r, \B}})\|g\|_{L^{p_0}(\Sigma,w)}, \qquad(f,g) \in \F, 
	\end{equation}  
	for all $w\in A_{p_0/r,\B}$, and where $\Psi: [1, \infty) \to [1, \infty)$ is a non-decreasing function. Then, much as we did there we may apply Theorem~\ref{thm:PhiAp} with the family of pairs $(f^r,g^r)$ with $(f,g)\in\F$) to easily obtain, with $\Phi_r(t):=\Phi(t^r)$, $t\ge 0$,  
	\begin{equation}\label{eq:uXv-4:rescale:mod}
	\int_{\Sigma}\Phi_r(f\,u)\,v\,d\mu\lesssim 
	\int_{\Sigma}\Phi_r(g\,u)\,v\,d\mu,\qquad(f,g) \in \F,
	\end{equation}
	provided
	\begin{align}
\label{thm:BFSAp:M:rescale:mod} \int_{\Sigma} \Phi((M_{\B}h)\, u^r)\, v\, d\mu 
	&\le \mathcal{N}_1 \int_{\Sigma} \Phi(|h|\, u^r)\,v\, d\mu,\quad\forall\,h \in \M_{\mu}, 
	\\
\label{thm:BFSAp:Mprime:rescale:mod} 	\int_{\Sigma} \overline{\Phi}((M'_{\B, v}h)\, u^{-r})\, v\, d\mu 
	&\le \mathcal{N}_2 \int_{\Sigma} \overline{\Phi}(|h|\, u^{-r})\,v\, d\mu,\quad\forall\,h \in \M_{\mu}. 
\end{align} 
	when $r<p_0$ and assuming only \eqref{thm:BFSAp:Mprime:rescale:mod} when $r=p_0$. Further details are left to the interested reader. 	
\end{remark}

The proof of the previous result will be based on the following proposition which is interesting on its own right.
\begin{proposition}\label{prop:extrapol-maintool:mod}
	Suppose that $(\Sigma, \mu)$ is a non-atomic $\sigma$-finite measure space with $\mu(\Sigma)>0$. Let $\B$ be a basis, $u$ and $v$ be $\B$-weights on $(\Sigma, \mu)$ such that $v(B)<\infty$ for every $B\in\B$, and let $\Phi$ be a Young function. Let $M_{\B}$ denote the Hardy-Littlewood maximal function on $(\Sigma,\mu)$ associated with $\B$ and let $M'_{\B, v}h := M_{\B}(h\,v)/v$ for each $h \in\M_{\mu}$. 
	Then, for every $f,g\in \M_\mu$ so that $0<\rho_v^{\Phi}(f\,u), \rho_v^{\Phi}(g\,u)<\infty $ the following hold:	
	\begin{list}{\textup{(\theenumi)}}{\usecounter{enumi}\leftmargin=1cm \labelwidth=1cm \itemsep=0.2cm \topsep=.2cm \renewcommand{\theenumi}{\alph{enumi}}}
		
	\item Let $p_0\in (1,\infty)$, and assume that there are $\mathcal{N}_1,\mathcal{N}_2<\infty$ so that both \eqref{eq:Phifu-1} and \eqref{eq:Phifu-2} hold. Then, there exists a $\B$-weight $w=w(f,g)\in A_{p_0,\B}$ satisfying $[w]_{A_{p_0,\B}}\le 2^{p_0}\,\mathcal{N}_1^{p_0-1} \,\mathcal{N}_2$ and such that for any $\theta\in (0,1]$
		\begin{equation}\label{eq:embedding:f-g:mod1}
		\rho_v^{\Phi}(f\,u) \leq 2^{1+\frac2{p_0'}}\,\theta^{\frac1{p_0'}}\big(\rho_v^{\Phi}(\theta^{-1}\,g\,u) + 	\rho_v^{\Phi}(f\,u)\big)^{\frac1{p_0'}}\,\|f\|_{L^{p_0}(\Sigma, w)}
			. 	
		\end{equation}
	and
		\begin{equation}\label{eq:embedding:f-g:mod2}
		\|g\|_{L^{p_0}(\Sigma,w)}\le  2\,	\theta^{\frac{1}{p_0}} \big(\rho_v^{\Phi}(\theta^{-1}\,g\,u) + 	\rho_v^{\Phi}(f\,u)\big)^{\frac1{p_0}}
		. 	
	\end{equation}
		In particular, there exists a $\B$-weight $w=w(f)\in A_{p_0,\B}$ such that $[w]_{A_{p_0,\B}}\le 2^{p_0}\,\mathcal{N}_1^{p_0-1} \,\mathcal{N}_2$ and 
		\begin{equation}\label{eq:embedding:f:mod}
			2^{-1-\frac{3}{p'_0}}\, \rho_v^{\Phi}(f\,u)^{^{\frac1{p_0}}} \leq  \|f\|_{L^{p_0}(\Sigma,w)}  \le 2^{1+\frac{1}{p_0}}\, \rho_v^{\Phi}(f\,u)^{^{\frac1{p_0}}} . 	
		\end{equation}
		
		\item Assume that there is $\mathcal{N}_2<\infty$ so that \eqref{eq:Phifu-2} holds. Then, there exists a $\B$-weight $w=w(f,g)\in A_{1,\B}$ satisfying $[w]_{A_{p_0,\B}}\le 2\,\mathcal{N}_2$ and such that
		\begin{equation}\label{eq:embedding:f-g:1:mod}
			\rho_v^{\Phi}(f\,u)\leq 2\,\|f\|_{L^{1}(\Sigma,w)}  
			\quad\text{and}\quad 
			\|g\|_{L^{1}(\Sigma,w)}\le  2\,	\theta\big(\rho_v^{\Phi}(\theta^{-1}\,g\,u) + 	\rho_v^{\Phi}(f\,u)\big). 
		\end{equation}
		In particular, there exists a $\B$-weight $w=w(f)\in A_{1,\B}$ satisfying $[w]_{A_{1,\B}}\le 2\,\mathcal{N}_2$ and such that
		\begin{equation}\label{eq:embedding:f:1:mod}
2^{-1}\, \rho_v^{\Phi}(f\,u)\leq  \|f\|_{L^{1}(\Sigma,w)}  \le 4\, \rho_v^{\Phi}(f\,u). 	
		\end{equation}
	\end{list}
\end{proposition}

Assuming this result momentarily we can easily prove Theorem \ref{thm:PhiAp}:

\begin{proof}[Proof of Theorem \ref{thm:PhiAp}]
Fix $(f, g) \in \F$. We may assume that $\rho_v^{\Phi}(f\,u)>0$ and $\rho_v^{\Phi}(g\,u)<\infty$, otherwise there is nothing to prove. We claim that $\rho_v^{\Phi}(g\,u)>0$. Otherwise $g=0$ $\mu$-a.e.~(since $u,v$ are $\B$-weights) and by \eqref{eq:Phifu-3} we have that $\|f\|_{L^{p_0}(\Sigma, w)}= 0$ for every $w \in A_{p_0, \B}$. In turn,  $f=0$ $\mu$-a.e., which violates our assumption  $\rho_v^{\Phi}(f\,u)>0$. We also claim that $f<\infty$ $\mu$-a.e. Otherwise, there exists a measurable set $E\subset \Sigma$ with $\mu(E)>0$ such that $f=\infty$ on $E$. By \eqref{eq:Phifu-3} and the fact that $\B$-weights are $\mu$-a.e.~positive, it follows that 
$\|g\|_{L^{p_0}(\Sigma, w)}= \infty \quad\text{ for every } w \in A_{p_0, \B}$. This clearly contradicts Proposition~\ref{prop:extrapol-maintool:mod} applied to $g$ since $0<\rho_v^{\Phi}(g\,u)<\infty$. 

To continue we recall that since $(\Sigma, \mu)$ is $\sigma$-finite, there exists an increasing sequence of $\mu$-measurable sets  $\{\Sigma_j\}_{j=1}^{\infty}$ such that $\Sigma=\bigcup_{j=1}^{\infty} \Sigma_j$ and $\mu(\Sigma_j)<\infty$ for all $j$. For every $N \ge 1$, we define 
	\[
	f_N := f\, \mathbf{1}_{\{x\in \Sigma_N : f(x)\leq N, u(x) \le N, v(x) \le N\}}
	\] 
	and one clearly has
	\begin{align}\label{34rfqfr}
		\rho_v^{\Phi}(f_N\,u) \le N \Phi(N^2) \mu(\Sigma_N) < \infty. 
	\end{align}
	Recalling that $f<\infty$ $\mu$-a.e.~and that $u, v$ are $\B$-weights  we have $f_N\,u\nearrow f\,u$ as $N \to \infty$ $\mu$-a.e., hence also $v$-a.e., since $v$ is a $\B$-weight. These and the monotone convergence theorem imply that $\rho_v^{\Phi}(f_N\,u)\nearrow \rho_v^{\Phi}(f\,u)$ as $N\to\infty$, hence the assumption  $\rho_v^{\Phi}(f\,u)>0$ readily yields that $\rho_v^{\Phi}(f_N\,u)>0$ for $N\ge N_0$.

	Consider first the case $p_0>1$. Fixed $N\ge N_0$, apply Proposition \ref{prop:extrapol-maintool:mod} part (a) to $f_N$ and $g$ (which satisfy $0<\rho_v^{\Phi}(f_N\,u), \rho_v^{\Phi}(g\,u)<\infty$) so that there exists $w=w(f_N, g) \in A_{p_0, \B}$ with $[w]_{A_{p_0, \B}} \le 2\,\mathcal{N}_1^{p_0-1} \mathcal{N}_2$ such that for any $\theta\in (0,1]$ there hold
	\begin{align}
		\label{eq:PfNu-1} 
			\rho_v^{\Phi}(f_N\,u) 
			\leq 
			2^{1+\frac2{p_0'}}\,\theta^{\frac1{p_0'}}\big(\rho_v^{\Phi}(\theta^{-1}\,g\,u) + 	\rho_v^{\Phi}(f_N\,u)\big)^{\frac1{p_0'}}\,\|f\|_{L^{p_0}(\Sigma, w)} 
				\end{align} 
			and
	\begin{align}
		\label{eq:PfNu-2} 
		\|g\|_{L^{p_0}(\Sigma,w)}
		\le  2\,	\theta^{\frac{1}{p_0}} \big(\rho_v^{\Phi}(\theta^{-1}\,g\,u) + 	\rho_v^{\Phi}(f\,u)\big)^{\frac1{p_0}}.
		\end{align}
	We can then invoke \eqref{eq:Phifu-3} to deduce that 
	\begin{align}\label{eq:PfPg}
		\rho_v^{\Phi}(f_Nu) 
		&\le 
		2^{1+\frac2{p_0'}}\,\theta^{\frac1{p_0'}}\big(\rho_v^{\Phi}(\theta^{-1}\,g\,u) + 	\rho_v^{\Phi}(f_N\,u)\big)^{\frac1{p_0'}}\,\|f\|_{L^{p_0}(\Sigma, w)} 
		\nonumber\\
		&\le 
		2^{1+\frac2{p_0'}}\,\Psi([w]_{A_{p_0}})\,\theta^{\frac1{p_0'}}\,\big(\rho_v^{\Phi}(\theta^{-1}\,g\,u) + 	\rho_v^{\Phi}(f_N\,u)\big)^{\frac1{p_0'}}\,	 \|g\|_{L^{p_0}(\Sigma, w)} 		
		\nonumber\\
		&
		\le
		2^{2+\frac2{p_0'}}\,\Psi(2\,\mathcal{N}_1^{p_0-1} \mathcal{N}_2)\,\theta\,\big(\rho_v^{\Phi}(\theta^{-1}\,g\,u) + 	\rho_v^{\Phi}(f_N\,u)\big)
				\nonumber\\
		&
		=: C_0\, \frac{\theta}{2}\,\big(\rho_v^{\Phi}(\theta^{-1}\,g\,u) + 	\rho_v^{\Phi}(f_N\,u)\big).
	\end{align}
Setting $\theta:=\min\{1, C_0^{-1}\}$, we can use \eqref{34rfqfr} to hide the last term on the right-hand side and conclude that 
	\begin{align*}
		\rho_v^{\Phi}(f_N\,u) \le C_0\,\theta\, \rho_v^{\Phi}(\theta^{-1}\,g\,u). 
	\end{align*}
	As mentioned above $\rho_v^{\Phi}(f_N\,u)\nearrow \rho_v^{\Phi}(f\,u)$ as $N\to\infty$ and therefore
	\begin{align}\label{eq:Nugu}
		\rho_v^{\Phi}(f\,u) \le C_0\,\theta\, \rho_v^{\Phi}(\theta^{-1}\,g\,u). 
	\end{align}
	Moreover, using that $I_\Phi<\infty$, then it follows from \eqref{eq:Phiindices} that there is a constant $C_{\Phi}>0$ such that 
\begin{equation*}
	\Phi(\lambda t) \leq C_{\Phi}\,\lambda^{2\,I_{\Phi}}\, \Phi(t), \quad\text{ for every } t>0 \text{ and } \lambda>1.
\end{equation*} 
Hence,  
	\[
		\rho_v^{\Phi}(f\,u) \le C_0\,\theta\, C_{\Phi}\, \theta^{-2\,I_{\Phi}}\,\rho_v^{\Phi}(g\,u)
		=
		C_{\Phi}\, \max\{C_0, C_0^{2\,I_\Phi}\}\,\rho_v^{\Phi}(g\,u),
	\]	
	where $C_0=	2^{3+\frac2{p_0'}}\,\Psi(2\,\mathcal{N}_1^{p_0-1} \mathcal{N}_2)$. The case $p_0=1$ follows using the same argument but invoking Proposition~\ref{prop:extrapol-maintool:mod} part (b), details are left to the interested reader.
	
To complete the proof we need to justify the vector-valued inequalities  \eqref{eq:Phifu-4-vv}, \eqref{eq:Phifu-4-vv:1}, and \eqref{eq:Phifu-5}. The first two ones are obtained much as in the proof of Theorem~\ref{thm:BFSAp}, details are left the reader. Regarding \eqref{eq:Phifu-5}, as shown at the end of  Theorem~\ref{thm:BFSAp} from
either \eqref{eq:Phifu-3} or \eqref{eq:Phifu-3:1} one can obtain a version of \eqref{eq:Phifu-3} valid for any $1<q<\infty$ in place of $p_0$. This becomes our new initial assumption and we readily obtain \eqref{eq:Phifu-4-vv} with $q$ in place of $p_0$, and this is eventually \eqref{eq:Phifu-5}, completing the proof. 
\end{proof}

\begin{remark}
It follows from the proof that one can remove the assumption $I_\Phi<\infty$ at the cost of obtaining estimates of the form \eqref{eq:Nugu}. That is, in place of \eqref{eq:Phifu-4} we would get
\begin{equation*}
\int_{\Sigma}\Phi(f\,u)\,v\,d\mu
\leq 
\min\{1, C_0\}\,
\int_{\Sigma}\Phi((\max\{1, C_0\}\,g\,u)\,v\,d\mu,\qquad(f,g) \in \F,
\end{equation*}
with $C_0=	2^{3+\frac2{p_0'}}\,\Psi(2\,\mathcal{N}_1^{p_0-1} \mathcal{N}_2)$.  One can also formulate analogous versions of \eqref{eq:Phifu-4-vv}, \eqref{eq:Phifu-4:1},  \eqref{eq:Phifu-4-vv:1}, and \eqref{eq:Phifu-5}. Details are left to the reader. 

\end{remark}

\begin{proof}[Proof of Proposition~\ref{prop:extrapol-maintool:mod}]
	Note that \eqref{eq:embedding:f-g:mod1}, \eqref{eq:embedding:f-g:mod2} (resp.~\eqref{eq:embedding:f-g:1:mod}) with $g:=f$ and $\theta=1$ immediately imply \eqref{eq:embedding:f:mod} (resp.~\eqref{eq:embedding:f-g:1:mod}). Thus, we just need to obtain \eqref{eq:embedding:f-g:mod1}, \eqref{eq:embedding:f-g:mod2}, and \eqref{eq:embedding:f-g:1:mod}. To do that,  we generalize some of the ideas of \cite[Theorem~4.15]{CMP}. Fix $f, g\in\M_\mu$ with $\rho_v^{\Phi}(f\,u), \rho_v^{\Phi}(g\,u)<\infty $. We start with (a), that is, we fix $p_0>1$ and assume \eqref{eq:Phifu-1} and \eqref{eq:Phifu-2} for some finite constants $\mathcal{N}_1$ and $\mathcal{N}_2$. For any $h_1,h_2\in\M_\mu$ with $h_1, h_2>0$ $\mu$-a.e., $\rho_v^{\Phi}(h_1\,u), \rho_v^{\overline{\Phi}}(h_2\,u^{-1})<\infty $, we define the Rubio de Francia iteration algorithms as:
	\begin{equation*}
		\mathcal{R}h_1 :=\sum_{k=0}^{\infty}\frac{M_{\B}^k h_1}{2^k \mathcal{N}_1^k} \qquad\text{and}\qquad
		\mathcal{R}'h_2 :=\sum_{k=0}^{\infty}\frac{(M'_{\B, v})^k h_2}{2^k \mathcal{N}_2^k},
	\end{equation*} 
	where $M_{\B}^0$ and $(M'_{\B, v})^0$ denote the identity operator, while for every $k\ge 1$, we write $M_{\B}^k=M_{\B}\circ \dots\circ M_{\B}$ for the $k$-th iteration of $M_{\B}$ and	$(M'_{\B,v})^k=M'_{\B,v}\circ \dots\circ M'_{\B,v}$ for the $k$-th iteration of $M'_{\B,v}$.
	We claim that the following hold: 
	\begin{alignat}{2}
		h_1 &\leq \mathcal{R}h_1\qquad\qquad &	h_2 &\leq \mathcal{R}' h_2, 
		\label{RdF1:mod}
		\\[4pt]
		\left[\mathcal{R}h_1\right]_{A_{1, \B}}&\leq 2\, \mathcal{N}_1, \qquad\qquad &			[(\mathcal{R}' h_2)\, v]_{A_{1, \B}}&\leq 2\, \mathcal{N}_2.
		\label{RdF3:mod}
	\end{alignat}
	Indeed, \eqref{RdF1:mod} is an immediate consequence of the definitions, \eqref{eq:Phifu-1}, and \eqref{eq:Phifu-2}.  
	To show \eqref{RdF3:mod}, using that $h_1, h_2$ are positive almost everywhere, we can obtain from \eqref{RdF1:mod}
	\begin{equation*}
		\mathcal{R}h_1(x) \ge h_1(x)>0, \quad \mu \text{-a.e.~} x \in \Sigma,
		\quad\text{and}\quad
		\mathcal{R}'h_2(x) \ge h_2(x)>0, \quad \mu \text{-a.e.~} x \in \Sigma.
	\end{equation*} 
	Besides, \eqref{eq:Young-ineq}, \eqref{eq:Phifu-1}, Lemma~\ref{lem:NN}, and the fact that $v(B)<\infty$ for every $B\in\B$  yield that for every $B\in\B$
	\begin{align*}
		\int_{B} (\mathcal{R}h_1)\,u\, v\,d\mu 
		&=
		\sum_{k=0}^{\infty}\frac{1}{2^k\, \mathcal{N}_1^k} 	\int_{B} (M_{\B}^k h_1)\,u\,dv
		\\
		&\le		
		\sum_{k=0}^{\infty}\frac{1}{2^k\, \mathcal{N}_1^k} 	\bigg( \int_{\Sigma} \Phi((M_{\B}^k h_1)\,u)\,dv+\int_{B} \overline{\Phi}(1)\,dv\bigg)
		\\
		&\le
		\sum_{k=0}^{\infty}\frac{1}{2^k\, \mathcal{N}_1^k} 	\big( \mathcal{N}_1^k\, \rho_v^{\Phi}(h_1\,u) + \overline{\Phi}(1)\,v(B)\big)
				\\
		&=		
		2\,\rho_v^{\Phi}(h_1\,u)+ 2\,\overline{\Phi}(1)\,v(B)
	<\infty
	\end{align*}
	and
		\begin{align*}
		\int_{B} (\mathcal{R}'h_2)\,u^{-1}\, v\,d\mu 
		&=
		\sum_{k=0}^{\infty}\frac{1}{2^k\, \mathcal{N}_2^k} 	\int_{B} ((M'_{\B,v})^k h_2)\,u^{-1}\,dv
		\\
		&\le		
		\sum_{k=0}^{\infty}\frac{1}{2^k\, \mathcal{N}_2^k} 	\bigg( \int_{\Sigma} \overline{\Phi}(((M'_{\B,v})^k h_2)\,u^{-1})\,dv+\int_{B} \Phi(1)\,dv\bigg)
		\\
		&\le
		\sum_{k=0}^{\infty}\frac{1}{2^k\, \mathcal{N}_2^k} 	\big( \mathcal{N}_2^k\, \rho_v^{\overline{\Phi}}(h_2\,u^{-1}) + \Phi(1)\,v(B)\big)
		\\
		&=		
		2\, \rho_v^{\overline{\Phi}}(h_2\,u^{-1})+ 2\,\Phi(1)\,v(B)
		<\infty. 
	\end{align*}
	These and the  fact that $u$ and $v$ are $\B$-weights readily imply that $\mathcal{R}h_1, \mathcal{R}'h_2<\infty$ $\mu$-a.e.~in $\Sigma_\B$. Hence, $\mathcal{R}h_1$ and $\mathcal{R}'h_2$ are $\B$-weights. Moreover,
	\[
	M_{\B}(\mathcal{R} h_1) 
	\leq 
	\sum_{k=0}^{\infty} \frac{M_{\B}^{k+1} h_1}{2^k \mathcal{N}_1^k} 
	\le  
	2\,\mathcal{N}_1\, \mathcal{R} h_1,
	\quad\text{ and }\quad
		M'_{\B,v}(\mathcal{R}' h_2) 
		\leq 
		\sum_{k=0}^{\infty} \frac{(M'_{\B, v})^{k+1} h_2}{2^k \mathcal{N}_2^k} 
		\leq 2\,\mathcal{N}_2\, \mathcal{R}'h_2.
\]
	These readily imply \eqref{RdF3:mod}. 

On the other hand we claim that for any $0<\theta\le 1$
\begin{equation}\label{RdF2:mod:new1}
	\int_\Sigma \mathcal{R}h_1 \,\mathcal{R}'h_2\,dv
	\le
	4\,	\theta\,\big(\rho_v^{\Phi}(\theta^{-1}\,h_1\,u) + \rho_v^{\overline{\Phi}}(h_2\,u^{-1})\big)
\end{equation}
and
\begin{equation}\label{RdF2:mod:new2}
	\int_\Sigma h_1 \,\mathcal{R}'h_2\,dv
	\le
	2\,\theta\,\big(\rho_v^{\Phi}(\theta^{-1}\,h_1\,u) + \rho_v^{\overline{\Phi}}(h_2\,u^{-1})\big).
\end{equation}
To see these we use Young's inequality \eqref{eq:Young-ineq}, the sublinearity of $M_\B$ and $M'_{\B,v}$, \eqref{eq:Phifu-1},  \eqref{eq:Phifu-2}, and Lemma~\ref{lem:NN} to obtain
\begin{align*}
&\int_\Sigma \mathcal{R}h_1 \,\mathcal{R}'h_2\,dv
=
\theta\,\sum_{j=0}^\infty\sum_{k=0}^\infty \frac1{2^j \mathcal{N}_1^j}\,\frac1{2^k \mathcal{N}_2^k} 
\int_\Sigma \theta^{-1}\,(M_{\B}^{j} h_1)\,u\,((M'_{\B,v})^{k} h_2)\,u^{-1}\,dv  
\\
&
\qquad\le
\theta\,\sum_{j=0}^\infty\sum_{k=0}^\infty \frac1{2^j \mathcal{N}_1^j}\,\frac1{2^k \mathcal{N}_2^k} 
\bigg(
\int_\Sigma \Phi\big(\theta^{-1}\,(M_{\B}^{j} h_1)\,u\big)\,dv  
+
\int_\Sigma \Phi\big(((M'_{\B,v})^{k} h_2)\,u^{-1}\big)\,dv  
\bigg)
\\
&
\qquad\le
\theta\,
\sum_{j=0}^\infty\sum_{k=0}^\infty \frac1{2^j \mathcal{N}_1^j}\,\frac1{2^k \mathcal{N}_2^k} 
\bigg(
\mathcal{N}_1^j\,\int_\Sigma \Phi\big(\theta^{-1}\,h_1\,u\big)\,dv  
+
\mathcal{N}_2^k\,\int_\Sigma \Phi\big(h_2\,u^{-1}\big)\,dv  
\bigg)
\\
&
\qquad\le
4\,\theta\,\big(\rho_v^{\Phi}(\theta^{-1}\,h_1\,u)  
+
\rho_v^{\overline{\Phi}}(h_2\,u^{-1})\big)
\end{align*}
and
\begin{align*}
&\int_\Sigma h_1 \,\mathcal{R}'h_2\,dv
	=
	\theta\,\sum_{k=0}^\infty \frac1{2^k \mathcal{N}_2^k} 
	\int_\Sigma \theta^{-1}\,h_1\,u\,((M'_{\B,v})^{k} h_2)\,u^{-1}\,dv  
	\\
	&
	\qquad\le
	\theta\,\sum_{k=0}^\infty \frac1{2^k \mathcal{N}_2^k} 
	\bigg(
	\int_\Sigma \Phi(\theta^{-1}\,h_1\,u)\,dv  
	+
	\int_\Sigma \Phi\big(((M'_{\B,v})^{k} h_2)\,u^{-1}\big)\,dv  
	\bigg)
	\\
	&
	\qquad\le
	\theta\,
	\sum_{k=0}^\infty \frac1{2^k \mathcal{N}_2^k} 
	\bigg(\int_\Sigma \Phi\big(\theta^{-1}\,h_1\,u\big)\,dv  
	+
	\mathcal{N}_2^k\,\int_\Sigma \Phi\big(h_2\,u^{-1}\big)\,dv  
	\bigg)
	\\
	&
	\qquad\le
	2\,\theta\,\big(\rho_v^{\Phi}(\theta^{-1}\,h_1\,u)  
	+
	\rho_v^{\overline{\Phi}}(h_2\,u^{-1})\big).
\end{align*}

Once we have obtained \eqref{RdF1:mod}, \eqref{RdF3:mod}, \eqref{RdF2:mod:new1}, and \eqref{RdF2:mod:new2} we proceed to show \eqref{eq:embedding:f-g:mod1} and \eqref{eq:embedding:f-g:mod2}. Set $h_1:= g$ and $h_2=\frac{\Phi(f\,u)}{f}\,\mathbf{1}_{\{f>0\}}$ and note that $h_1, h_2\ge 0$ $\mu$-a.e.~verify 
\begin{align}\label{avvasrver}
	\rho_v^{\Phi}(h_1\,u)= \rho_v^{\Phi}(g\,u)<\infty \qquad\text{and}\qquad \rho_v^{\overline{\Phi}}(h_2\,u^{-1})\le \rho_v^{\Phi}(f\,u)<\infty,
\end{align}
where we have used that $\Phi$ and $\overline{\Phi}$ are Young functions and \eqref{eq:Young-2}. Since $\rho_v^{\Phi}(f\,u)$, $\rho_v^{\Phi}(g\,u)>0$ then $\rho_v^{\Phi}(h_1\,u),  \rho_v^{\overline{\Phi}}(h_2\,u^{-1})>0$ and we can then apply Lemma \ref{lem:Phh} with $\varepsilon=\frac12$ and both $\Phi$ and $\overline{\Phi}$ to find
$\widetilde{h}_1,\widetilde{h}_2> 0$ $\mu$-a.e.~such that 
\begin{align}\label{eq:hhh:mod1}
	h_1 \le 2\,\widetilde{h}_1, \qquad 		 \rho_v^{\Phi}(\widetilde{h}_1\,u) \le  \rho_v^{\Phi}(h_1\,u)=\rho_v^{\Phi}(g\,u) 
\end{align} 
and
\begin{align}\label{eq:hhh:mod2}
	h_2 \le 2\, \widetilde{h}_2,\qquad 
	\rho_v^{\overline{\Phi}}(\widetilde{h}_2\,u^{-1}) \le  \rho_v^{\Phi}(h_2\,u^{-1})\le \rho_v^{\Phi}(f\,u).  
\end{align} 
With these in mind we note that
	\begin{align}\label{eq:fhu:mod}
		\rho_v^{\Phi}(f\,u)
		= 
		\int_{\Sigma} \Phi(f\,u)\,dv
		=
		\int_{\Sigma} f\, h_2\,u\,dv
		\le
		2\, \int_{\Sigma} f\, \widetilde{h}_2\,u\,dv
		.
	\end{align}
	Set
	\begin{equation*}
		w:=(\mathcal{R}\widetilde{h}_1)^{1-p_0} (\mathcal{R}'\widetilde{h}_2)\, v. 
	\end{equation*}
	From \eqref{RdF3:mod} and \eqref{eq:Ap-product}, we deduce that $w \in A_{p_0, \B}$ with
	\begin{equation*}
		[w]_{A_{p_0}} \leq [\mathcal{R}\widetilde{h}_1]_{A_{1, \B}}^{p_0-1} [(\mathcal{R}'\widetilde{h}_2)\,v]_{A_{1, \B}}
		\leq 2^{p_0} \mathcal{N}_1^{p_0-1} \mathcal{N}_2.
	\end{equation*}
	Thus, for any $\theta\in (0,1]$, by \eqref{eq:fhu:mod},  Hölder's' inequality, \eqref{RdF2:mod:new1}, \eqref{eq:hhh:mod1}, and \eqref{eq:hhh:mod2}
	\begin{align*}
	\rho_v^{\Phi}(f\,u)
		&\le 
		2\,\int_{\Rn} f\, (\mathcal{R}\widetilde{h}_1)^{-\frac{1}{p'_0}}\, 
		(\mathcal{R}\widetilde{h}_1)^{\frac{1}{p'_0}}\, \mathcal{R}'\widetilde{h}_2\, dv
		\\
		&\leq 
		2\,\left(\int_{\Sigma} f^{p_0}\,(\mathcal{R}\widetilde{h}_1)^{1-p_0}\, \mathcal{R}'\widetilde{h}_2\, dv\right)^{\frac{1}{p_0}} 
		\left(\int_{\Sigma} \mathcal{R}\widetilde{h}_1\,\mathcal{R}'\widetilde{h}_2\,dv\right)^{\frac{1}{p_0'}} 
		\\
		&\le 
		2^{1+\frac2{p_0'}}\,\theta^{\frac1{p_0'}}\big(\rho_v^{\Phi}(\theta^{-1}\,\widetilde{h}_1\,u) + \rho_v^{\overline{\Phi}}(\widetilde{h}_2\,u^{-1})\big)^{\frac1{p_0'}}\,\|f\|_{L^{p_0}(\Sigma, w)}
			\\
		&\le 
		2^{1+\frac2{p_0'}}\,\theta^{\frac1{p_0'}}\big(\rho_v^{\Phi}(\theta^{-1}\,g\,u) + 	\rho_v^{\Phi}(f\,u)\big)^{\frac1{p_0'}}\,\|f\|_{L^{p_0}(\Sigma, w)}. 
	\end{align*}
	This shows \eqref{eq:embedding:f-g:mod1}. On the other hand, $g=h_1\le 2\, \widetilde{h}_1\le 2\, \mathcal{R}\widetilde{h}_1$ by \eqref{eq:hhh:mod2} and  \eqref{RdF1:mod}. Thus, \eqref{RdF2:mod:new2}, \eqref{eq:hhh:mod1}, and \eqref{eq:hhh:mod2} yield
	\begin{multline*}
		\|g\|_{L^{p_0}(\Sigma,w_0)} 
=\left(\int_{\Sigma} g^{p_0}\,(\mathcal{R}\widetilde{h}_1)^{1-p_0}\,
		(\mathcal{R}' \widetilde{h}_2)\,v\, d\mu\right)^{\frac{1}{p_0}}
\leq 2^{\frac1{p_0'}} \left(\int_{\Sigma} g\, (\mathcal{R}'\widetilde{h}_2)\,v\,d\mu\right)^{\frac{1}{p_0}}
		\\
\le 
		2\,	\theta^{\frac{1}{p_0}} \big(\rho_v^{\Phi}(\theta^{-1}\,\widetilde{h}_1\,u) + \rho_v^{\overline{\Phi}}(\widetilde{h}_2\,u^{-1})\big)^{\frac1{p_0}}
\le 
		2\,	\theta^{\frac{1}{p_0}} \big(\rho_v^{\Phi}(\theta^{-1}\,g\,u) + \rho_v^{\Phi}(f\,u)\big)^{\frac1{p_0}}.
	\end{multline*} 
	This shows \eqref{eq:embedding:f-g:mod2} and completes the proof of the case $p_0>1$.
	
	Let us next deal with (b), that is, we consider the case $p_0=1$ and only assume that there exists $\mathcal{N}_2<\infty$ so that \eqref{eq:Phifu-2} holds. We follow the proof of the previous case and this time we do not use $\mathcal{R}$, that is, \eqref{eq:Phifu-1} needs not to be assumed (hence $\mathcal{N}_1$ could be infinity).  With the notation above, let us set $w:=\mathcal{R}'(\widetilde{h}_2)\,v$. The second estimate in \eqref{RdF3:mod} implies $w \in A_{1, \B}$ and $[w]_{A_{1, \B}} \le 2\,\mathcal{N}_2$. Additionally, \eqref{eq:fhu:mod} and the first estimate in \eqref{RdF1:mod} give
	\begin{align*}
		\rho_v^{\Phi}(f\,u)
	\le
	2\, \int_{\Sigma} f\, \widetilde{h}_2\,u\,dv	
	=2\,  \|f\|_{L^1(\Sigma, w)}. 
	\end{align*}
	On the other hand, for any $\theta\in (0,1]$, by \eqref{RdF2:mod:new2} and \eqref{eq:hhh:mod2}
		\begin{multline*}
		\|g\|_{L^1(\Sigma, w)} 
		= 
		\int_{\Sigma} g\, (\mathcal{R}' \widetilde{h}_2)\,v\, d\mu
		\le
		2\,\theta\,\big(\rho_v^{\Phi}(\theta^{-1}\,g\,u)  
		+
		\rho_v^{\overline{\Phi}}(\widetilde{h}_2\,u^{-1})\big)
		\\
	\le 
		2\,\theta\,\big(\rho_v^{\Phi}(\theta^{-1}\,g\,u)  + \rho_v^{\Phi}(f\,u^{-1})\big).
	\end{multline*}
	These prove \eqref{eq:embedding:f-g:1:mod} and the proof is then complete. 
\end{proof}

The next goal is to prove an extrapolation theorem for $A_{\infty, \B}$ weights. 

\begin{theorem}\label{thm:PhiAi}
	Suppose that $(\Sigma, \mu)$ is a non-atomic $\sigma$-finite measure space with $\mu(\Sigma)>0$. Let $\B$ be a Muckenhoupt basis and let $\F$ be a family of extrapolation pairs. Let $u$ and $v$ be $\B$-weights on $(\Sigma, \mu)$ such that $v(B)<\infty$ for every $B\in\B$, and let $\Phi$ be a Young function such that $\Phi\in\Delta_2$ \textup{(}equivalently, $I_\Phi<\infty$\textup{)} . Let $M_{\B}$ denote the Hardy-Littlewood maximal function on $(\Sigma,\mu)$ associated with $\B$ and let $M'_{\B, v}h := M_{\B}(h\,v)/v$ for each $h \in \M_{\mu}$. Assume that there exists $\mathcal{N}<\infty$ such that
	\begin{align}\label{eq:PhiAi-1} 
		\int_{\Sigma} \overline{\Phi}((M'_{\B,v} f)\, u^{-1})\, v\, d\mu 
		\le \mathcal{N} \int_{\Sigma} \overline{\Phi}(f\, u^{-1})\, v\, d\mu,\quad\forall f \in \M_{\mu}. 
	\end{align}
	If for some $p_0 \in (0,\infty)$ and for every $w \in A_{\infty, \B}$, 
	\begin{equation}\label{eq:PhiAi-2}
		\|f\|_{L^{p_0}(\Sigma,w)} \leq \Psi([w]_{A_{\infty,\B}}) \|g\|_{L^{p_0}(\Sigma,w)}, \quad(f,g) \in \F, 
	\end{equation} 
	where $\Psi:[1,\infty) \to [1,\infty)$ is a non-decreasing function, then for every $p\in (0, \infty)$, 
	\begin{equation}\label{eq:PhiAi-3}
		\int_{\Sigma}\Phi(f^p\, u)\,v\,d\mu \leq C\,\Psi(2\, \mathcal{N})^p 
		\int_{\Sigma}\Phi(g^p\, u)\,v\,d\mu,\qquad(f,g) \in \F, 
	\end{equation}
	and	for every $q \in (0, \infty)$, 
	\begin{equation}\label{eq:PhiAi-4}
		\int_{\Sigma}\Phi \bigg(\Big(\sum_j f_j^q \Big)^{\frac1q}\,u \bigg)\,v\,d\mu 
		\leq C\, \int_{\Sigma}\Phi \bigg(\Big(\sum_j g_j^q \Big)^{\frac1q}\,u \bigg)\,v\,d\mu,\qquad \{(f_j, g_j)\} \subset \F. 
	\end{equation}
\end{theorem}

\begin{proof} 
We recall that in the proof of Theorem~\ref{thm:BFSAi} we have already obtained that \eqref{eq:PhiAi-2} implies \eqref{qwerwege:22} and hence for every $p\in (0,\infty)$ and for every $w\in A_{1,\B}$ one has
	\begin{equation}\label{qwerwege:2:mod}
		\|f^p\|_{L^{1}(\Sigma,w)} \leq C\, \|g^p\|_{L^{1}(\Sigma,w)}, \quad(f,g) \in \F.
	\end{equation}
We are now ready to invoke Theorem~\ref{thm:PhiAp} part (a) (again the reader may  find convenient to introduce the family $\F_p$ consisting of the pairs $(f^p, g^p)$ with $(f,g)\in\F$) to conclude that \eqref{eq:Phifu-4:1} yields \eqref{eq:PhiAi-3} as desired. To obtain \eqref{eq:PhiAi-4} we observe that as in the proof of Theorem~\ref{thm:BFSAi} we have shown that \eqref{qwerwege:3} holds for every $0<q<\infty$ and every $w\in A_{\infty,\B}$. The same argument we have used to show that \eqref{qwerwege:2:mod} yields  \eqref{eq:PhiAi-3} can be then repeated to see that \eqref{qwerwege:3} implies \eqref{eq:PhiAi-4}. This completes the proof. 
\end{proof}

\section{Applications}\label{sec:applications}
The goal of this section is to present some applications of the extrapolation results obtained above. We will see that those can be applied to not only study the well-posedness of the Dirichlet problem for elliptic systems, but also to establish some weighted inequalities for some significant operators and the associated commutators in various situations, some explored and some unexplored before. 

\subsection{The Dirichlet problem for elliptic systems in the upper half-space}

We fix $d,M\in{\N}$ with $d\geq 2$ and we consider a homogeneous, constant (complex) coefficient, $M\times M$ second-order system $L$ in $\R^d$. Specifically, for every vector-valued function $u=(u_\beta)_{1\leq\beta\leq M}$, we write 
\begin{equation}\label{eq:Lu}
Lu:=\left(a_{jk}^{\alpha\beta}\partial_j\partial_k u_\beta\right)_{1\leq\alpha\leq M},
\end{equation} 
where $a_{jk}^{\alpha\beta}\in\mathbb{C}$ for every $j,k\in\{1,\dots,d\}$ and $\alpha,\beta\in\{1,\dots,M\}$. Here and elsewhere in this section, we use the convention of summation over repeated indices. We also assume that $L$ is {\tt elliptic}, in the sense that there exists a constant $\kappa_0>0$ such that the following Legendre-Hadamard condition holds: 
\begin{equation}\label{eq:elliptic} 
\begin{array}{c}
\Re\left[a_{jk}^{\alpha\beta}\xi_j\xi_k\overline{\zeta_\alpha}\zeta_\beta\right]
\geq\kappa_0\left|\xi\right|^2\left|\zeta\right|^2\,\,\text{ for every}
\\[10pt]
\xi=\left(\xi_j\right)_{1\leq j\leq d}\in\R^d
\,\,\text{ and }\,\,\zeta=\left(\zeta_\alpha\right)_{1\leq\alpha\leq M}\in\mathbb{C}^M.
\end{array}
\end{equation} 
In the scalar case (i.e. $M=1$), elliptic operators include the Laplacian $\Delta=\sum\limits_{j=1}^d \partial_j^2$ or, more generally, operators of the form ${\rm div}(A\nabla)$, where $A=(a_{jk})_{1\leq j,k\leq d}\in\mathbb{C}^{d\times d}$ satisfies the scalar version of \eqref{eq:elliptic}, that is,
$
\inf_{\xi\in \mathbb{S}^{d-1}}{\rm Re}\,\big[a_{rs}\xi_r\xi_s\bigr]>0,
$, 
where $\mathbb{S}^{d-1}$ stands for the unit sphere in $\R^d$. Regarding the case $M>1$, an example of an elliptic system is the complex version of the Lamé system of elasticity in $\R^d$, given by 
\begin{equation*}
L:=\mu\Delta+(\lambda+\mu)\nabla{\rm div},
\end{equation*} 
where the constants $\lambda,\mu\in{\mathbb{C}}$ (called Lamé moduli in the literature) satisfy  $\Re\mu>0$ and $\Re(2\mu+\lambda)>0$,
conditions which are indeed equivalent to \eqref{eq:elliptic}.

We are interested in showing well-posedness for boundary value problems in the upper half-space. With this purpose in mind, we fix $d\in\mathbb{N}$ with $d\geq 2$ and denote the upper half-space in $\R^d$ as 
\begin{equation*}
\R^d_{+}:=\{(x',t) \in \R^d:\,x'\in\R^{d-1},t>0\}.
\end{equation*}
We also identify the boundary $\partial\R^d_{+}$ with $\R^{d-1}$ via $\partial \R^d_{+}\ni(x',0)\equiv x'\in\R^{d-1}$. The cone with vertex at $x'\in\R^{d-1}$ and aperture $\kappa>0$ is given by
\begin{equation*}
\Gamma_{\kappa}(x'):=\{(y',t)\in\R^d_{+}:|x'-y'|<\kappa t\}.
\end{equation*}
Given a vector-valued function $u:\R^d_{+}\to\mathbb{C}^M$, we define its nontangential boundary trace (whenever it is meaningful) as  
\begin{equation*}
\big(u\big|^{{}^{\kappa-{\rm n.t.}}}_{\partial\R^d_{+}}\big)(x'):=
\lim_{\substack{\R^d_{+}\ni y\to(x',0)\\y\in\Gamma_{\kappa}(x')}} u(y),
\quad x' \in \R^{d-1},
\end{equation*}
and the nontangential maximal function of $u$ as
\begin{equation*}
\mathcal{N}_{\kappa}u(x') := \esssup\{|u(y)|:y\in\Gamma_{\kappa}(x')\},
\quad x' \in \R^{d-1}.
\end{equation*}

In order to frame this setting in our general discussion, we let $\Sigma=\R^{d-1}$ and let $\mu=\L^{d-1}$ be the Lebesgue measure in $\R^{d-1}$. In what follows we will implicitly assume that $\L^{d-1}$ is the underlying measure and write $\mathbb{M}$ (in place of $\mathbb{M}_{\L^{d-1}}$) to denote the set of Lebesgue measurable functions in $\R^{d-1}$. Analogously, when we write that some condition occurs a.e.~we mean that it does $\L^{d-1}$-a.e. We let $\B$ denote the collection of all euclidean balls $\R^{d-1}$, in which case $M_\B$ is the classical Hardy-Littlewood maximal function (with respect to uncentered balls) and we will simply write $M$. Of course, one can equivalently work with cubes in place of balls as the corresponding maximal function is pointwise equivalent to $M$. In this context $w$ is a $\B$-weight (we will simply say that $w$ is a weight) if $w\in\mathbb{M}$ with $0<w<\infty$ a.e. In this fashion, $A_{p,\B}$ or $RH_{s,\B}$  are the classical Muckenhoupt and reverse Hölder classes and will be denoted by $A_p$ and $RH_s$, respectively.

For every elliptic system as above there exists an associated Agmon-Douglis-Nirenberg Poisson kernel in $\R^d_+$  \cite{ADN1, ADN2}, see also \cite{KMR2, Sol}. This, \cite[Theorems 2.4 and 3.1]{K-MMMM}, and \cite[Theorem 1.1]{MMMMM18} (see also \cite[Theorem 3.2]{K-MMMM}) allow us to formulate the following result:

\begin{theorem}\label{thm:Poisson}
Let $L$ be a homogeneous, second-order, constant complex coefficient, elliptic $M\times M$ system in $\R^d$ as in \eqref{eq:Lu}-\eqref{eq:elliptic}.
Then the following statements hold:
\begin{list}{\textup{(\theenumi)}}{\usecounter{enumi}\leftmargin=1cm \labelwidth=1cm \itemsep=0.2cm 
		\topsep=.2cm \renewcommand{\theenumi}{\alph{enumi}}}

\item There exists a matrix-valued function 
$P^L=(P^L_{\alpha\beta})_{1\leq\alpha,\beta\leq M}:\R^{d-1}\to\mathbb{C}^{M\times M}$, 
called the {\tt Poisson kernel} for $L$ in $\R^d_{+}$, such that 
$P^L\in{\mathscr{C}}^\infty(\R^{d-1})$, and for some constant $C\in(0,\infty)$
\begin{equation}\label{eq:Poisson-decay}
|P^L(x')|\leq\frac{C}{(1+|x'|^2)^{d/2}} \qquad\text{and}\qquad \int_{\R^{d-1}}P^L(x')\,dx'=I_{M\times M},
\end{equation} 
where $I_{M\times M}$ denotes the $M\times M$ identity matrix. 

\item The function $K^L(x',t):=P_t^L(x')= t^{1-d}P^L(x'/t)$ satisfies $K^L\in{\mathscr{C}}^\infty\big(\overline{\R^d_{+}}\setminus \{0\}\big)$. Furthermore, if we set $K^L:=\big(K^L_{\alpha\beta}\big)_{1\leq\alpha,\beta\leq M}$ then
\begin{equation}\label{eq:Poisson-LK0}
LK^L_{\cdot\beta}=0\,\,\text{ in }\,\,\R^d_{+} \,\,\text{ for every }\,\,\beta\in\{1,\dots,M\},
\end{equation}
where $K^L_{\cdot\beta}:=\big(K^L_{\alpha\beta}\big)_{1\leq\alpha\leq M}$ is the $\beta$-th column in $K^L$. 

\item Assume that $f=(f_\beta)_{1\le\beta\le M}:\R^{d-1}\to\C^M$ is a Lebesgue measurable function such that 
\begin{equation}\label{zdfgwseg}
\int_{\R^{d-1}} \frac{|f(x')|}{1+|x'|^d} dx' < \infty,
\end{equation}
and define $u(x',t):=(P^L_t\ast f)(x')$ with $P^L$ as in \textup{(a)}. Then $u$ is meaningfully defined via an absolutely convergent integral, and for every aperture $\kappa>0$, it satisfies 
\begin{equation}\label{eq:Poisson-u-prop} 
u\in\mathscr{C}^\infty(\R^d_{+},{\mathbb{C}}^M),\quad\, 
Lu=0\,\,\text{ in }\,\,\R^d_{+},\quad\,\,
u\big|^{{}^{\kappa-{\rm n.t.}}}_{\partial\R^d_{+}}=f\text{ a.e.~on }\,\R^{d-1},
\end{equation}
and there is a constant $C\in(0,\infty)$ such that 
\begin{equation}\label{eq:Poisson-u-prop2} 
|f(x')|\leq\mathcal{N}_{\kappa}u(x')\leq C\,M f(x'), \qquad \text{for a.e.~}x'\in\R^{d-1}.
\end{equation}

\item Assume that $u\in\big[{\mathscr{C}}^{\infty}(\R^d_{+})\big]^M$ satisfies $Lu=0$ in $\R^d_{+}$ and 
\begin{equation*}
	\int_{\R^{d-1}}\frac{({\mathcal{N}}_{\kappa}u)(x')}{1+|x'|^{d-1}}\,dx'<\infty.
\end{equation*}
Then, 
\[
\begin{cases}
u\big|^{{}^{\kappa-{\rm n.t.}}}_{\partial{\R}^d_{+}}\ \text{exists a.e.~on}\ \R^{d-1},
\\[12pt]
\displaystyle u\big|^{{}^{\kappa-{\rm n.t.}}}_{\partial{\R}^d_{+}} \in L^1\Big(\R^{d-1},\frac{dx'}{1+|x'|^{d-1}}\Big),
\\[14pt]
u(x',t)=\Big(P^L_t\ast\big(u\big|^{{}^{\kappa-{\rm n.t.}}}_{\partial \R^d_{+}}\big)\Big)(x')\ \text{for each}\ (x',t)\in \R^d_{+}.
\end{cases}
\]
 In particular, if $u\big|^{{}^{\kappa-{\rm n.t.}}}_{\partial \R^d_{+}}=0$ a.e.~on $\R^{d-1}$, then $u\equiv 0$.

\end{list}
\end{theorem}

With Theorem~\ref{thm:Poisson} in hand, we are now ready to prove well-posedness for boundary value problems.

\begin{theorem}\label{thm:BVP-BFS}
Let $L$ be an elliptic constant complex coefficient second-order $M\times M$ system as in \eqref{eq:Lu}-\eqref{eq:elliptic}. Let $v$ and $w$ be weights on $(\R^{d-1},\L^{d-1})$ such that $v \in L^1_{\loc}(\R^{d-1},\L^{d-1})$. Fix an aperture parameter $\kappa>0$ and a Banach function space $\X_v$ over $(\R^{d-1}, v)$, where as usual $dv=v\,d \L^{d-1}$.   Let $M$ denote the Hardy-Littlewood maximal function on $(\R^{d-1},\L^{d-1})$ and let $M'_{v}h:=M(h\,v)/v$ for each $h \in\M$. Assume that there is $C_0<\infty$ so that 
\begin{align}
\label{eq:B1} \norm{(Mh)\, w}_{\X_v} &\leq C_0\, \|h\,w\|_{\X_v}, \quad \forall\,h \in\M, 
\\
\label{eq:B2} \|(M'_{v} h)\, w^{-1} \|_{\X'_v} &\leq C_0\, \|h\,w^{-1}\|_{\X'_v}, \quad \forall\,h \in\M.
\end{align}
Then, the Dirichlet Problem 
\begin{equation}\label{eq::BVP-BFS}
\left\lbrace
\begin{array}{l}
u\in\mathscr{C}^{\infty}(\R^d_{+},\mathbb{C}^M),
\\[4pt]
Lu=0\,\,\text{ in }\,\,\R^d_{+},
\\[4pt] 
(\mathcal{N}_{\kappa}u)\, w\in\X_v,
\\[4pt]
\restr{u}{\partial\R^d_{+}}^{{}^{\kappa-{\rm n.t.}}}=f
\,\,\text{ on }\,\,\R^{d-1},\quad f\,w\in\X_v.
\end{array}
\right.
\end{equation}
is well-posed. More specifically, there is a unique solution and it is given by 
\begin{equation}\label{eq:BVP-BFS-sol}
u(x',t)=(P_t^L*f)(x'), \qquad(x',t)\in{\R}^d_{+},
\end{equation}
where $P^L$ denotes the Poisson kernel for $L$ in $\R^d_+$ from Theorem~\ref{thm:Poisson}. Furthermore, there is a constant $C\in[1,\infty)$ such that 
\begin{equation}\label{eq:BVP-BFS-estimate}
\|f\,w\|_{\X_v}\le \|(\mathcal{N}_{\kappa}u)\,w\|_{\X_v} \leq C\, \|f\,w\|_{\X_v}.
\end{equation}
\end{theorem}

\begin{proof}
Let $p \in (1, \infty)$ and $\nu \in A_{p}$. For every $h \in \M$ write 
\begin{align}\label{eq:FG-1}
F := \Big(\int_{\R^{d-1}} \frac{|h(x')|}{1+|x'|^{d-1}}\,dx'\Big)\mathbf{1}_{B(0', 1)}  \qquad\text{and}\qquad G:=|h|, 
\end{align}
where $0'$ denotes the origin of $\mathbb{R}^{d-1}$.
It is not difficult to see that $M(\mathbf{1}_{B(0',1)})(x')\approx (1+|x'|^{d-1})^{-1}$ for every $x'\in\R^{d-1}$ (see for instance \cite[Lemma~2.1]{K-MMMM}). Using Hölder's inequality and \eqref{eq:Msharp}, we obtain that 
\begin{align}\label{eq:FG-2}
\|F\|_{L^p(\R^{d-1}, \nu)} 
&=\Big(\int_{\R^{d-1}} \frac{|h(x')|}{1+|x'|^{d-1}}\,dx'\Big) \, \nu(B(0', 1))^{\frac1p}
\nonumber\\
&\approx \Big(\int_{\R^{d-1}}|h(x')|M(\mathbf{1}_{B(0',1)})(x')\,dx'\Big) \, \nu(B(0', 1))^{\frac1p}
\nonumber\\
&\leq \|h\|_{L^p(\R^{d-1}, \nu)}\, \|M(\mathbf{1}_{B(0',1)})\|_{L^{p'}(\R^{d-1}, \nu^{1-p'})} \, \nu(B(0', 1))^{\frac1p}
\nonumber\\
&\leq C\, [\nu]_{A_{p}}\, \|h\|_{L^p(\R^{d-1}, \nu)} \, \nu(B(0', 1))^{\frac1p} \, \nu^{1-p'}(B(0', 1))^{\frac{1}{p'}} 
\nonumber\\
&\leq C\, [\nu]_{A_{p}}^{1+\frac1p}\, \|G\|_{L^p(\R^{d-1}, \nu)}.  
\end{align}
This together with Theorem \ref{thm:BFSAp} gives
\begin{align*}
\Big(\int_{\R^{d-1}} \frac{|h(x')|}{1+|x'|^{d-1}}\,dx'\Big) \, \|\mathbf{1}_{B(0', 1)}\, w\|_{\X_v}
=\|F\,w\|_{\X_v} \le C \|G\,w\|_{\X_v} = C \|h\,w\|_{\X_v}, 
\end{align*}
which, by \eqref{eq:BuXv}, immediately implies that 
\begin{align}\label{5q2fg5gg}
\int_{\R^{d-1}} \frac{|h(x')|}{1+|x'|^{d-1}}\,dx' < \infty, \qquad\text{for every $h \in \M$ with $\|h\,w\|_{\X_v}<\infty$}.
\end{align}

Given $f \in \M$ with $\|f\,w\|_{\X_v}<\infty$,  \eqref{5q2fg5gg} applied to $f$ immediately implies \eqref{zdfgwseg}. Hence, \eqref{eq:Poisson-u-prop} in Theorem~\ref{thm:Poisson} says that $u$ defined in \eqref{eq:BVP-BFS-sol} satisfies 
the first, second, and fourth condition in \eqref{eq::BVP-BFS}. Also, \eqref{eq:Poisson-u-prop2} and \eqref{eq:B1} yield
\begin{equation*}
	\|f\,w\|_{\X_v}\le \|(\mathcal{N}_{\kappa}u)\, w\|_{\X_v}
	\leq C\,\|(M f)\,w\|_{\X_v}
	\leq 
	C\,\|f\,w\|_{\X_v}.
\end{equation*}
Thus, the third condition in \eqref{eq::BVP-BFS}, and \eqref{eq:BVP-BFS-estimate}, hold. Finally, \eqref{5q2fg5gg} applied to $\mathcal{N}_{\kappa}u$ and Theorem~\ref{thm:Poisson} item (d) readily imply that $u$, the solution of  \eqref{eq::BVP-BFS}, is unique, hence it must be the one given in \eqref{eq:BVP-BFS-sol}.
\end{proof}

\begin{example}\label{ex:Lpvar:Dir}
Recalling the definition of  the variable Lebesgue space $L^{p(\cdot)}(\R^{d-1},\L^{d-1})$ in Example~\ref{ex:Lpvar} and for $L$ as above 
we obtain the following consequence. Assume that $1<p_-\le p_+<\infty$, $p(\cdot)\in{\rm LH}$, and $w\in A_{p(\cdot)}$ 
(cf.~\eqref{eq:Lpvar-pp}--\eqref{eq:Lpvar-Apvar}). Then \eqref{eq:Ap=AX} allows us to apply Theorem~\ref{thm:BVP-BFS} with $v\equiv 1$ and $\X_v=L^{p(\cdot)}(\R^{d-1},\L^{d-1})$ to obtain that the Dirichlet Problem 
\begin{equation}\label{eq::BVP-Lp-var}
	\left\lbrace
	\begin{array}{l}
		u\in\mathscr{C}^{\infty}(\R^d_{+},\mathbb{C}^M),
		\\[4pt]
		Lu=0\,\,\text{ in }\,\,\R^d_{+},
		\\[4pt] 
		(\mathcal{N}_{\kappa}u)\, w\in L^{p(\cdot)}(\R^{d-1},\L^{d-1}),
		\\[4pt]
		\restr{u}{\partial\R^d_{+}}^{{}^{\kappa-{\rm n.t.}}}=f
		\,\,\text{ on }\,\,\R^{d-1},\quad f\,w\in L^{p(\cdot)}(\R^{d-1},\L^{d-1}).
	\end{array}
	\right.
\end{equation}
is well-posed. More specifically, there is a unique solution, given by $u(x',t)=(P_t^L*f)(x')$, $(x',t)\in{\R}^d_{+}$, and satisfying 
\begin{equation}\label{eq:BVP-Lpvar-estimate}
	\|(\mathcal{N}_{\kappa}u)\,w\|_{L^{p(\cdot)}(\R^{d-1},\L^{d-1})} \leq C\, \|f\,w\|_{L^{p(\cdot)}(\R^{d-1},\L^{d-1})}.
\end{equation}
The case $w\equiv 1$ was obtained in \cite[Example~3]{K-MMMM}. 
\end{example}

\begin{example}\label{ex:RIBFS:Dir}
Let $\X$ be a rearrangement invariant Banach function space over $(\R^{d-1},\L^{d-1})$ such that $1<p_\X\le q_\X<\infty$. Assume that  $u^{p_{\X}}\, v \in A_{p_{\X}}$, $u^{q_{\X}}\, v \in A_{q_{\X}}$, and $v \in A_{\infty}$. Then for $L$ as above, \eqref{eq:intro-CMP-ri} allows us to apply Theorem~\ref{thm:BVP-BFS} with $\X_v=\X(v)$ (cf. \eqref{def:Xv:ri})  to obtain that the Dirichlet Problem 
\begin{equation}\label{eq::BVP-RIBFS}
	\left\lbrace
	\begin{array}{l}
		u\in\mathscr{C}^{\infty}(\R^d_{+},\mathbb{C}^M),
		\\[4pt]
		Lu=0\,\,\text{ in }\,\,\R^d_{+},
		\\[4pt] 
		(\mathcal{N}_{\kappa}u)\, w\in \X(v),
		\\[4pt]
		\restr{u}{\partial\R^d_{+}}^{{}^{\kappa-{\rm n.t.}}}=f
		\,\,\text{ on }\,\,\R^{d-1},\quad f\,w\in \X(v).
	\end{array}
	\right.
\end{equation}
is well-posed. More specifically, there is a unique solution, given by $u(x',t)=(P_t^L*f)(x')$, $(x',t)\in{\R}^d_{+}$, and satisfying 
\begin{equation}\label{eq:BVP-RIBFS-estimate}
	\|(\mathcal{N}_{\kappa}u)\,w\|_{\X(v)} \leq C\, \|f\,w\|_{\X(v)}.
\end{equation}
The case $w\equiv 1$ was obtained in \cite[Theorem~1.5]{K-MMMM}.

This covers the cases $\X=L^p$, $L^{p,q}$, or $L^{p}(\log L)^{\alpha}$ with $1<p<\infty$, $1\le q\le \infty$, and $\alpha\in\re$, in which case we have $p_\X=q_\X=p$, and the weights  satisfy $u^{p}\, v \in A_{p}$, $v \in A_{\infty}$. Analogously, we can consider the spaces $\X=(L^p+L^q)$ or $\X=(L^p\cap L^q)$, with $1<p,q<\infty$, in which case $p_\X=\min\{p,q\}$ and $q_\X=\max\{p,q\}$ and the weights satisfy $u^{p}\, v \in A_{p}$, $u^{q}\, v \in A_{q}$, and $v \in A_{\infty}$.
\end{example}

\begin{theorem}\label{thm:BVP-Phi}
Let $L$ be an elliptic constant complex coefficient second-order $M\times M$ system as in \eqref{eq:Lu}-\eqref{eq:elliptic}. Let $v$ and $w$ be weights on $(\R^{d-1},\L^{d-1})$ such that $v \in L^1_{\loc}(\R^{d-1},\L^{d-1})$. Fix an aperture parameter $\kappa>0$ and a Young function $\Phi\in\Delta_2$ (equivalently, $I_\Phi<\infty$). Let $M$ denote the Hardy-Littlewood maximal function on $(\R^{d-1},\L^{d-1})$ and let $M'_{v}h:=M(h\,v)/v$ for each $h \in\M$. Assume that there is $C_0<\infty$ so that 
\begin{align}
\label{eq:PB1} \int_{\R^{d-1}} \Phi((M_{\B}h)\, w)\, v\, d\L^{d-1} 
&\le C_0 \int_{\R^{d-1}} \Phi(|h|\, w)\,v\, d\L^{d-1},\quad\forall\,h \in \M, 
\\
\label{eq:PB2} \int_{\R^{d-1}} \overline{\Phi}((M'_{\B, v}h)\, w^{-1})\,v \, d\L^{d-1} 
&\le C_0 \int_{\R^{d-1}} \overline{\Phi}(|h|\, w^{-1})\,v\, d\L^{d-1},\quad\forall\,h \in \M. 
\end{align}
Then, the Dirichlet Problem 
\begin{equation}\label{eq::BVP-Phi}
\left\lbrace
\begin{array}{l}
u\in\mathscr{C}^{\infty}(\R^d_{+},\mathbb{C}^M),
\\[4pt]
Lu=0\,\,\text{ in }\,\,\R^d_{+},
\\[4pt] 
\Phi((\mathcal{N}_{\kappa}u)\, w)\in L^1(\R^{d-1}, v),
\\[4pt]
\restr{u}{\partial\R^d_{+}}^{{}^{\kappa-{\rm n.t.}}}=f
\,\,\text{ on }\,\,\R^{d-1},\quad \Phi(|f|\,w)\in L^1(\R^{d-1}, v).
\end{array}
\right.
\end{equation}
is well-posed. More specifically, there is a unique solution and it is given by 
\begin{equation}\label{eq:BVP-Phi-sol}
u(x',t)=(P_t^L*f)(x'),\qquad(x',t)\in{\R}^d_{+},
\end{equation}
where $P^L$ denotes the Poisson kernel for $L$ in $\R^d_+$ from Theorem~\ref{thm:Poisson}. Furthermore, there is a constant $C\in[1,\infty)$ such that
\begin{equation}\label{eq:BVP-Phi-estimate}
\int_{\R^{d-1}} \Phi(|f|\,w)\,v\, d\L^{d-1}\le \int_{\R^{d-1}}\Phi((\mathcal{N}_{\kappa}u)\, w)\,v \, d\L^{d-1} \leq C \int_{\R^{d-1}} \Phi(|f|\,w)\,v\, d\L^{d-1}.
\end{equation}
\end{theorem}

\begin{proof}
Proceeding as in the proof of Theorem~\ref{thm:BVP-BFS}, and with the notation introduced in \eqref{eq:FG-1} we see that \eqref{eq:FG-2} and Theorem~\ref{thm:PhiAp} imply
for every $h\in \mathbb{M}$
\begin{multline}\label{eq:ABwv}
\int_{B(0', 1)}  \Phi \Big(\Big(\int_{\R^{d-1}} \frac{|h(x')|}{1+|x'|^{d-1}}\,dx'\Big)\,w\Big )\,v\, d\L^{d-1}
=
\int_{\R^{d-1}} \Phi(F\, w)\,v \, d\L^{d-1} 
\\
\le 
C\, 
\int_{\R^{d-1}} \Phi(G\, w)\,v \, d\L^{d-1}
=
 C \int_{\R^{d-1}} \Phi(|h|\, w)\,v \, d\L^{d-1}. 
\end{multline}
Since $0<w, v<\infty$ $\L^{d-1}$-a.e., it then follows that 
\begin{align}\label{5q2fg5gg:mod}
\int_{\R^{d-1}} \frac{|h(x')|}{1+|x'|^{d-1}}\,dx' < \infty, \quad\text{for every $h \in \M$ with $\Phi(|h|\,w)\in L^1(\R^{d-1}, v)$}.
\end{align}

Given $f \in \M$ with $\Phi(|f|\,w)\in L^1(\R^{d-1}, v)$,  \eqref{5q2fg5gg:mod} applied to $f$ immediately implies \eqref{zdfgwseg}. Hence, \eqref{eq:Poisson-u-prop} in Theorem~\ref{thm:Poisson} says that $u$ defined in \eqref{eq:BVP-Phi-sol} satisfies 
the first, second, and fourth condition in \eqref{eq::BVP-Phi}. Also, \eqref{eq:Poisson-u-prop2},  the fact that $\Phi\in\Delta_2$, and \eqref{eq:PB1} yield
\begin{multline*}
\int_{\R^{d-1}}\Phi(|f|\,w)\,v\,d\L^{d-1}
\le 	\int_{\R^{d-1}} \Phi((\mathcal{N}_{\kappa}u)\,w)\,v\,d\L^{d-1} 
	\leq \int_{\R^{d-1}} \Phi(C\,(Mf)\, w)\,d\L^{d-1} 
	\\
	\leq C\, \int_{\R^{d-1}} \Phi((Mf)\, w)\,d\L^{d-1} 
	\leq C\, \int_{\R^{d-1}}\Phi(|f|\,w)\,v\,d\L^{d-1}.
\end{multline*}
Thus, the third condition in \eqref{eq::BVP-Phi}, and \eqref{eq:BVP-Phi-estimate}, hold. Finally, \eqref{5q2fg5gg:mod} applied to $\mathcal{N}_{\kappa}u$ and Theorem~\ref{thm:Poisson} item (d) readily imply that $u$, the solution of \eqref{eq::BVP-Phi}, is unique, hence it must be the one given in \eqref{eq:BVP-Phi-sol}.
\end{proof}

\begin{example}\label{ex:Phi:mod:Dir}
As discussed in Example~\ref{ex:Phi:mod}, if $\Phi$ is a Young function such that $1<i_\Phi\le I_\Phi<\infty$, and we assume that $u^{i_\Phi}\, v \in A_{i_\Phi}$, $u^{I_\Phi}\, v \in A_{I_\Phi}$, and $v \in A_{\infty}$, then 
\eqref{eq:intro-CMP-ri:mod} allows us to invoke Theorem~\ref{thm:BVP-Phi}. This can be applied to the cases $\Phi(t)=t^p\,(\log(e+t))^\alpha$, or $\Phi(t)=t^p\,(\log(e+t))^\alpha\,(\log\log(e^e+t))^\beta$, with $1<p<\infty$, $\alpha, \beta\in\re$, in which case $i_{\Phi}=I_{\Phi}=p$, and the weights satisfy $u^{p}\, v \in A_{p}$  and $v \in A_{\infty}$. Also, if $\Phi(t)\approx \min\{t^p,t^q\}$ or $\Phi(t)\approx\max\{t^p,t^q\}$ with $1<p, q<\infty$ we have  
$i_{\Phi}=\min\{p,q\}$ and $I_{\Phi}=\max\{p,q\}$ and our weights satisfy  $u^{p}\, v \in A_{p}$, $u^{q}\, v \in A_{q}$, and $v \in A_{\infty}$.
\end{example}

\subsection{Layer potential operators on uniformly rectifiable domains} \label{subsec:LP-UR}
We say that a Lebes\-gue measurable set $\Omega \subset \R^d$ has locally finite perimeter if its measure theoretic boundary 
\begin{align*}
\partial_{*}\Omega := \bigg\{x \in \partial \Omega: 
\limsup_{r \to 0^{+}} \frac{\L^d(B(x,r) \cap \Omega)}{r^d} >0, 
\limsup_{r \to 0^{+}} \frac{\L^d(B(x,r) \backslash \Omega)}{r^d} >0 \bigg\} 
\end{align*} 
satisfies 
\begin{align*}
\mathcal{H}^{d-1}(\partial_{*}\Omega \cap K)<\infty,\quad \text{for each compact set } K \subset \R^d, 
\end{align*}
where $\mathcal{H}^{d-1}$ denotes the $(d-1)$-dimensional Hausdorff measure. 

Alternatively, a Lebesgue measurable set $\Omega \subset \R^d$ has locally finite perimeter if $\mu_{\Omega}:=\nabla \mathbf{1}_{\Omega}$ (in the sense of distributions) is an $\R^d$-valued Borel measure in $\R^d$ of locally finite total variation. Then the work of De Giorgi-Federer (cf., e.g., \cite{EG}) gives that 
\begin{align*}
\mu_{\Omega} = \nabla \mathbf{1}_{\Omega} 
= -\nu\,  \mathcal{H}^{d-1}|_{\partial_{*}\Omega}, 
\end{align*}
where $\nu \in [L^{\infty}(\partial_{*}\Omega, \H^{d-1})]^d$ is an $\R^d$-valued function satisfying $|\nu(x)|=1$ at $\H^{d-1}$-a.e.~$x \in \partial_{*}\Omega$. We shall refer to $\nu$ above as the {\tt geometric measure theoretic outward unit normal} to $\Omega$. 

\begin{definition}
A closed set $E \subset \R^d$ is called an {\tt Ahlfors regular set} if there exists a constant $C \geq 1$ such that
\begin{align*}
C^{-1} r^{d-1} \leq \mathcal{H}^{d-1}(B(x, r) \cap E) \leq C r^{d-1},\quad \forall x \in E \text{ and } r \in (0, 2\diam(E)). 
\end{align*}
An open, nonempty, proper subset $\Omega$ of $\R^d$ is called an {\tt Ahlfors regular domain} provided $\partial \Omega$ is an Ahlfors regular set and $\mathcal{H}^{d-1}(\partial \Omega \backslash \partial_{*} \Omega)=0$. 
\end{definition}

\begin{definition}
A closed set $E \subset \R^d$ is said to be a {\tt uniformly rectifiable set} (or simply a UR set) if $E$ is an Ahlfors regular set and there exist $\theta, M \in (0, \infty)$ such that for every $x \in E$ and $r \in (0, 2 \diam(E))$ it is possible to find a Lipschitz
map $\phi: B_{d-1}(0, r) \to \R^d$ with Lipschitz constant at most $M$ and such that 
\begin{align*}
\mathcal{H}^{d-1}(E \cap B(x,r) \cap \phi(B_{d-1}(0,r))) \geq \theta r^{d-1}. 
\end{align*}
An open, nonempty, proper subset $\Omega$ of $\R^d$ is called a {\tt uniformly rectifiable domain} (or simply UR domain) provided $\partial \Omega$ is a UR set and $\mathcal{H}^{d-1}(\partial \Omega \backslash \partial_{*} \Omega)=0$. 
\end{definition}

\begin{definition}
An open, nonempty, proper subset $\Omega$ of $\R^d$ is said to satisfy a {\tt local John condition} if there exist $\theta \in (0, 1)$  and $r_0>0$ (with the requirement that $r_0=\infty$ if $\partial \Omega$ is unbounded) such that for every $x \in \partial \Omega$  and $r \in (0, r_0)$ one may find $x_r \in B(x, r) \cap \Omega$ such that $B(x_r, \theta r) \subset \Omega$ and with the property that for each $y \in B(x, r) \cap \partial \Omega$ there exists a rectifiable path $\gamma_y:[0, 1] \to \overline{\Omega}$ whose length is no more than $\theta^{-1} r$ and such that
\begin{align*}
\gamma_y(0) = y,\quad  \gamma_y(1) = x_r, \quad 
\dist(\gamma_y(t), \partial \Omega) >\theta |\gamma_y(t)-y|,~ \forall t \in (0, 1]. 
\end{align*}
Finally, a nonempty open set $\Omega \subset \R^d$ which is not dense in $\R^d$ is said to satisfy a {\tt two-sided local John condition} if both $\Omega$ and $\R^d \backslash \overline{\Omega}$ satisfy a local John condition.
\end{definition}

\begin{lemma}[{\cite[Corollary~3.14]{HMT}}]\label{lem:John}  
Let $\Omega \subset \R^d$ be a domain satisfying a two-sided local John condition and whose boundary is Ahlfors regular. 
Then $\Omega$ is a UR domain of locally finite perimeter. 
\end{lemma}

Throughout this subsection, abbreviate $\sigma=\mathcal{H}^{d-1}|_{\partial \Omega}$ and denote by $\nu$ the geometric measure theoretic outward unit normal to $\Omega$. Define the maximal operators $T_{*}$ and $T^{\#}_{*}$ by 
\begin{equation}\label{xfgbgsb}
T_{*}f(x) := \sup_{\varepsilon>0} |T_{\varepsilon}f(x)| \quad\text{and}\quad 
T_{*}^{\#}f(x) := \sup_{\varepsilon>0} |T^{\#}_{\varepsilon}f(x)|, 
\end{equation}
where 
\begin{align*}
T_{\varepsilon}f(x) = \int_{\substack{y \in \partial \Omega \\ |x-y|>\varepsilon}} 
\langle x-y, \nu(y) \rangle K(x-y) f(y) d\sigma(y),  
\\
T^{\#}_{\varepsilon}f(x) = \int_{\substack{y \in \partial \Omega \\ |x-y|>\varepsilon}} 
\langle y-x, \nu(x) \rangle K(x-y) f(y) d\sigma(y),
\end{align*}
where  $K \in \mathscr{C}^N(\R^d \backslash \{0\})$ is a complex-valued function which is even and positive homogeneous of degree $-d$, and $N=N(d) \in \N$ is large enough. We also consider the principal-value singular integral operators $T$ and $T^{\#}$ (whenever they exist): 
\begin{equation}\label{eq:layer-def}
Tf(x) := \lim_{\varepsilon \to 0^{+}} T_{\varepsilon}f(x) \quad\text{and}\quad 
T^{\#}f(x) := \lim_{\varepsilon \to 0^{+}} T^{\#}_{\varepsilon}f(x).  
\end{equation}

Let $\w_{d-1}$ denote the area of the unit sphere in $\R^d$. If we take $K(x)=\w_{d-1}^{-1}\,|x|^{-d}$ above, then 
$K \in \mathscr{C}^N(\R^d \backslash \{0\})$ is even and homogeneous of degree $-d$ for any $N \in \N$. 
It is easy to see that the first operator in \eqref{eq:layer-def} coincides in this case with the harmonic double layer potential: 
\begin{align}\label{eq:layer-double} 
\mathcal{K}_{\Delta}f(x) := \lim_{\varepsilon \to 0^{+}} \frac{1}{\w_{d-1}} 
\int_{\substack{y \in \partial \Omega \\ |x-y|>\varepsilon}} 
\frac{\langle x-y, \nu(y) \rangle}{|x-y|^d} f(y) d\sigma(y). 
\end{align}
One can also introduce the Riesz transform: 
\begin{align}\label{eq:Rj-f}
R_j f(x) := \lim_{\varepsilon \to 0^{+}} \frac{1}{\w_{d-1}} 
\int_{\substack{y \in \partial \Omega \\ |x-y|>\varepsilon}} 
\frac{x_j-y_j}{|x-y|^d} f(y) d\sigma(y), \quad j=1,\ldots,d.  
\end{align}
The layer potential operators have some applications in geometric measure theory and PDE. Indeed, Hofmann, Mitrea and Taylor \cite[Section~4]{HMT} characterized bounded regular SKT domains by means of the compactness or close to compactness of the harmonic double layer $\mathcal{K}_{\Delta}$ along with the boundedness of the commutators of the Riesz transform and the components of the outer unit normal. The case of\, unbounded SKT domains is considered in the recent work \cite{MMMMM20} (see also \cite{Mar}) and the compactness or close to compactness is replaced by the smallness of the boundedness constants of such operators.  Also, the Riesz transform was used by Mitrea et al.~to investigate the regularity of various domains including Lyapunov domains of order $\alpha$, UR domains, regular SKT domains, and Reifenberg flat domains, see \cite[Theorems~1.1-1.4, 7.7]{MiMiV}. On the other hand, the layer potential operators can be also used in the study of boundary value problems.  For any bounded Ahlfors regular domain satisfying a two-sided John condition, using the method of layer potentials, Hofmann, Mitrea and Taylor \cite[Section~7]{HMT} established the well-posedness of elliptic boundary value problems such as the Dirichlet, Neumann and transmission problems for the Laplace operator, the Stokes system, the Lamé system, and Maxwell's equations. For bounded $\Lip \cap {\rm vmo}_1$ domains, the same method was also utilized in \cite[Section~4]{HMT2} to study the Dirichlet, regularity and oblique derivative problems, as well as the Poisson problem with a Dirichlet boundary condition.  In the unbounded case the reader is referred to \cite{MMMMM20, Mar}. More general results about layer potentials and applications can be found in \cite{HMaM} and \cite{HMM}. 

To state our results, we need to introduce the John-Nirenberg space of functions of bounded mean oscillations on Ahlfors regular sets. Given a domain $\Omega \subset \R^d$, for each $x \in \partial \Omega$ and $r>0$ define the surface ball 
$\Delta:=\Delta(x, r) :=B(x, r) \cap \partial \Omega$ and denote $r_{\Delta}:=r$. For any constant $\lambda>0$, we also define 
$\lambda \Delta:=\Delta(x, \lambda r)$. We shall then denote by ${\rm BMO}(\partial \Omega, \sigma)$ the space of all functions $f \in L^1_{\loc}(\partial \Omega, \sigma)$  with the property 
\begin{align*}
\|f\|_{{\rm BMO}(\partial \Omega, \sigma)} 
:= \sup_{\Delta \subset \partial \Omega} \fint_{\Delta} |f-f_{\Delta}| d\sigma, 
\end{align*}
where the supremum is taken over all surface balls $\Delta \subset \partial \Omega$ and 
$h_{\Delta}= \fint_{\Delta} h d\sigma$ for any locally integrable function.

Given a linear operator $\mathbb{T}$, we define (whenever it makes sense) the first order commutator of $\mathbb{T}$ and the operator $\mathbf{M}_b$ of pointwise multiplication by a measurable function $b$ by 
\[
\mathbf{C}_{b}^{1}(\mathbb{T}) f(x) 
=
[\mathbf{M}_b, \mathbb{T}]f(x)
=
b(x)\,\mathbb{T}f(x)- \mathbb{T}(b\,f)(x)
=
\mathbb{T}\big( (b(x)-b(\cdot)) f(\cdot) \big)(x).
\]
One can also define the higher order commutators of $\mathbb{T}$ with a measurable function $b$ by
the recursive formula $\mathbf{C}_{b}^{k}(\mathbb{T}) f=\mathbf{C}_{b}^{1}(\mathbb{T})\circ \mathbf{C}_{b}^{k-1}(\mathbb{T}) f$ for every $k\ge 2$. One can then see that
\[
\mathbf{C}_{b}^{k}(\mathbb{T}) f(x) 
=
\mathbb{T}\big( (b(x)-b(\cdot))^k f(\cdot) \big)(x), \qquad k\ge 0,
\]
where it is understood that $\mathbf{C}_{b}^{0}(\mathbb{T}) f=\mathbb{T} f$. The previous definition can be extended to linearizable operators, that is, to operators $\mathbb{T}$ so that $\mathbb{T}f=\|\vec{\mathbb{T}}f\|_{\mathbb{B}}$ for some $\mathbb{B}$-valued linear operator $\vec{\mathbb{T}}$ and some Banach space $\mathbb{B}$. In this way we set, for all $k\ge 0$, 
\[
\mathbf{C}_{b}^{k}(\mathbb{T}) f(x) 
:=
\|\mathbf{C}_{b}^{k}(\vec{\mathbb{T}}) f(x) \|_{\mathbb{B}}
=
\big\|\vec{\mathbb{T}}\big( (b(x)-b(\cdot))^k f(\cdot) \big)\big\|_{\mathbb{B}}
=
\mathbb{T}\big( (b(x)-b(\cdot))^k f(\cdot) \big)(x).
\]
All these motivate the following general definition. Given some operator $\mathbb{T}$, $m \in \N_{+}$, $\alpha=(\alpha_1, \ldots, \alpha_m) \in \N^m$, and ${\bf b}=(b_1,\ldots,b_m)$ a vector of measurable functions, we define the $\alpha$-th order commutator of $\mathbb{T}$ with $\mathbf{M}_{\bf b}$ as 
\begin{align*}
\mathbf{C}_{\bf b}^{\alpha}(\mathbb{T}) f(x) := \mathbb{T}\bigg( \prod_{i=1}^m (b_i(x)-b_i(\cdot))^{\alpha_i} f(\cdot) \bigg)(x).  
\end{align*}
In the case where $\mathbb{T}$ is linear $\mathbf{C}_{\bf b}^{\alpha}(\mathbb{T})$ is then the composition of the commutators 
$\mathbf{C}_{b_1}^{\alpha_1}(\mathbb{T}) f$, \dots,  $\mathbf{C}_{b_m}^{\alpha_m}(\mathbb{T}) f$. 

The following theorem extends the results in \cite[Chapters~4 \& 8]{MMMMM20} and \cite[Chapter~5]{Mar}.

\begin{theorem}\label{thm:doamin-extra}
Let $\Omega \subset \R^d$ be an Ahlfors regular domain satisfying a two-sided local John condition. Let $\B$ the collection of all surface balls on $\partial \Omega$ and write $M_{\B}h$ for the associated Hardy-Littlewood maximal function. Let $u$ and $v$ be two measurable functions on $(\partial \Omega, \sigma)$ such that $0<u,v<\infty$ $\sigma$-a.e.~and $v \in L^1_{\loc}(\partial \Omega, \sigma)$. Let $N=N(d) \in \N$ be a sufficiently large integer and $K \in \mathscr{C}^N(\R^d \backslash \{0\})$ be a complex-valued function which is even and positive homogeneous of degree $-d$. Let $m \in \N_+$, ${\bf b} \in [\BMO(\partial \Omega, \sigma)]^m$ and $\nu \in [\BMO(\partial \Omega, \sigma)]^d$. Then for each operator $\mathbb{T}\in \{T,T^{\#},T_{*},T_{*}^{\#}\}$ as in \eqref{xfgbgsb}--\eqref{eq:layer-def}, the following statements hold:
\begin{list}{\textup{(\theenumi)}}{\usecounter{enumi}\leftmargin=1cm \labelwidth=1cm \itemsep=0.2cm 
		\topsep=.2cm \renewcommand{\theenumi}{\alph{enumi}}}
	
\item If  $\X_v$ is a Banach function space over $(\partial \Omega, vd\sigma)$ such that 
\begin{align*}
\norm{(M_{\B}h)\,u}_{\X_v} &\leq C_0\, \|h\,u\|_{\X_v}, \quad \forall\,h \in\M_{\sigma}, 
\\
\|(M'_{\B, v} h)\,u^{-1} \|_{\X'_v} &\leq C_0\, \|h\,u^{-1}\|_{\X'_v}, \quad \forall\,h \in\M_{\sigma}, 
\end{align*}
for some $C_0<\infty$, then
\begin{align*}
	\big\| |\mathbb{T}f|\, u \big\|_{\X_v} 
	 \leq C\,  	\| f\, u \big\|_{\X_v}\quad\text{ and }\quad 
	\big\| |\mathbf{C}_{\bf b}^{\alpha}(\mathbb{T})f|\, u \big\|_{\X_v} 
	\leq C\,  \Big(\prod_{i=1}^m \|b_i\|_{\BMO}^{\alpha_i}\Big)	\| f\, u \big\|_{\X_v}. 
\end{align*}
and, for every $q \in (1, \infty)$,  
\begin{align*}
\bigg\|\Big(\sum_j |\mathbb{T}f_j|^q \Big)^{\frac1q}\,  u \bigg\|_{\X_v} 
& \leq C  \bigg\|\Big(\sum_j |f_j|^q \Big)^{\frac1q}\, u\bigg\|_{\X_v},
\\
\bigg\|\Big(\sum_j |\mathbf{C}_{\bf b}^{\alpha}(\mathbb{T})f_j|^q \Big)^{\frac1q}\,  u \bigg\|_{\X_v} 
& \leq C\,\Big(\prod_{i=1}^m \|b_i\|_{\BMO}^{\alpha_i}\Big)\,  \bigg\|\Big(\sum_j |f_j|^q \Big)^{\frac1q}\, u\bigg\|_{\X_v}. 
\end{align*}

\item  If $\Phi$ is a Young function such that $\Phi \in \Delta_2$ (equivalently, $I_{\Phi}<\infty$) and 
\begin{align*}
\int_{\partial \Omega} \Phi((M_{\B}h)\, u)\, v\, d\sigma 
&\le C_0 \int_{\partial \Omega} \Phi(|h|\, u)\, v\, d\sigma,\quad\forall\,h \in \M_{\sigma}, 
\\
\int_{\partial \Omega} \overline{\Phi}((M'_{\B, v}h)\,u^{-1})\, v\, d\sigma 
&\le C_0\, \int_{\partial \Omega} \overline{\Phi}(|h|\, u^{-1})\,v\, d\sigma,\quad\forall\,h \in \M_{\sigma}, 
\end{align*} 
for some $C_0<\infty$, then
\begin{align*}
	\int_{\partial \Omega} \Phi( |\mathbb{T}f|\, u )\,v\, d\sigma
	& \leq C  \int_{\partial \Omega} \Phi(|f|\,u) \,v\, d\sigma, 
	\\
	\int_{\partial \Omega} \Phi( |\mathbf{C}_{\bf b}^{\alpha}(\mathbb{T})f|\, u )\,v\, d\sigma
	& \leq C\,\Big(\prod_{i=1}^m \|b_i\|_{\BMO}^{\alpha_i}\Big)\,  \int_{\partial \Omega} \Phi(|f|\,u) \,v\, d\sigma, 
\end{align*}
and, for every $q \in (1, \infty)$,    
\begin{align*}
\int_{\partial \Omega} \Phi \bigg(\Big(\sum_j |\mathbb{T}f_j|^q \Big)^{\frac1q} u \bigg)\,v\, d\sigma
& \leq C  \int_{\partial \Omega} \Phi \bigg(\Big(\sum_j |f_j|^q \Big)^{\frac1q} u \bigg)\,v\, d\sigma, 
\\
\int_{\partial \Omega} \Phi \bigg(\Big(\sum_j |\mathbf{C}_{\bf b}^{\alpha}(\mathbb{T})f_j|^q \Big)^{\frac1q} u \bigg)\,v\, d\sigma
& \leq C\,  
\Big(\prod_{i=1}^m \|b_i\|_{\BMO}^{\alpha_i}\Big)\,
\int_{\partial \Omega} \Phi \bigg(\Big(\sum_j |f_j|^q \Big)^{\frac1q} u \bigg)\,v\, d\sigma. 
\end{align*}
\end{list}
\end{theorem}

\begin{proof}
In view of Theorems \ref{thm:BFSAp} and \ref{thm:PhiAp}, we are reduced to proving that for every (or some) $p \in (1, \infty)$ and every $w \in A_{p, \B}$, 
\begin{align}
\label{eq:layer-1} \|\mathbb{T}f\|_{L^p(\partial \Omega, w)} &\leq C \|f\|_{L^p(\partial \Omega, w)}, 
\\
\label{eq:layer-2} \|C_{\bf b}^{\alpha}(\mathbb{T})f\|_{L^p(\partial \Omega, w)} &\leq C 
\prod_{i=1}^m \|b_i\|_{\BMO(\partial \Omega, \sigma)}^{\alpha_i} \|f\|_{L^p(\partial \Omega, w)},   
\end{align} 
where the constant $C$ depends only on $d$, $p$, $\alpha$, $[w]_{A_{p, \B}}$, $\|\nu\|_{[{\rm BMO}(\partial \Omega, \sigma)]^d}$, the local John constants of $\Omega$ and the Ahlfors-David regular constant of $\partial \Omega$.  The inequality \eqref{eq:layer-1} can be found in \cite[Chapter~4]{MMMMM20} or \cite[Chapter~5]{Mar}. Note that $T$ and $T^{\#}$ are linear, and $T_{*}$ and $T_{*}^{\#}$ are linearizable. Thus, the estimate \eqref{eq:layer-2} follows at once from \eqref{eq:layer-1} and  \cite[Theorem~3.22]{BMMST}. 
\end{proof}

In the previous result we can take $v\equiv 1$ and $\X_v=L^{p(\cdot)}(\R^{d-1},\L^{d-1})$ as in Example~\ref{ex:Lpvar}, and assume that $1<p_-\le p_+<\infty$, $p(\cdot)\in{\rm LH}$, $u\in A_{p(\cdot)}$  (cf.~\eqref{eq:Lpvar-pp}--\eqref{eq:Lpvar-Apvar}). In the case of rearrangement Banach function spaces we can consider $\X=L^p$, $L^{p,q}$, or $L^{p}(\log L)^{\alpha}$ with $1<p<\infty$, $1\le q\le \infty$, and $\alpha\in\re$, in which case we have $p_\X=q_\X=p$ and weights $u,v$ satisfying $u^{p}\, v \in A_{p}$, $v \in A_{\infty}$. Analogously, we can consider the spaces $\X=(L^p+L^q)$ or $\X=(L^p\cap L^q)$, with $1<p,q<\infty$, in which case $p_\X=\min\{p,q\}$ and $q_\X=\max\{p,q\}$ if $\partial\Omega$ is unbounded and the weights satisfy $u^{p}\, v \in A_{p}$, $u^{q}\, v \in A_{q}$, and $v \in A_{\infty}$. 

Regarding the modular inequalities, we can consider a Young function $\Phi$ such that $1<i_\Phi\le I_\Phi<\infty$, and assume that $u^{i_\Phi}\, v \in A_{i_\Phi}$, $u^{I_\Phi}\, v \in A_{I_\Phi}$, and $v \in A_{\infty}$.  This can be applied to the cases $\Phi(t)=t^p\,(\log(e+t))^\alpha$, or $\Phi(t)=t^p\,(\log(e+t))^\alpha\,(\log\log(e^e+t))^\beta$, with $1<p<\infty$, $\alpha, \beta\in\re$, in which case $i_{\Phi}=I_{\Phi}=p$, and the weights satisfy $u^{p}\, v \in A_{p}$  and $v \in A_{\infty}$. Also, if $\Phi(t)\approx \min\{t^p,t^q\}$ or $\Phi(t)\approx\max\{t^p,t^q\}$ with $1<p, q<\infty$ we have $i_{\Phi}=\min\{p,q\}$ and $I_{\Phi}=\max\{p,q\}$ and our weights satisfy  $u^{p}\, v \in A_{p}$, $u^{q}\, v \in A_{q}$, and $v \in A_{\infty}$.

\subsection{Non-homogeneous square functions} 
In this section we work with $\B$ being the collection of all cubes in $\Rn$ with sides parallel to the coordinate axes and with a
Borel measure $\mu$ on $\Rn$ of order $m \in \R_{+}$, that is, 
\begin{align*}
\mu(B(x,r)) \leq C_{\mu} r^m,\ \ x \in \Rn, \ r > 0.
\end{align*}
One can equivalently write this condition (with a different constant) using cubes in place of balls. As in \cite{OP} after a possible rotation we may assume that $\mu(\partial Q)=0$ for every cube $Q\subset\re^n$ with sides parallel to the coordinate axes. Let $M_{\mu}$ be the centered maximal function 
\begin{align}\label{def:HL-cent}
	M_{\mu}f(x) := \sup_{r>0} \fint_{Q(x, r)} |f|\, d\mu, 
\end{align}
with $Q(x, r)$ denoting the cube centered at $x$ with sidelength $r$.  This should be compared with $M_{\B}$ which is the Hardy-Littlewood maximal function associated with $\B$, and therefore the sup is taken over all cubes containing the point in question. Thus, $M_{\mu}f(x)\le M_{\B}f(x)$ for every $x\in\re^n$ and every $\mu$-measurable function, but since $\mu$ may not be doubling these two maximal functions are not comparable. It was obtained in \cite[Theorem 3.1]{OP} that $M_{\mu}$ is bounded on $L^p(w)$ for any $1<p<\infty$ and $w \in A_{p, \B}$. However, in general $M_\B$ is not expected to satisfy weighted norm inequalities. In the language introduced earlier, $\B$ is a basis, but may not be a Muckenhoupt basis.  Nonetheless, we are going to be able to extrapolate as Theorems~\ref{thm:BFSAp} and \ref{thm:PhiAp} do not require such a condition.

We introduce the vertical square function 
\begin{align*}
g_{\mu}(f)(x) = \bigg(\int_{0}^{\infty} |\theta_t^{\mu} f(x)|^2 \frac{dt}{t}\bigg)^{\frac12}, \qquad\text{where}\quad  
\theta_t^{\mu} f(x) =\int_{\Rn} s_t(x,y) f(y) d\mu(y).
\end{align*}
The kernel $s_t : \R^{n} \times \R^{n} \rightarrow \C$ satisfies, for some $\alpha>0$,  the size condition
\begin{align}\label{eq:size} 
|s_t(x, y)| \lesssim \frac{t^{\alpha}}{(t+|x-y|)^{m+\alpha}}
\end{align}
and the smoothness condition
\begin{align}
\label{eq:smooth} |s_t(x, y)-s_t(x', y)|  + |s_t(y, x)-s_t(y,x')| 
\lesssim \frac{|x-x'|^{\alpha}}{(t+|x-y|)^{m+\alpha}}, 
\end{align}
whenever $|x-x'| < t/2$.

\begin{theorem}\label{thm:square}
Let $\mu$ be a Borel measure of order $m$ on $\Rn$ and $\B$ be the collection of all cubes in $\Rn$ with sides parallel to the coordinate axes. Let $u$ and $v$ be weights on $(\Rn, d\mu)$ such that $v \in L^1_{\loc}(\Rn, d\mu)$. Assume that for every cube $Q \subset \Rn$  there is a measurable function $b_Q$ such that  $\supp b_Q \subset Q$, $\|b_Q\|_{L^{\infty}(\mu)} \lesssim 1$, 
\begin{align}\label{eq:hypo-square}
\left|\fint_{Q} b_Q \ d\mu \right| \gtrsim 1, 
\quad\text{ and }\quad \sup_{Q} \frac{1}{\mu(3Q)} \int_{0}^{\ell(Q)} \int_{Q} 
|\theta_t b_Q(x)|^2 d\mu(x) \frac{dt}{t} \lesssim 1.  
\end{align} 
\begin{list}{\textup{(\theenumi)}}{\usecounter{enumi}\leftmargin=1cm \labelwidth=1cm \itemsep=0.2cm 
		\topsep=.2cm \renewcommand{\theenumi}{\alph{enumi}}}
	
\item If $\X_v$ is a Banach function space on $(\Rn,v\,d\mu)$ such that 
\begin{align*}
\norm{(M_{\B}h)\,u}_{\X_v} &\leq C_0 \|h\,u\|_{\X_v}, \quad \forall\,h \in\M_{\mu}, 
\\
\|(M'_{\B, v} h)\,u^{-1} \|_{\X'_v} &\leq C_0 \|h\,u^{-1}\|_{\X'_v}, \quad \forall\,h \in\M_{\mu}, 
\end{align*}
for some $C_0<\infty$, then 
\begin{align}\label{eq:gmuf-X}
\|g_{\mu}(f)\,u\|_{\X_v} \leq C  \|f\, u\|_{\X_v}. 
\end{align}
\item  If $\Phi$ is a Young function such that $\Phi \in \Delta_2$ (equivalently, $I_{\Phi}<\infty$) and 
\begin{align*}
\int_{\Rn} \Phi((M_{\B}h)\, u)\, v\, d\mu 
&\le C_0 \int_{\Rn} \Phi(|h|\, u)\, v\, d\mu,\quad\forall\,h \in \M_{\mu}, 
\\
\int_{\Rn} \overline{\Phi}((M'_{\B, v}h)\,u^{-1}) v\, d\mu 
&\le C_0 \int_{\Rn} \overline{\Phi}(|h|\, u^{-1})\,v\, d\mu,\quad\forall\,h \in \M_{\mu},  
\end{align*} 
for some $C_0<\infty$,
then 
\begin{equation}\label{eq:gmuf-Phi}
\int_{\Rn} \Phi(g_{\mu}(f)\,u)\, v\, d\mu \leq C \int_{\Rn} \Phi(|f|\, u)\, v\, d\mu. 
\end{equation}
\end{list}
\end{theorem}

\begin{proof}
In order to obtain \eqref{eq:gmuf-X} and \eqref{eq:gmuf-Phi},  we will prove the weighted inequality: 
\begin{align}\label{eq:gLp}
\|g_{\mu}(f)\|_{L^p(w)} \leq C \|f\|_{L^p(w)}, \quad\forall\,p \in (1, \infty) \quad\text{and}\quad\forall\,w \in A_{p, \B}.    
\end{align}
Thus, \eqref{eq:gmuf-X} and \eqref{eq:gmuf-Phi} respectively follow from Theorems \ref{thm:BFSAp} and \ref{thm:PhiAp}. 

To show \eqref{eq:gLp}, we borrow some idea from \cite[Theorem~2.22]{T2} and \cite[Theorem~5.3]{MMV}, and establish a good-$\lambda$ inequality. We would like to mention that the estimate \eqref{eq:gLp} in the unweighted case was studied in \cite{MMO}. For any $0<t_0<1$,  we define the $t_0$-truncated version of $g_{\mu}$ by 
\begin{align}
g_{\mu,t_0}(f)(x) = \bigg(\int_{t_0}^{t_0^{-1}} |\theta_t^{\mu} f(x)|^2 \frac{dt}{t}\bigg)^{\frac12}, \quad x \in \Rn.   
\end{align}
Fix $t_0\in(0,1)$ and set
\[
\Omega_{\lambda}:=\{x \in \Rn: g_{\mu,t_0}(f)(x) > \lambda\}, \quad \lambda>0. 
\]
We may assume that $f$ is compactly supported and bounded with $\supp f \subset B_0$ for some ball $B_0 \subset \Rn$.   
We would like to apply Whitney decomposition, and to do that we first show that $\Omega_{\lambda}$ is an open proper subset of $\Rn$. 
Indeed, by the smoothness condition \eqref{eq:smooth}, we have for all $|x-x'|<t_0/2$ and $t \geq t_0$, 
\begin{multline*}
|\theta_t^{\mu}f(x) - \theta_t^{\mu}f(x')| 
\lesssim \int_{\Rn} \frac{|x-x'|^{\alpha}}{(t+|x-y|)^{m+\alpha}} |f(y)| d\mu(y) 
\\
\leq t^{-\alpha_0} |x-x'|^{\alpha} \|f\|_{L^p(\mu)} 
\bigg(\int_{\Rn} \frac{d\mu(y)}{(t_0+|x-y|)^{(m+\alpha-\alpha_0)\,p'}} \bigg)^{\frac1{p'}}  
\leq C_{t_0} t^{-\alpha_0} |x-x'|^{\alpha} \|f\|_{L^p(\mu)},  
\end{multline*}
where $\alpha_0$ is an auxiliary parameter such that $0<\alpha_0<\alpha$. This implies that 
\begin{align*}
|g_{\mu, t_0}f(x) - g_{\mu, t_0}f(x')| 
&\leq C_{t_0} |x-x'|^{\alpha} \|f\|_{L^p(\mu)}, 
\end{align*}
and hence, the mapping $x \mapsto g_{\mu, t_0}f(x)$ is continuous. Thus, $\Omega_{\lambda}$ is an open set.  
On the other hand, by \eqref{eq:size} it is not hard to see that
\begin{align}
|g_{\mu, t_0}f(x)| & \leq \|f\|_{L^{\infty}(\mu)} \bigg(\int_{t_0}^{\infty} \bigg( \int_{B_0} 
\frac{t^{\alpha}}{(t+|x-y|)^{m+\alpha}} d\mu(y) \bigg)^2 \frac{dt}{t} \bigg)^{\frac12}  
\nonumber \\
&\leq \|f\|_{L^{\infty}(\mu)}  \left(\int_{t_0}^{\infty} 
\frac{t^{-2\alpha'_0}\mu(B_0)^2}{(t_0+\dist(x, B_0))^{2(m-\alpha'_0)}} \frac{dt}{t} \right)^{\frac12}  
\nonumber \\
&\leq C(f, t_0) \frac{\mu(B_0)}{(t_0+\dist(x, B_0))^{(m-\alpha'_0)}} \to 0,  \quad\text{as}\quad |x| \to \infty, 
\label{Qvavffa}
\end{align}
where $0<\alpha'_0<m$ is some fixed parameter. This readily implies that $\Omega_{\lambda} \subsetneq \Rn$. Keeping these facts in mind and applying Whitney decomposition 
\cite[Charper~VI]{Stein}, one can find a collection $\{Q_i\}_{i}$ of closed dyadic cubes with disjoint interiors such that 
\begin{equation}\label{eq:W-dec}
\Omega_{\lambda} = \bigcup_{i} Q_i,\quad  
\kappa Q_i \cap \Omega_{\lambda}^c \neq \emptyset,
\quad\text{and}\quad \sum_{i} \mathbf{1}_{10Q_i}(x)\le\kappa_0  \mathbf{1}_{\Omega}(x),
\end{equation} 
for some $\kappa>20$ and where $\kappa_0$ depends just on $n$. 

For later use we observe that from \eqref{eq:size} one can easily see that for any measurable function $h$ we have 
\begin{multline}\label{aFWfveV}
\int_{\Rn} \frac{|h(y)|}{(t+|x-y|)^{m+\alpha}} d\mu(y) 
\\
\lesssim 
t^{-(m+\alpha)}\int_{B(x,t)}|h(y)|d\mu(y) 
+
\sum_{j=0}^\infty \int_{B(x,2^{j+1}t)\setminus B(x,2^{j}t)} \frac{|h(y)|}{|x-y|^{m+\alpha}} d\mu(y) 
\\
\lesssim
M_{\mu}h(x) \sum_{j=-1}^\infty (2^j\,t)^{-(m+\alpha)} \mu(Q(x,2^{j+2}t)) 
\lesssim t^{-\alpha}\, M_{\mu}h(x),  
\end{multline} 
where $M_\mu$ is the centered Hardy-Littlewood maximal function, see \eqref{def:HL-cent}.

To proceed we observe that under the hypotheses of Theorem \ref{thm:square}, using non-homo\-ge\-neous analysis and probabilistic methods, it was obtained in  \cite[Theorem~1.4]{MMO} that $g_{\mu}$ is bounded on $L^2(\mu)$, and that only uses the smoothness condition with respect to the second variable \eqref{eq:smooth}.   
Subsequently, by means of the non-homogeneous Calderón-Zygmund decomposition of Tolsa \cite[Lemma~2.4]{T1} or \cite[Lemma~2.14]{T2}, the same authors proved that the $L^2(\mu)$ boundedness of $g_{\mu}$ implies that
\begin{align}\label{eq:g-weak}
g_{\mu}:L^1(\mu) \to L^{1,\infty}(\mu) \quad \text{boundedly}, 
\end{align}
and, by interpolation, 
\begin{align}\label{eq:g-Lp}
g_{\mu}:L^p(\mu) \to L^p(\mu) \quad \text{boundedly for  all } 1 <p \leq 2.  
\end{align}

We are going to apply \eqref{eq:g-weak} to derive the following localized good-$\lambda$ inequality: for any given $\varepsilon>0$ and $\gamma \in (0,1)$, there exists $\delta=\delta(\varepsilon,\gamma)>0$ such that for any given a Whitney cube $Q_i$,
\begin{equation}\label{eq:QQ}
\mu(E_{Q_i}) := 
\mu \big(\big\{x \in Q_i: g_{\mu, t_0}(f)(x) > (1+\varepsilon)\lambda, M_{\mu}f(x) \leq \delta \lambda  \big\}\big)
\leq \gamma \mu(4Q_i).
\end{equation}

Let us then show \eqref{eq:QQ}. Fix a Whitney cube $Q_i$ and we may assume that $E_{Q_i}\not=\emptyset$, otherwise the estimate is trivial. This and the fact that $\kappa Q_i \cap \Omega_{\lambda}^c \neq \emptyset$ allow us to pick $x_i\in E_{Q_i}$ and $x_i' \in \kappa Q_i \cap \Omega_{\lambda}^c$.  
We claim that 
\begin{align}\label{eq:ggM}
g_{\mu, t_0}(f\, \mathbf{1}_{(2Q_i)^c})(x) \leq g_{\mu, t_0}(f)(x_i') + C_0\, M_{\mu}f(x), \quad \forall\,x \in Q_i.  
\end{align}
To see this, fix $x \in Q_i$ and split 
\begin{align}\label{eq:J1J2J3}
g_{\mu, t_0}(f\, \mathbf{1}_{(2Q_i)^c})(x) 
&\leq \bigg(\int_{0}^{2\kappa \diam(Q_i)} |\theta_t^{\mu}(f\, \mathbf{1}_{(2Q_i)^c})(x)|^2 \frac{dt}{t} \bigg)^{\frac12}
\nonumber \\
&\qquad + \bigg(\int_{\max\{t_0,2\kappa \diam(Q_i)\}}^{t_0^{-1}} |\theta_t^{\mu}(f\, \mathbf{1}_{(2Q_i)^c})(x_i')|^2 \frac{dt}{t} \bigg)^{\frac12}
\nonumber\\
&\qquad + \bigg(\int_{2\kappa \diam(Q_i)}^{\infty} |\theta_t^{\mu}(f\, \mathbf{1}_{(2Q_i)^c})(x) 
-\theta_t^{\mu}(f\, \mathbf{1}_{(2Q_i)^c})(x_i')|^2 \frac{dt}{t} \bigg)^{\frac12}
\nonumber\\
&=:\mathcal{J}_1 + \mathcal{J}_2 + \mathcal{J}_3,
\end{align}
where it is understood that $\mathcal{J}_2=0$ when $2\kappa \diam(Q_i)>t_0^{-1}$. 
To bound $\mathcal{J}_1$, we note that $Q(x, \ell(Q_i)/2) \subset 2Q_i$. Then, using the size condition \eqref{eq:size}, we have for any $t>0$, 
\begin{align*}
|\theta_t^{\mu}(f\, \mathbf{1}_{(2Q_i)^c})(x)| 
&\lesssim \int_{(2Q_i)^c} \frac{t^{\alpha}}{(t+|x-y|)^{m+\alpha}}  |f(y)| d\mu(y)
\\
&\le \sum_{j=0}^{\infty} \int_{Q(x, 2^j\ell(Q_i)) \setminus Q(x, 2^{j-1}\ell(Q_i))}   
\frac{t^{\alpha}}{|x-y|^{m+\alpha}}  |f(y)| d\mu(y)
\\
&\lesssim \sum_{j=0}^{\infty} \frac{t^{\alpha} \mu(Q(x, 2^j \ell(Q_i)))}{(2^j \ell(Q_i))^{m+\alpha}} 
\fint_{Q(x, 2^j \ell(Q_i))}  |f(y)| d\mu(y)
\\
&\lesssim t^{\alpha} \sum_{j=0}^{\infty} (2^j \ell(Q_i))^{-\alpha} M_{\mu} f(x) 
\lesssim t^{\alpha} \ell(Q_i)^{-\alpha} M_{\mu} f(x).   
\end{align*}
As a result,
\begin{align}\label{eq:J1}
\mathcal{J}_1 
\lesssim 
\bigg(\int_{0}^{2\kappa \diam(Q_i)} t^{2\alpha} \frac{dt}{t} \bigg)^{\frac12} 
\ell(Q_i)^{-\alpha} M_{\mu} f(x) 
\lesssim 
M_{\mu} f(x). 
\end{align}
Let us estimate $\mathcal{J}_2$. Using that $2Q_i \subset Q(x,  3\ell(Q_i))$ and the size condition \eqref{eq:size} again, we arrive at
\begin{multline*}
|\theta_t^{\mu}(f\, \mathbf{1}_{2Q_i})(x_i')| 
\lesssim \int_{2Q_i} \frac{t^{\alpha}}{(t+|x_i'-y|)^{m+\alpha}} |f(y)| d\mu(y) 
\\
\le \frac{1}{t^m} \int_{Q(x, 3\ell(Q_i))} |f(y)| d\mu(y)
\lesssim t^{-m} \ell(Q_i)^m M_{\mu}f(x). 
\end{multline*}
This immediately yields 
\begin{multline}\label{eq:J2}
\mathcal{J}_2 \leq \bigg(\int_{t_0}^{t_0^{-1}} |\theta_t^{\mu} f(x_i')|^2 \frac{dt}{t} \bigg)^{\frac12} 
+ \bigg(\int_{2\kappa \ell(Q_i)}^{\infty} |\theta_t^{\mu}(f\, \mathbf{1}_{2Q_i})(x_i')|^2 \frac{dt}{t} \bigg)^{\frac12}
\\
\leq g_{\mu, t_0}f(x_i') + \ell(Q_i)^m \bigg(\int_{2\kappa \diam(Q_i)}^{\infty} t^{-2m} \frac{dt}{t}\bigg)^{\frac12} M_{\mu}f(x) 
\leq g_{\mu, t_0}f(x_i') + C M_{\mu}f(x).  
\end{multline}
In order to control $\mathcal{J}_3$, we observe that $|x-x_i'| \leq \kappa \diam(Q_i)$. Hence if $t>2\kappa \diam(Q_i)$, the smoothness condition \eqref{eq:smooth} and \eqref{aFWfveV} give
\begin{multline*}
|\theta_t^{\mu}(f\, \mathbf{1}_{(2Q_i)^c})(x) - \theta_t^{\mu}(f\, \mathbf{1}_{(2Q_i)^c})(x_i')|
\lesssim \int_{(2Q_i)^c} |s_t(x, y) - s_t(x_i', y)|\, |f(y)| d\mu(y) 
\\
\qquad\lesssim \int_{\Rn} \frac{|x-x_i'|^{\alpha}}{(t+|x-y|)^{m+\alpha}} |f(y)| d\mu(y) 
\lesssim t^{-\alpha} \ell(Q_i)^{\alpha} M_{\mu}f(x).  
\end{multline*}
As a consequence, 
\begin{align}\label{eq:J3}
\mathcal{J}_3 \lesssim \ell(Q_i)^{\alpha}
\bigg(\int_{2\kappa \diam(Q_i)}^{\infty} t^{-2\alpha} \frac{dt}{t} \bigg)^{\frac12} M_{\mu}f(x) 
\lesssim M_{\mu}f(x).  
\end{align}
Accordingly, \eqref{eq:ggM} follows from \eqref{eq:J1J2J3}, \eqref{eq:J1}, \eqref{eq:J2} and \eqref{eq:J3}. 

Now let $x \in E_{Q_i}$. Invoking \eqref{eq:ggM} and the fact that $x_i'\not\in\Omega$, we get  
\begin{multline*}
(1+\varepsilon)\,\lambda < g_{\mu, t_0}f(x) 
\leq g_{\mu, t_0}(f\, \mathbf{1}_{2Q_i})(x) +  g_{\mu, t_0}(f\, \mathbf{1}_{(2Q_i)^c})(x)
\\
\leq g_{\mu, t_0}(f\, \mathbf{1}_{2Q_i})(x) + g_{\mu, t_0}f(x_i') + C_0 M_{\mu}f(x) 
\leq g_{\mu, t_0}(f\, \mathbf{1}_{2Q_i})(x) + \lambda + C_0\, \delta \lambda. 
\end{multline*}
Choosing $\delta$ small enough so that $0<\delta \le \frac{\varepsilon}{2C_0}$, 
we obtain $g_{t_0}(f\, \mathbf{1}_{2Q_i})(x) > \varepsilon \lambda/2$,  that is, 
\begin{equation}\label{eq:EQi}
E_{Q_i} \subset \{x \in Q_i: g_{\mu, t_0}(f\, \mathbf{1}_{2Q_i})(x) > \varepsilon \lambda/2\}.
\end{equation}
Using this, that $x_i\in E_{Q_i}$, that  $2Q_i \subset Q(x_i, 3\ell(Q_i)) \subset 4 Q_i$, \eqref{eq:EQi}, and \eqref{eq:g-weak} we arrive at
\begin{multline*}
\mu(E_{Q_i}) \leq \mu(\{x \in Q_i: g_{\mu, t_0}(f\, \mathbf{1}_{2Q_i})(x) > \varepsilon \lambda/2\})
\lesssim  \frac{2\,C_1}{ \varepsilon \lambda} \int_{2Q_i} |f| d\mu  
\\
\leq \frac{2\,C_1}{ \varepsilon \lambda} \int_{Q(x_0, 3\ell(Q_i))} |f| d\mu  
\leq \frac{2C_1}{ \varepsilon \lambda} M_{\mu}f(x_0) \mu(4Q_i) 
\leq 2C_1 \delta \varepsilon^{-1} \mu(4Q_i)  \leq \gamma \mu(4Q_i),  
\end{multline*}
provided $0<\delta \leq \frac{\varepsilon \gamma}{2C_1}$. It is worth mentioning that both $C_0$ and $C_1$ are independent of $t_0$. We have then shown that if $0<\delta \leq \min\big\{\frac{\varepsilon}{2C_0}, \frac{\varepsilon \gamma}{2C_1}\big\}$, then  \eqref{eq:QQ} holds. 

Next, let us establish \eqref{eq:gLp}. Let $1<p \leq 2$ and $w \in A_{p, \B}$. By \cite[Lemma 2.3]{OP}, there are positive constants $C_w$ and $\theta_w$ such that 
${w(E)}/{w(Q)} \leq C_w ({\mu(E)}/{\mu(Q)})^{\theta_w}$ for any cube $Q$ and any subset $E \subset Q$,  
This and \eqref{eq:QQ} give $w(E_{Q_i})\leq C_w \gamma^{\theta_w} w(4Q_i)$. 
We then sum  over $i$ and use \eqref{eq:W-dec} to conclude that 
\begin{multline*}
w \big(\big\{x \in \Rn: g_{\mu,t_0}(f)(x) > (1+\varepsilon)\lambda, M_{\mu}f(x) \leq \delta \lambda  \big\}\big)
\\
=
w \big(\big\{x \in \Omega: g_{\mu,t_0}(f)(x) > (1+\varepsilon)\lambda, M_{\mu}f(x) \leq \delta \lambda  \big\}\big)
\\
\le
\sum_{i} w(E_{Q_i})
\leq 
C_w \gamma^{\theta_w}
\sum_{i} w(4\,Q_i)
\le
C_w \gamma^{\theta_w} \kappa_0 
w(\Omega_{\lambda}). 
\end{multline*}
As a consequence, 
\begin{align}\label{eq:gfMf}
\| g_{\mu,t_0}(f)\|_{L^p(w)}^p
&=(1+\varepsilon)^p p \int_{0}^\infty \lambda^{p-1} 
w(\{x: g_{\mu,t_0}(f)(x) > (1+\varepsilon)\lambda \}) d\lambda 
\nonumber \\
&\leq (1+\varepsilon)^p p C_w \gamma^{\theta_w} \kappa_0 \int_{0}^{\infty} \lambda^{p-1} 
w\big(\{x: g_{\mu,t_0}f(x) > \lambda \}\big) d\lambda
\nonumber \\
&\qquad\qquad+(1+\varepsilon)^{p} p \int_{0}^{\infty} \lambda^{p-1} 
w\big(\{x: M_{\mu}f(x) > \delta \lambda \}\big) d\lambda
\nonumber \\
&
= (1+\varepsilon)^{p} C_w \gamma^{\theta_w} \kappa_0  \| g_{\mu,t_0}(f) \|_{L^p(w)}^p 
+ (1+\varepsilon)^{p} \delta^{-p} \| M_{\mu}f \|_{L^{p}(w)}^{p}
\nonumber \\
&
\le
(1+\varepsilon)^{p} C_w \gamma^{\theta_w} \kappa_0  \| g_{\mu,t_0}(f) \|_{L^p(w)}^p 
+ C\,(1+\varepsilon)^{p} \delta^{-p} \|f \|_{L^{p}(w)}^{p}
\end{align}
where in the last inequality we have used that $M_{\mu}$ is bounded on $L^p(w)$ for any $1<p<\infty$ and $w \in A_{p, \B}$ (see \cite[Theorem 3.1]{OP}). On the other hand, 
\eqref{eq:size} and \eqref{aFWfveV} imply
\[
g_{\mu, t_0}(f)(x)
\lesssim \bigg(\int_{t_0}^{t_0^{-1}} \frac{dt}{t}\bigg)^{\frac12}\,M_\mu f(x)
=
(2\,\log(t_0^{-1}))^{1/2}\,M_\mu f(x)
\]
and
\begin{align*}
	\|g_{\mu, t_0}(f)\|_{L^{p}(w)}\le C_{t_0}\,\|M_\mu f\|_{L^{p}(w)}
	\le
	 C_{t_0}\,\|f\|_{L^{p}(w)}<\infty, 
\end{align*}
since $f$ is a bounded compactly supported function. Thus, taking $\gamma>0$ small enough so that 
$(1+\varepsilon)^p c_0 \gamma^{\theta} \kappa_0<1/2$ we obtain from \eqref{eq:gfMf}
\begin{align*}
\|g_{\mu,t_0}(f)\|_{L^p(w)} \lesssim \|f\|_{L^p(w)}, 
\end{align*}
with an implicit constant that does not depend on $t_0$. This and the monotone convergence theorem easily completes the proof of \eqref{eq:gLp}. 
\end{proof}

\subsection{Singular integral operators} 

Throughout this subsection, we let $\Sigma=\re^{n}$ and let $\mu=\mathcal{L}^{n}$ be the Lebesgue measure in $\re^{n}$. In what follows we will implicitly assume that $\mathcal{L}^{n}$ is the underlying measure and write $\mathbb{M}$ (in place of $\mathbb{M}_{\mathcal{L}^{n}}$) to denote the set of Lebesgue measurable functions in $\re^{n}$. Analogously, when we write that some condition occurs a.e.~we mean that it does $\mathcal{L}^{n}$-a.e. We let $\B$ be the collection of all euclidean balls $\re^{n}$, in which case $M_\B$ is the classical Hardy-Littlewood maximal function (with respect to uncentered balls) and we will simply write $M$. Of course, one can equivalently work with cubes in place of balls as the corresponding maximal function is pointwise equivalent to $M$. In this context $w$ is a $\B$-weight (we will simply say that $w$ is a weight) if $w\in\mathbb{M}$ with $0<w<\infty$ a.e. In this fashion, $A_{p,\B}$ or $RH_{s,\B}$  are the classical Muckenhoupt and reverse Hölder classes and will be denoted by $A_p$ and $RH_s$, respectively.

Given a symbol $a \in \mathscr{C}^{\infty}(\Rn \times \Rn)$,  we define the {\tt pseudo-differential operator} $T_a$ by
\begin{equation*}
T_{a} f(x)=\int_{\Rn} e^{i x \cdot \xi} a(x, \xi) \widehat{f}(\xi) d\xi, 
\end{equation*}
where $f \in \mathcal{S}$ and $\widehat{f}$ denotes the Fourier transform of $f$. We say that the symbol $a$ belongs to the Hörmander class $S^m_{\rho,\delta}$ introduced in \cite{H} if it satisfies 
$
|\partial_{x}^{\alpha} \partial_{\xi}^{\beta} a(x, \xi)| 
\lesssim(1+|\xi|)^{m-\rho|\beta|+\delta|\alpha|}
$,
for all multi-indices $\alpha,\beta \in \mathbb{N}^n$, where $m \in \R$ and $0 \leq \rho, \delta \leq 1$. 

Recently, Beltran \cite{B} proved that if $a \in S^m_{\rho,\delta}$ with $m \in \R$,  $0 \leq \delta \leq \rho \leq 1$ and $\delta < 1$, then 
\begin{equation}\label{eq:Ta}
\int_{\Rn} |T_a f(x)|^2 w dx 
\lesssim \int_{\Rn}|f(x)|^2 M^2 \mathcal{M}_{\rho, m} M^5 w(x) dx, 
\end{equation}
for any weight $w$, where 
\begin{equation*}
\mathcal{M}_{\rho, m} w(x) := \sup _{(y, r) \in \Lambda_{\rho}(x)} 
\frac{w(B(y,r))}{|B(y, r)|^{1+2m/n}},  
\end{equation*}
and $\Lambda_{\rho}(x) :=\left\{(y, r) \in \Rn \times(0,1) :|y-x| \leq r^{\rho} \right\}$.  

Observe that $B(y,r) \subset B(x,2r^{\rho})$ for any $(y, r) \in \Lambda_{\rho}(x)$. Thus, if we pick $m \in \R$ and $\rho \in [0, 1]$ such that $\rho=1+2m/n$, then 
\begin{equation}\label{eq:MwMw}
\mathcal{M}_{\rho, m} w 
\leq \sup _{(y, r) \in \Lambda_{\rho}(x)} \frac{|B(x,2r^{\rho})|}{|B(y, r)|^{1+2m/n}} 
\frac{w(B(x,2r^{\rho}))}{|B(x,2r^{\rho})|}
\leq C M w(x). 
\end{equation}
Moreover, $w \in A_{1}$ implies that $Mw(x) \leq C w(x)$, a.e.~$x \in \Rn$. Combining \eqref{eq:Ta} and \eqref{eq:MwMw} yields 
\begin{equation*}
\int_{\Rn} |T_a f(x)|^2 w(x) dx 
\leq C \int_{\Rn}|f(x)|^2 M^8 w(x) dx 
\leq C \int_{\Rn}|f(x)|^2 w(x) dx . 
\end{equation*}
This, Theorem \ref{thm:BFSAp}, Theorem~\ref{thm:PhiAp}, and Remarks~\ref{remark:rescale} and \ref{remark:rescale:mod} (applied with $p_0=r=2$) readily imply the following result: 

\begin{theorem}
Let $a \in S^m_{\rho,\delta}$ with $m=-n(1-\rho)/2$, $0 \leq \delta \leq \rho \leq 1$ and $\delta < 1$. Let $u$ and $v$ be weights on $\Rn$ such that $v \in L^1_{\loc}(\Rn)$. If $\X_v$ is a Banach function space on $(\Rn, v)$ such that 
\begin{align}\label{eq:TaM} 
\|(M'_{v} h)\,u^{-2}\|_{\X'_v} \le C_0 \|h u^{-2}\|_{\X'_v},\quad\forall\,h \in \M, 
\end{align} 
for some $C_0<\infty$, then 
\begin{equation}\label{eq:Taf}
\|(T_a f)\,u\|_{\X_v^2}  \leq C \|f\,u\|_{\X_v^2}.  
\end{equation}
On the other hand, if $\Phi$ is a Young function such that $\Phi\in\Delta_2$ \textup{(}equivalently, $I_\Phi<\infty$\textup{)}  and
\begin{align}\label{eq:TaM:mod}
	\int_{\Rn} \overline{\Phi}((M'_{v}h)\, u^{-2})\, v\, dx 
	\le C_0 \int_{\Rn} \overline{\Phi}(|h|\, u^{-2})\,v\, dx,\quad\forall\,h \in \M_{}, 
\end{align} 
for some $C_0<\infty$, then
\begin{equation}\label{eq:Taf:mod}
\int_{\Rn} \Phi_2(|T_a f| \, u)\, v\, dx 
\le C 
\int_{\Rn} \Phi_2(|f|\, u)\,v\, dx, 
\end{equation}
where $\Phi_2(t):=\Phi(t^2)$, $t\ge 0$.
\end{theorem}

\begin{remark}
Let $2 \le p<\infty$ and $\X_v=L^{p/2}(v)$. Then \eqref{eq:TaM} becomes 
\begin{align*}
\|(Mh)\,u^{-2} v^{-1}\|_{L^{(p/2)'}(v)} \le C \|h u^{-2} v^{-1}\|_{L^{(p/2)'}(v)},  
\end{align*}
which holds if and only if $(u^{-2}v^{-1})^{(p/2)'}v \in A_{(p/2)'}$. The later is also equivalent to $u^p v \in A_{p/2}$. On the other hand, \eqref{eq:Taf} can be rewritten as 
\begin{equation*}
\|T_a f\|_{L^p(u^p v)}  \leq C \|f\|_{L^p(u^p v)}.  
\end{equation*}
Therefore, taking $u \equiv 1$, we will recover the end-point weighted estimate for the pseudo-differential operator $T_a$ established in \cite[Theorem 1.3]{CT} and \cite[Theorem 3.10]{MRS} respectively.  
\end{remark}

Next, we will use $A_{\infty}$ extrapolation to establish the Coifman-Fefferman inequalities on Banach function spaces.

Given a Young function $\Phi$ and a ball $B$ we define the normalized Luxemburg norm as
\begin{equation*}
	\|f\|_{\Phi, B }
	:=
	\inf\Big\{\lambda>0:
	\fint_{B} \Phi\Big(\frac{|f(x)|}{\lambda}\Big)dx \leq 1\Big\}. 
\end{equation*}
It is not difficult to see if $\Phi_1$ and $\Phi_2$ are Young functions such that $\Phi_1(t)\approx\Phi_2(t)$ for all $t\ge t_0$ for some $t_0>1$ then $\|f\|_{\Phi_1, B }\approx	\|f\|_{\Phi_2, B }$. This means that in this context cases we will not be concerned about the value of the Young
functions for $t$ small. 

Denoting  by $\overline{\Phi} $ the complementary function associated to $\Phi $ one has the generalized Hölder's inequality
\begin{equation}\label{faer}
	\fint_B |f\, g|\,dx
\leq
2\, \|f\|_{\Phi,B} \|g\|_{\overline{\Phi} ,B}.
\end{equation}
Taking in particular $g\equiv 1$ one has 
\[
\fint_B |f|\,dx
\le
2\, \|f\|_{\Phi,B}.
\]

There is a further generalization that turns  out to be  useful for
our purposes, see \cite{O}: If $\Phi_1 $, $\Phi_2$, $\Phi_3$ are Young functions
such that
$
\Phi_1 ^{-1}(t) \, \Phi_2^{-1}(t)\, \Phi_3^{-1}(t)\leq t,
$
for all $t\ge 1$, then
\begin{equation}\label{Holder:1-ABC}
\fint_B |f\, g\,h|\,dx\leq C\,\|f\|_{\Phi_1 ,B}\,\|g\|_{\Phi_2,B}\,\|h\|_{\Phi_3,B}.
\end{equation}
Note that this implies
\begin{equation}\label{HolderC-AB}
	\|f\,g\|_{\overline{\Phi}_3,B}\leq C\,\|f\|_{\Phi_1 ,B}\,\|g\|_{\Phi_2,B}
	\qquad\text{and}\qquad
	\|f\|_{\overline{\Phi}_3,B}\leq C\,\|f\|_{\Phi_1, B},
\end{equation}
The first estimate is obtained by duality and for the second one, we take $g\equiv 1$.

\begin{remark}\label{remark:Linfty}\rm
Let us observe that $\Phi(t)=t$, for $t\in (0,\infty)$, is not a Young function. Nonetheless, one can extend the previous definitions with the understanding that 
$\overline{\Phi} (t)=0$ if $t\le 1$ and	$\overline{\Phi}(t)=\infty$ otherwise, whose (generalized) inverse is $\overline{\Phi }^{\,-1}(t)\equiv 1$. In such scenario, one can easily see that
\[
\|f\|_{\Phi, B }=\fint_B |f|\,dx
	\qquad\text{and}\qquad
\|f\|_{\overline{\Phi}, B }=\|f\|_{L^\infty(B)},
\]
and \eqref{faer} holds. Note also that \eqref{Holder:1-ABC} and \eqref{HolderC-AB} remain true if one of the $\Phi_j$'s is $\overline{\Phi}$.
\end{remark}

To continue, if $\Psi$ is a Young function and $k \ge 0$, we say that the kernel $K$ satisfies the {\tt $L^{\Psi,k}$-Hörmander condition}, denoted by $K \in \mathcal{H}_{\Psi,k}$, if it satisfies 
\begin{align*}
\sup _{B \subset \Rn} \sup _{x,z \in \frac12 B} \sum_{j=1}^{\infty} 
(2^j r_B)^n j^k \norm{(K(x-\cdot)-K(z-\cdot)) \mathbf{1}_{2^j B \setminus 2^{j-1} B}}_{\Psi, 2^jB}<\infty, 
\end{align*}
where the first supremum is taken over all balls in $\Rn$. When $k=0$, we simply write $\mathcal{H}_{\Psi}$. An operator $T$ is said to be a {\tt singular integral operator} if $\|T\|_{L^2 \to L^2} < \infty$ and it admits the following representation 
\begin{equation*}
T f(x)=\int_{\Rn} K(x-y) f(y) dy, \quad x \not\in \supp(f) 
\end{equation*}
for all $f \in \mathcal{S}(\Rn)$. As in Section \ref{subsec:LP-UR}, given  a measurable function $b$ and $k\ge 0$, the $k$-th order commutator of $T$ is then
\begin{equation*}
\mathbf{C}_{b}^{k}(T) f(x) 
=
\int_{\R^{n}}(b(x)-b(y))^{k} K(x,y) f(y) dy,\quad k \geq 0, 
\end{equation*} 
with the understanding that $\mathbf{C}_{b}^{0}(T) f(x) =T f(x)$.

\begin{theorem}\label{thm:ABC}
Let $T$ be a singular integral operator with the kernel $K$, and let $k\ge 0$. Let $\Phi$ and $\Psi$ be two Young functions such that $\overline{\Phi}^{-1}(t)\, \overline{\Psi}^{-1}(t)\, \overline{\mathcal{E}}_k^{-1}(t) \le t$ for $t\ge 1$, where $\overline{\mathcal{E}}_k(t)=e^{t^{1/k}}-1$ for $t\in (0,\infty)$ if $k\ge 1$ and $\overline{\mathcal{E}}_0\equiv 1$ if $k=0$.  Let $u$ and $v$ be weights on $\Rn$ such that $v \in L^1_{\loc}(\Rn)$. Assume that $\X_v$ is a Banach function space on $(\Rn, v)$ such that 
\begin{align}\label{eq:TMBv}
\|(M'_{v} h)\,u^{-1}\|_{\X'_v} \le C_0 \|h\, u^{-1}\|_{\X'_v},\quad\forall\,h \in \M. 
\end{align} 
for some $C_0<\infty$. If $b \in \BMO(\R^n,\L^n)$ and $K \in \mathcal{H}_{\Psi, k}$, then 
\begin{align}\label{eq:MAf-2} 
	\|(\mathbf{C}_{b}^{k}(T) f)\,u\|_{\X_v}  &\leq C \|b\|_{\BMO(\R^n,\L^n)}^k \|(M_{\overline{\Phi}} f)\,u\|_{\X_v}. 
\end{align}
On the other hand, assume that $\Theta$ is a Young function such that $\Theta\in\Delta_2$ \textup{(}equivalently, $I_\Phi<\infty$\textup{)}  and
\begin{align}\label{eq:TMBv:mod}
	\int_{\Rn} \overline{\Theta}((M'_{v}h)\, u^{-1})\, v\, dx 
	\le C_0 \int_{\Rn} \overline{\Theta}(|h|\, u^{-1})\,v\, dx,\quad\forall\,h \in \M, 
\end{align} 
for some $C_0<\infty$. If $b \in \BMO(\R^n,\L^n)$ and $K \in \mathcal{H}_{\Psi, k}$, then 
\begin{equation}\label{eq:MAf-2:mod} 
	\int_{\Rn} \Theta(|\mathbf{C}_{b}^{k}(T) f| \, u)\, v\, dx 
	\le C 
	\int_{\Rn} \Theta((M_{\overline{\Phi}} f)\,u)\,v\, dx. 
\end{equation}

\end{theorem}

Note that for $k=0$, then one can take $\overline{\Psi}=\Phi$ and the previous result gives that if $K \in \mathcal{H}_{\Psi}$ then 
\begin{align*}
\|(T f)\,u\|_{\X_v}  &\leq C \|(M_{\overline{\Phi}} f)\,u\|_{\X_v} 
\quad\text{ and }\quad 
\int_{\Rn} \Theta(|T f| \, u)\, v\, dx 
\le C \int_{\Rn} \Theta((M_{\overline{\Phi}} f)\,u)\,v\, dx. 
\end{align*}

Under the hypotheses of the previous result it was proved in \cite[Theorem~A]{LRT} and \cite[Theorem~3.3]{LMRT} that for every $p \in (0, \infty)$ and for every $w \in A_{\infty}$,  
\begin{align*}
\|T_{b}^k f\|_{L^p(w)} &\leq C \|b\|_{\BMO(\R^n,\L^n)}^k \|M_{\overline{\Phi}} f\|_{L^p(w)}. 
\end{align*}
As a consequence of this, Theorem~\ref{thm:BFSAi} and Theorem~\ref{thm:PhiAi} readily imply \eqref{eq:MAf-2} and \eqref{eq:MAf-2:mod} as desired.  

If $\Psi(t)=t^r$ with $1 \leq r \leq \infty$, then $K \in \mathcal{H}_{\Psi}$ coincides with the so-called $L^r$-Hörmander condition in \cite{MPT}. Under this condition, the third author et al. showed in \cite{MPT} that for every $p \in (0, \infty)$ and for every $w \in A_{\infty}$,  
\begin{equation}\label{eq:Lr}
\|Tf\|_{L^p(w)} \leq C \|M_{r'} f\|_{L^p(w)}.   
\end{equation}
It is remarkable that such result is sharp in the sense that the inequality \eqref{eq:Lr} does not hold if $M_{r'}$ is replaced by $M_s$, $1 \leq s < r'$. Evidently, Theorem \ref{thm:ABC} recovers \eqref{eq:Lr} and extends it to Banach function spaces and modular inequalities.

Let us end this subsection by studying singular integrals of Calderón-type. 

\begin{theorem}\label{thm:TABphi}
Fix $n, m, d \in \N$, let $N=N(n, m) \in \N$ be a sufficiently large integer, and let $\phi: \Rn \to \R^d$ be such that
\begin{align*}
	C^{-1} |x-y| \leq |\phi(x) - \phi(y)| \leq C|x-y|,\quad \forall\,x, y \in \Rn. 
\end{align*}
for some $C\ge 1$. Let $F \in W^{1,1}_{\loc}(\Rn)$ be a complex-valued function, $G_j \in W^{1,1}_{\loc}(\Rn)$ be a real-valued function with $\nabla F, \nabla G_j \in [{\rm BMO}(\Rn, \L^n)]^n$ for each $j=1,\ldots,m$, and write $G=(G_1,\ldots, G_m)$.  Suppose that $\eta \in \mathscr{C}^{N+2}(\R^m)$ is a complex-valued even function with the property that $\partial^{\alpha} \eta \in L^1(\R^m, \L^m)$ for every multi-index $|\alpha| \leq N+2$, and 
$\sup_{\xi \in \R^m} (1+|\xi|) |\eta(\xi)| < \infty$.  
Define the maximal singular integral 
\begin{align}\label{eq:TAB}
	T_{\phi, *}^{F, G, \eta}f(x) := \sup_{\varepsilon>0} \bigg|
	\int_{\substack{y \in \Rn \\ |\phi(x)-\phi(y)|>\varepsilon}} K_{F, G, \eta}(x, y) f(y) dy \bigg|, 
\end{align}
where 
\begin{align}\label{eq:KAB(x,y)}
	K_{F, G, \eta}(x, y) = 
	\frac{F(x)-F(y)-\langle \nabla F(y), x-y\rangle}{|x-y|^{n+1}} \eta\bigg(\frac{G(x)-G(y)}{|x-y|}\bigg).  
\end{align}
Let $u$ and $v$ be weights on $\Rn$ such that $v \in L^1_{\loc}(\Rn)$. 
\begin{list}{\textup{(\theenumi)}}{\usecounter{enumi}\leftmargin=1cm \labelwidth=1cm \itemsep=0.2cm 
		\topsep=.2cm \renewcommand{\theenumi}{\alph{enumi}}}

\item If $\X_v$ is a Banach function space over $(\Rn, v)$ such that 
\begin{align*}
\norm{(M h)\,u}_{\X_v} &\leq C_0 \|h\,u\|_{\X_v}, \quad \forall\,h \in\M, 
\\
\|(M'_v h)\,u^{-1} \|_{\X'_v} &\leq C_0 \|h\,u^{-1}\|_{\X'_v}, \quad \forall\,h \in\M, 
\end{align*}
for some $C_0<\infty$,
then 
\begin{align}\label{eq:TAB-X}
\|(T_{\phi, *}^{F,G,\eta}f)  u\|_{\X_v} \leq C  \|f\, u\|_{\X_v}. 
\end{align}
\item If $\Phi$ is a Young function such that $\Phi \in \Delta_2$ (equivalently, $I_{\Phi}<\infty$) and 
\begin{align*}
\int_{\Rn} \Phi((M h)\, u)\, v\, dx 
&\le C_0 \int_{\Rn} \Phi(|h|\, u)\, v\, dx,\quad\forall\,h \in \M, 
\\
\int_{\Rn} \overline{\Phi}((M'_v h)\,u^{-1}) v\, dx
&\le C_0 \int_{\Rn} \overline{\Phi}(|h|\, u^{-1})\,v\, dx,\quad\forall\,h \in \M,  
\end{align*} 
for some $C_0<\infty$, then
\begin{equation}\label{eq:TAB-Phi}
\int_{\Rn} \Phi((T_{\phi, *}^{F,G,\eta}f)\,u)\, v\, dx \leq C \int_{\Rn} \Phi(|f|\, u)\, v\, dx. 
\end{equation}
\end{list}
\end{theorem}

Under the hypotheses of Theorem \ref{thm:TABphi}, \cite{Mar, MMMMM20} showed that the operator $T_{\phi,*}^{F,G,H}$ in \eqref{eq:TAB} is well-defined and bounded on $L^p(\Rn, w)$ for all $1<p<\infty$ and $w \in A_p(\Rn, \L^n)$. Thus, Theorem 
\ref{thm:TABphi} is a consequence of this, Theorem \ref{thm:BFSAp}, and Theorem \ref{thm:PhiAp}. 

Let us illustrate that Theorem \ref{thm:TABphi} contains many applications. 
\begin{itemize}
\item First order commutator $[H, F D]$. Let $H$ be the Hilbert transform on the real line, and $D$ be the one-dimensional derivative operator. Assume that $F \in W^{1,1}_{\loc}(\R)$ with $F' \in \BMO(\R, \L)$ and $f \in \mathscr{C}_0^{\infty}(\R)$. For any differentiability point $x\in\re$ of $F$ (hence $\L$-a.e.)  one can see \cite{Mar, MMMMM20}  that 
\begin{align}\label{eq:HAD}
[H, M_F D] f(x) 
&= H(F f')(x)-F(x) (Hf)'(x) \nonumber 
\\
&=
\lim_{\varepsilon\to 0^+} \frac1\pi \int_{\substack{y \in \R \\ |x-y|>\varepsilon}}  \frac{F(x)-F(y)-F'(y)(x-y)}{(x-y)^2} f(y) dy.   
\end{align}
Therefore, Theorem \ref{thm:TABphi} can be applied with $n=1$, $m=1$, $G \equiv 0$, $\phi(x)=x$, and any even function 
$\eta \in \mathscr{C}_0^{\infty}(\R)$ such that $\eta(0)=1$. 

\item Calderón commutator. On the other hand, recall that the first Calderón commutator is defined by 
\begin{align*}
C_F^1 f(x)=\lim_{\varepsilon\to 0^+} \frac1\pi \int_{\substack{y \in \R \\ |x-y|>\varepsilon}} \frac{F(x)-F(y)}{(x-y)^2} f(y) dy.  
\end{align*}
Then it follows from \eqref{eq:HAD} that 
\begin{align}\label{eq:HAD-CA-HA}
C_F^1 f =[H, M_F D] f + H(F' f). 
\end{align}
Hunt, Muckenhoupt and Wheeden \cite[Theorem~9]{HMW} proved that $H$ is bounded on $L^p(\R, w)$ for all weights $w \in A_p$ and $1<p<\infty$. Hence, if $F' \in L^{\infty}(\R)$, by Theorem \ref{thm:BFSAp}, one concludes that \eqref{eq:TAB-X} and \eqref{eq:TAB-Phi} hold for the Hilbert transform $H$. Hence, from \eqref{eq:HAD-CA-HA}, $C_F^1$ enjoys the same property. This covers the unweighted inequality given in \cite[Theorem~1.1]{Mus}, which was reproved by time-frequency analysis and originated in \cite{Cal}.  

\item Cauchy integrals. Let $\kappa \in (0, \infty)$ and $\Sigma$ be a $\kappa$-chord-arc curve passing through infinity in 
$\mathbb{C}$ (cf. \cite{Mar, MMMMM20}).  The Cauchy integral operator on $\Sigma$ is defined by 
\begin{align}\label{eq:C-Sigma}
\mathcal{C}_{\Sigma} f(z) := \text{p.v. } \frac{1}{2\pi i} \int_{\zeta \in \Sigma} \frac{f(\zeta)}{\zeta-z} d\zeta,\quad z \in \Sigma. 
\end{align}
As in \cite{Mar, MMMMM20} $C_{\Sigma}$ can be rewritten as 
\begin{align}\label{eq:C-R}
\mathcal{C}_{\R} f(t) := \text{p.v. } \frac{i}{2\pi} \int_{s \in \R} \frac{z'(s)}{z(t)-z(s)} f(s)ds, \quad t \in \R.  
\end{align}
where $\re\ni s\mapsto z(s)\in\C$ is the arc-length parametrization of $\Sigma$. 
Then we have 
\begin{align}
(\mathcal{C}_{\R} -(i/2)H) f(t) = \text{p.v. } \frac{i}{2\pi} \int_{s \in \R} 
\frac{z(t)-z(s)-z'(s)(t-s)}{(z(t)-z(s))(t-s)} f(s)ds. 
\end{align}
If one chooses the appropriate functions as in \cite{Mar, MMMMM20}, then the hypotheses of Theorem \ref{thm:TABphi} are verified. 
Thus, the conclusions \eqref{eq:TAB-X} and \eqref{eq:TAB-Phi} hold for $(\mathcal{C}_{\R} -(i/2)H)$ and hence for $\mathcal{C}_{\R}$ replacing $T_{\phi,*}^{F,G, \eta}$.

\item Double layer potential for Laplace's equation:  
\begin{align}\label{eq:KKf(x)}
\mathcal{K}f(x) =\lim_{\varepsilon\to 0^+} \int_{\substack{y \in \R^n \\ |x-y|>\varepsilon}}  \frac{F(x)-F(y)-\langle \nabla F(y), x-y \rangle}
{(|x-y|^2+(F(x)-F(y))^2)^{(n+1)/2}} f(y)dy.  
\end{align}
If we take $d=n$, $m=1$, $F=G$, $\phi(x)=x$ and $\eta(t)=(1+t^2)^{-(n+1)/2}$, then $K_{F,G,\eta}$ in \eqref{eq:KAB(x,y)} agrees with the kernel in \eqref{eq:KKf(x)}. Also, the principle-value singular integrals is pointwise controlled by the corresponding maximal singular integrals. Then Theorem \ref{thm:TABphi} can be applied to get weighted norm inequalities for $\mathcal{K}$ in \eqref{eq:KKf(x)}. 

We would like to observe that by means of the operator $\mathcal{K}$, Fabes et al. \cite{FJR} studied the Dirichlet and Neumann problems for Laplace's equation on a bounded $\mathscr{C}^1$ domain $\Omega \subset \Rn$ ($n \geq 3$).  More precisely, when the datum $f$ (resp. its derivative) belongs to  $L^p(\partial \Omega)$, the solution of the Dirichlet problem (resp. Neumann problem) was formulated in the form of the classical double (resp. single) layer potential. In addition, $T_{\phi,*}^{F,G,\eta}$ and $\mathcal{K}$ were used to prove the compactness of boundary layer potentials on $L^p(\partial \Omega)$ and on the Sobolev space $L_1^p(\partial \Omega)$ \cite{FJR}; and on $L^p(\Gamma)$, where $\Gamma$ is the boundary of a bounded ${\rm VMO}_1$ domain (that is, a domain whose boundary is given in local coordinates by the graph of a function whose gradient belongs to {\rm VMO}), see in \cite[Theorem~1.17]{Hof}.  
\end{itemize}

In Theorem~\ref{thm:TABphi} we can take $v\equiv 1$ and $\X_v=L^{p(\cdot)}(\R^{n-1},\L^{n-1})$ as in Example~\ref{ex:Lpvar}, and assume that $1<p_-\le p_+<\infty$, $p(\cdot)\in{\rm LH}$, $u\in A_{p(\cdot)}$  (cf.~\eqref{eq:Lpvar-pp}--\eqref{eq:Lpvar-Apvar}). In the case of rearrangement Banach function spaces we can consider $\X=L^p$, $L^{p,q}$, or $L^{p}(\log L)^{\alpha}$ with $1<p<\infty$, $1\le q\le \infty$, and $\alpha\in\re$, in which case we have $p_\X=q_\X=p$ and weights $u,v$ satisfying $u^{p}\, v \in A_{p}$, $v \in A_{\infty}$. Analogously, we can consider the spaces $\X=(L^p+L^q)$ or $\X=(L^p\cap L^q)$, with $1<p,q<\infty$, in which case $p_\X=\min\{p,q\}$ and $q_\X=\max\{p,q\}$ and the weights satisfy $u^{p}\, v \in A_{p}$, $u^{q}\, v \in A_{q}$, and $v \in A_{\infty}$.

Regarding the modular inequalities, we can consider a Young function $\Phi$ such that $1<i_\Phi\le I_\Phi<\infty$, and assume that $u^{i_\Phi}\, v \in A_{i_\Phi}$, $u^{I_\Phi}\, v \in A_{I_\Phi}$, and $v \in A_{\infty}$.  This can be applied to the cases $\Phi(t)=t^p\,(\log(e+t))^\alpha$, or $\Phi(t)=t^p\,(\log(e+t))^\alpha\,(\log\log(e^e+t))^\beta$, with $1<p<\infty$, $\alpha, \beta\in\re$, in which case $i_{\Phi}=I_{\Phi}=p$, and the weights satisfy $u^{p}\, v \in A_{p}$  and $v \in A_{\infty}$. Also, if $\Phi(t)\approx \min\{t^p,t^q\}$ or $\Phi(t)\approx\max\{t^p,t^q\}$ with $1<p, q<\infty$ we have  
$i_{\Phi}=\min\{p,q\}$ and $I_{\Phi}=\max\{p,q\}$ and our weights satisfy  $u^{p}\, v \in A_{p}$, $u^{q}\, v \in A_{q}$, and $v \in A_{\infty}$.

\subsection{Schrödinger operators with potentials}  
Let us consider the following Schrö\-din\-ger operator with inverse-square potentials on $\Rn \, (n \geq 3)$,
\begin{equation*}
\mathscr{L}_a=-\Delta+\frac{a}{|x|^2} \quad \text { with } \quad 
a \geq -\left(\frac{n-2}{2}\right)^{2}.  
\end{equation*}
The Schrödinger operator $\mathscr{L}_a$ is understood as the Friedrichs extension of $-\Delta+\frac{a}{|x|^2}$ defined initially on $\mathscr{C}^{\infty}(\Rn \setminus \{0\})$. The condition $a \geq -\left(\frac{n-2}{2}\right)^2$ guarantees that $\mathscr{L}_a$ is nonnegative. This operator has a wide range of applications in physics and mathematics including the Dirac equation with Coulomb potential, and the study of perturbations of classic space-time metrics, see \cite{BPST, VZ, ZZ}. Recently,  harmonic analysis tools have been developed to investigate some problems related to the Schrödinger operator $\mathscr{L}_a$. The paper \cite{KMVZZ1} studied the global well-posedness and scattering for both the defocusing and focusing energy-critical NLS with inverse-square potential. Additionally, \cite{KMVZ} obtained the sharp thresholds of well-posedness and scattering for the focusing cubic NLS with inverse-square potential. 

In order to state our result we need to introduce some notation. Much as in the previous section $\Sigma=\re^{n}$, $\mu=\mathcal{L}^{n}$ is the Lebesgue measure in $\re^{n}$, $\B$ is collection of all euclidean balls $\re^{n}$, in which case $M_\B$ is the classical Hardy-Littlewood maximal function (with respect to uncentered balls) and we will simply write $M$. Of course, one can equivalently work with cubes in place of balls as the corresponding maximal function is pointwise equivalent to $M$. In this context $w$ is a $\B$-weight (we will simply say that $w$ is a weight) if $w\in\mathbb{M}$ with $0<w<\infty$ a.e. In this fashion, $A_{p,\B}$ or $RH_{s,\B}$  are the classical Muckenhoupt and reverse Hölder classes and will be denoted by $A_p$ and $RH_s$, respectively. Write
\begin{equation*}
\begin{split}
\p_{-} &:= \max\left\{1, \frac{n}{n-\sigma} \right\}; \qquad 
\p_{+} := \frac{n}{\max\{r+\sigma, 0\}},\ r>0; \qquad \widetilde{\p}_{+} := \frac{n}{\max\{r,\sigma\}};
\end{split}
\end{equation*}
and $\sigma :=(n-2-\sqrt{(n-2)^{2}+4 a})/2$. Recently,  in \cite{BDDLL, KMVZZ2} it was proved that if $n \geq 3$, $a \geq -\left(\frac{n-2}{2}\right)^2$ and $0<r<2$, then 
\begin{equation*}
\|(-\Delta)^{r/2}f\|_{L^p(\Rn, w)} \lesssim \|\mathscr{L}_a^{r/2} f\|_{L^p(\Rn, w)},
\quad \forall\,w \in A_{p/\p_{-}} \cap RH_{(\p_{+}/p)'},\  \p_{-}<p<\p_{+},
\end{equation*}
and
\begin{equation*}
\|\mathscr{L}_a^{r/2} f\|_{L^p(\Rn, w)}\lesssim
\|(-\Delta)^{r/2}f\|_{L^p(\Rn, w)},
\quad \forall\,w \in A_{p/\p_{-}} \cap RH_{(\widetilde{\p}_{+}/p)'},\ \p_{-}<p<\widetilde{\p}_{+}.
\end{equation*}
Also, in \cite{SY}, it was shown that for every $p \in (1, 2)$ and for every $w \in A_p \cap RH_{(2/p)'}$,  
\begin{align}\label{eq:LfLp}
	\|\nabla \mathscr{L}_V^{-1/2}f \|_{L^p(\Rn, w)} \leq C \|f \|_{L^p(\Rn, w)},  
\end{align}
where $\mathscr{L}_V=-\Delta+V$ and $V \in L^1_{\loc}(\Rn)$ is a non-negative function. 
Here we mention that the inequality \eqref{eq:LfLp} fails for general potentials $V \in L^1_{\loc}(\Rn)$ when $p>2$, see \cite{Shen}. By Theorem \ref{thm:lim-RIBFS},  we can extend these results to Banach function spaces. 

\begin{theorem}\label{thm:Sch} 
Let $n \geq 3$, $a \geq -\left(\frac{n-2}{2}\right)^2$ and $0<r<2$. 	Let $\X$ and $\Y$ be rearrangement invariant Banach function spaces over $(\Rn, \mathcal{L}^n)$ such that $\X^{\frac1r}$ is a Banach function space for some $r>\p_-$ and $\p_{-}<p_{\X} \le q_{\X}< \p_{+}$, and $1<p_{\Y} \le q_{\Y}< 2$
\begin{equation*}
	\big\|\big((-\Delta)^{r/2}f\big)\, u\big\|_{\X} \leq C \big\|\big(\mathscr{L}_a^{r/2} f \big)\,u\big\|_{\X},
\end{equation*}
for every weight $u$ verifying $u^{p_{\X}} \in A_{p_{\X}/\p_{-}} \cap RH_{(\p_{+}/p_{\X})'}$ and $u^{q_{\X}} \in A_{q_{\X}/\p_{-}} \cap RH_{(\p_{+}/q_{\X})'}$;
\begin{equation*}
	\big\|\big(\mathscr{L}_a^{r/2} f \big)\,u\big\|_{\X}\le C \big\|\big((-\Delta)^{r/2}f\big)\, u\big\|_{\X},   
\end{equation*}
for every weight $u$ verifying $u^{p_{\X}} \in A_{p_{\X}/\p_{-}} \cap RH_{(\widetilde{\p}_{+}/p_{\X})'}$ and $u^{q_{\X}} \in A_{q_{\X}/\p_{-}} \cap RH_{(\widetilde{\p}_{+}/q_{\X})'}$; and 
\begin{align*}
	\big\|\big(\nabla \mathscr{L}_V^{-1/2}f\big)\,u\big\|_{\Y} \leq C \|f\,u\|_{\Y},  
\end{align*}
for every weight $u$ verifying  $u^{p_{\Y}} \in A_{p_{\Y}} \cap RH_{(2/p_{\Y})'}$ and $u^{q_{\Y}} \in A_{q_{\Y}} \cap RH_{(2/q_{\Y})'}$.
\end{theorem}

In the previous result  for the first estimate we can consider $\X=L^p$, $L^{p,q}$, or $L^{p}(\log L)^{\alpha}$ with $\p_-<p<\p_+$, $\p_-\le q\le \infty$, and $\alpha\in\re$, in which case we have $p_\X=q_\X=p$ and weights $u$ satisfying $u^{p} \in A_{p/\p_{-}} \cap RH_{(\p_{+}/p)'}$. Analogously, we can consider the spaces $\X=(L^p+L^q)$ or $\X=(L^p\cap L^q)$, with $\p_-<p,q<\p_+$, in which case $p_\X=\min\{p,q\}$ and $q_\X=\max\{p,q\}$ and the weights satisfy the corresponding conditions.  For the 
second and third estimates we need to write $\widetilde{\p}_{+}$ in place of ${\p}_{+}$.

\subsection{Operators associated with the Kato conjecture}

Let $A$ be an $n \times n$ matrix of complex and $L^{\infty}$-valued coefficients defined on $\Rn$. We assume that $A$ satisfies the following ellipticity condition: there exist $0<\lambda \le \Lambda<\infty$ such that
\[
\lambda |\xi|^2 \le \Re A(x) \xi \cdot \bar{\xi} \quad\text{and}\quad 
|A(x)\xi \cdot \bar{\eta}| \le \Lambda |\xi| |\eta|, 
\]
for all $\xi, \eta \in \C^n$ and almost every $x \in \Rn$. We have used the notation $\xi \cdot \eta=\sum_{j=1}^n \xi_j \eta_j$, and therefore $\xi \cdot \bar{\eta}$ is the usual inner product in $\C^n$. Note that then $A(x)\xi \cdot \bar{\eta}=\sum_{j, k}a_{j,k}(x) \xi_k \bar{\eta}_j$. Associated with this matrix we define the second order divergence form operator $Lu =-\div(A \nabla u)$, which is understood in the standard weak sense as a maximal-accretive operator on the space $L^2(\Rn, dx)$ with domain $D(L)$  by means of a sesquilinear form. 

Associated to this operator we can consider the functional calculus $\varphi(L)$ where $\varphi$ is holomorphic and bounded in an appropriate sector, the Riesz transform $\nabla L^{-1/2}$, and some square functions. The $L^p$ theory for these operators was developed in the monograph \cite{Aus}. The weighted norm inequalities were obtained in \cite{AM3} using a generalized Calderón-Zygmund theory from \cite{AM1}.  As in \cite{Aus} and \cite{AM2}, we denote by $(p_{-}(L), p_{+}(L))$, respectively  $(q_{-}(L), q_{+}(L))$, the maximal open interval on which the Heat semigroup $\{e^{-tL}\}_{t>0}$, respectively its gradient $\{\sqrt{t}\nabla e^{-tL}\}_{t>0}$,  is uniformly bounded on $L^p(\Rn)$. It is obtained in \cite{Aus} that $p_-(L)=q_-(L)$ and $2<q_+(L)\le p_+(L)$. 

Recall the weighted norm inequalities from \cite{AM3}: 
\begin{align*}
\|\varphi(L) f\|_{L^p(w)} \lesssim \|\varphi\|_{\infty} \|f\|_{L^p(w)},\quad \forall\,p \in (p_{-}(L), p_{+}(L)), \, \, 
w \in A_{p/p_{-}(L)} \cap RH_{(p_{+}(L)/p)'};
\end{align*}
\begin{align*}
\|\nabla L^{-1/2} f\|_{L^p(w)} \lesssim \|f\|_{L^p(w)}, \quad \forall\,p \in (q_{-}(L), q_{+}(L)), \, \, 
w \in A_{p/q_{-}(L)} \cap RH_{(q_{+}(L)/p)'}; 
\end{align*}
and
\begin{align*}
	\|L^{1/2} f \|_{L^p(w)} \lesssim \|\nabla f\|_{L^p(w)},\quad \forall\,p \in (p_{-}(L), p_{+}(L)), \, \, 
	w \in A_{p/p_{-}(L)} \cap RH_{(p_{+}(L)/p)'}. 
\end{align*}

By Theorem \ref{thm:lim-RIBFS}, these estimates give the following:  

\begin{theorem}
Let $L$ be an elliptic operator as above and let $\X$ be a rearrangement invariant Banach function space over $(\Rn, \Ln)$ such that $p_-(L)<p_{\X} \le q_{\X}<\infty$ and $\X^{\frac1r}$ is a  Banach function space for some $r>p_-(L)$. 
\begin{list}{\textup{(\theenumi)}}{\usecounter{enumi}\leftmargin=1cm \labelwidth=1cm \itemsep=0.2cm 
		\topsep=.2cm \renewcommand{\theenumi}{\alph{enumi}}}

\item If $q_{\X}<p_{+}(L)$, then for every weight $u$ such that $u^{p_{\X}} \in A_{p_{\X}/p_{-}(L)} \cap RH_{(p_{+}(L)/p_{\X})'}$ and $u^{q_{\X}} \in A_{q_{\X}/p_{-}(L)} \cap RH_{(p_{+}(L)/q_{\X})'}$,   
\begin{align*}
\|(\varphi(L) f)\,u\|_{\X} \lesssim \|\varphi\|_{\infty} \|f\,u\|_{\X} 
\quad\text{and}\quad 	
\|(L^{1/2} f)\,u\|_{\X} \lesssim \|(\nabla f)\,u\|_{\X},
\end{align*}
where $\varphi$ is holomorphic and bounded in an appropriate sector.

\item If $q_{\X}<q_{+}(L)$, then for every weight $u$ such that $u^{p_{\X}} \in A_{p_{\X}/p_{-}(L)} \cap RH_{(q_{+}(L)/p_{\X})'}$ and $u^{q_{\X}} \in A_{q_{\X}/p_{-}(L)} \cap RH_{(q_{+}(L)/q_{\X})'}$,   
\begin{align*}
\|(\nabla L^{-1/2} f)\,u\|_{\X} \lesssim \|f\,u\|_{\X},  
\end{align*}
and, consequently the following Kato type estimate holds
\begin{align*}
	 \|(L^{1/2} f)\,u\|_{\X}\approx \|(\nabla f)\,u\|_{\X}.  
\end{align*}
\end{list} 
\end{theorem}

For (a) we can consider $\X=L^p$, $L^{p,q}$, or $L^{p}(\log L)^{\alpha}$ with $p_-(L)<p<p_+(L)$, $p_-(L)\le q\le \infty$, and $\alpha\in\re$, in which case we have $p_\X=q_\X=p$ and weights $u$ satisfying $u^{p} \in A_{p/p_{-}(L)} \cap RH_{(p_{+}(L)/p)'}$.  Analogously, we can consider the spaces $\X=(L^p+L^q)$ or $\X=(L^p\cap L^q)$, with $p_-(L)<p,q<p_+(L)$, in which case $p_\X=\min\{p,q\}$ and $q_\X=\max\{p,q\}$ and the weights satisfy the corresponding conditions. The same can be done for (b)   with $q_{+}(L)$ in place of $p_{+}(L)$.

Let us next study several conical square functions. Introduce the conical square functions written in terms of the Heat semigroup $\{e^{-tL}\}_{t>0}$ (hence the subscript $\mathrm{H}$): for every $m \in \N_0:=\N\cup \{0\}$,
\begin{align*}
S_{m, \mathrm{H}}f(x) :=\bigg(\iint_{\Gamma(x)}| (t^2 L)^{m+1} e^{-t^2 L}f(y)|^2 \frac{dydt}{t^{n+1}}\bigg)^{\frac12}, 
\end{align*}
\begin{align*}
G_{m, \mathrm{H}}f(x) &:=\bigg(\iint_{\Gamma(x)}| t\nabla_y (t^2 L)^m e^{-t^2 L}f(y)|^2 \frac{dydt}{t^{n+1}}\bigg)^{\frac12}, 
\\
\mathcal{G}_{m, \mathrm{H}}f(x) &:=\bigg(\iint_{\Gamma(x)}| t\nabla_{y,t} (t^2 L)^m e^{-t^2 L}f(y)|^2 \frac{dydt}{t^{n+1}}\bigg)^{\frac12}, 
\end{align*} 
where $\Gamma(x)=\{(x,t)\in \Rn \times (0,\infty):|x-y|<t\}$.

In the same manner, let us consider conical square functions associated with the Poisson semigroup $\{e^{-tL}\}_{t>0}$ (hence the subscript $\mathrm{P}$): for every $k \in \N_0:=\N\cup \{0\}$, 
\begin{align*}
S_{k, \mathrm{P}}f(x) :=\bigg(\iint_{\Gamma(x)}| (t \sqrt{L})^{2(k+1)} e^{-t \sqrt{L}}f(y)|^2 \frac{dydt}{t^{n+1}}\bigg)^{\frac12}, 
\end{align*}
and
\begin{align*}
G_{k, \mathrm{P}}f(x) &:=\bigg(\iint_{\Gamma(x)}| t\nabla_y (t \sqrt{L})^{2k} e^{-t\sqrt{L}}f(y)|^2 \frac{dydt}{t^{n+1}}\bigg)^{\frac12}, 
\\
\mathcal{G}_{k, \mathrm{P}}f(x) &:=\bigg(\iint_{\Gamma(x)}| t\nabla_{y,t} (t \sqrt{L})^{2k} e^{-t\sqrt{L}}f(y)|^2 \frac{dydt}{t^{n+1}}\bigg)^{\frac12}. 
\end{align*} 

For every $k \in \N_0$ let us set $p_{+}(L)^{k,*} :=\frac{p_{+}(L)n}{n-(2k+1) p_{+}(L)}$, if $p_{+}(L)<\frac{n}{2k+1}$, $p_{+}(L)^{k,*}:=\infty$ otherwise. With these notation in hand, we present the weighted estimates for conical square functions defined above. Indeed, the third author and Prisuelos-Arribas \cite{MP} obtained that for every $m, k \in \N_0$
\begin{align*}
S_{m, \mathrm{H}}, G_{m, \mathrm{H}}, \mathcal{G}_{m, \mathrm{H}} \text{ are bounded on } L^p(\Rn, w), w \in A_{p/p_{-}(L)}, p \in (p_{-}(L), \infty),
\end{align*}
and
\begin{align*}
S_{k, \mathrm{P}}, G_{k, \mathrm{P}}, \mathcal{G}_{k, \mathrm{P}} \text{ are bounded on } L^p(\Rn, w), w \in A_{p/p_{-}(L)} \cap RH_{(p_{+}(L)^{k,*}/p)'},
\end{align*}
for every $p \in (p_{-}(L), p_{+}(L)^{k,*})$. 

Based on these facts and Theorem \ref{thm:lim-RIBFS}, we conclude the weighted inequalities for conical square functions on rearrangement invariant Banach function spaces as follows. 
\begin{theorem}
Let $L$ be an elliptic operator as above and let $\X$ be a rearrangement invariant Banach function space over $(\Rn, \Ln)$ such that $p_-(L)<p_{\X} \le q_{\X}<\infty$ and $\X^{\frac1r}$ is a  Banach function space for some $r>p_-(L)$. 
\begin{list}{\textup{(\theenumi)}}{\usecounter{enumi}\leftmargin=1cm \labelwidth=1cm \itemsep=0.2cm 
		\topsep=.2cm \renewcommand{\theenumi}{\alph{enumi}}}
	
\item For every weight $w$ such that $w^{p_{\X}} \in A_{p_{\X}/p_{-}(L)}$ and $w^{q_{\X}} \in A_{q_{\X}/p_{-}(L)}$,  
\begin{equation*}
\|(T_1 f)w\|_{\X} \lesssim \|f\,w\|_{\X},\quad\forall\,T_1 \in\{S_{m, \mathrm{H}}, G_{m, \mathrm{H}}, \mathcal{G}_{m, \mathrm{H}}\}, \, m\in \N_0.  
\end{equation*}

\item If $q_{\X}<p_{+}(L)^{k,*}$, then for every weight $w$ such that $w^{p_{\X}} \in A_{p_{\X}/p_{-}(L)} \cap RH_{(p_{+}(L)^{k,*}/p_{\X})'}$ and $w^{q_{\X}} \in A_{q_{\X}/p_{-}(L)} \cap RH_{(p_{+}(L)^{k,*}/q_{\X})'}$, 
\begin{equation*}
\|(T_2 f)w\|_{\X} \lesssim \|f\,w\|_{\X},\quad\forall\,T_2 \in \{S_{k, \mathrm{P}}, G_{k, \mathrm{P}}, \mathcal{G}_{k, \mathrm{P}}\}, \, k \in \N_0.  
\end{equation*}
\end{list} 
\end{theorem}

For (a) we can consider $\X=L^p$, $L^{p,q}$, or $L^{p}(\log L)^{\alpha}$ with $p_-(L)<p<\infty$, $p_-(L)\le q\le \infty$, and $\alpha\in\re$, in which case we have $p_\X=q_\X=p$ and weights $u$ satisfying $u^{p} \in A_{p/p_{-}(L)}$.  Analogously, we can consider the spaces $\X=(L^p+L^q)$ or $\X=(L^p\cap L^q)$, with $p_-(L)<p,q<\infty$, in which case $p_\X=\min\{p,q\}$ and $q_\X=\max\{p,q\}$ and the weights satisfy the corresponding conditions. The same can be done for (b)  with the corresponding changes.



\end{document}